\documentclass[a4paper,leqno]{amsart}

\usepackage[english]{babel}
\usepackage[T1]{fontenc}

\usepackage{amsmath,amssymb,amsthm}
\usepackage{mathtools}  

\usepackage{graphics,tikz,caption,subcaption}

\usepackage{mathrsfs}
\usepackage{stmaryrd}

\usepackage[colorlinks	=	true, linkcolor	=	red, urlcolor	=	red, citecolor	=	red]{hyperref}
\usepackage{bookmark}									
\usepackage{cleveref}

\usepackage{multicol}
\usepackage{todonotes}
\usepackage{enumitem}
\usepackage{verbatim}
\usepackage[utf8]{inputenx}

\usetikzlibrary{decorations.pathreplacing}

\theoremstyle{definition}
\newtheorem{thm}{Theorem}[section]
\newtheorem{exm}[thm]{Example}
\newtheorem{defi}[thm]{Definition}
\newtheorem{lemm}[thm]{Lemma}
\newtheorem{cor}[thm]{Corollary}
\newtheorem{rem}[thm]{Remark}

\newtheorem{prop}[thm]{Proposition}
\newtheorem{ques}[thm]{Question}

\theoremstyle{remark}

\DeclareMathOperator{\N}{N}

\DeclareMathOperator{\Mod}{mod}

\DeclareMathOperator{\diam}{diam}

\DeclareMathOperator{\loc}{loc}
\DeclareMathOperator{\re}{Re}

\DeclareMathOperator{\interior}{int}

\allowdisplaybreaks

\newcommand{\norm}[1]{ \left\Vert #1 \right\Vert }	
\newcommand{\abs}[1]{ \left| #1 \right| }           
\newcommand{\apmd}[2][]{												
	\ifthenelse{\equal{#1}{}}%
					{ \operatorname{N}_{#2}	}%
					{ \operatorname{N}_{#1,#2} 	}}

\newcommand{\aint}[2][]{
	\ifthenelse{\equal{#1}{}}%
					{%
\mathchoice%
      {\mathop{\kern 0.2em\vrule width 0.6em height 0.69678ex depth -0.58065ex
              \kern -0.8em \intop}\nolimits_{\kern -0.45em#2}^{#1}}%
      {\mathop{\kern 0.1em\vrule width 0.5em height 0.69678ex depth -0.60387ex
              \kern -0.6em \intop}\nolimits_{#2}^{#1}}%
      {\mathop{\kern 0.1em\vrule width 0.5em height 0.69678ex depth -0.60387ex
              \kern -0.6em \intop}\nolimits_{#2}^{#1}}%
      {\mathop{\kern 0.1em\vrule width 0.5em height 0.69678ex depth -0.60387ex
              \kern -0.6em \intop}\nolimits_{#2}^{#1}}}%
					{%
\mathchoice%
      {\mathop{\kern 0.2em\vrule width 0.6em height 0.69678ex depth -0.58065ex
              \kern -0.8em \intop}\nolimits_{\kern -0.45em#1}^{#2}}%
      {\mathop{\kern 0.1em\vrule width 0.5em height 0.69678ex depth -0.60387ex
              \kern -0.6em \intop}\nolimits_{#1}^{#2}}%
      {\mathop{\kern 0.1em\vrule width 0.5em height 0.69678ex depth -0.60387ex
              \kern -0.6em \intop}\nolimits_{#1}^{#2}}%
      {\mathop{\kern 0.1em\vrule width 0.5em height 0.69678ex depth -0.60387ex
              \kern -0.6em \intop}\nolimits_{#1}^{#2}}}}					

\begin{document}
    \title[Quasiconformal geometry and removable sets]{Quasiconformal geometry and removable sets for conformal mappings}
	
	\author{Toni Ikonen}
	\address{Department of Mathematics and Statistics, University of Jyvaskyla, P.O. Box 35 (MaD), FI-40014, University of Jyvaskyla, Finland.}
	\email{toni.m.h.ikonen@jyu.fi}
	
	\author{Matthew Romney}
	\address{Department of Mathematics and Computer Science, University of Cologne, Weyertal 86-90, 50931 Köln, Germany.}
	\email{mromney@math.uni-koeln.de}

	\thanks{Both authors were supported by the Academy of Finland, project number 308659. The first author was also supported by the Vilho, Yrjö and Kalle Väisälä Foundation. The second author was also supported by the Deutsche Forschungsgemeinschaft, grant SPP 2026.} 
	
	
	\begin{abstract}
		We study metric spaces defined via a conformal weight, or more generally a measurable Finsler structure, on a domain $\Omega \subset \mathbb{R}^2$ that vanishes on a compact set $E \subset \Omega$ and satisfies mild assumptions. Our main question is to determine when such a space is quasiconformally equivalent to a planar domain. We give a characterization in terms of the notion of planar sets that are removable for conformal mappings. We also study the question of when a quasiconformal mapping can be factored as a 1-quasiconformal mapping precomposed with a bi-Lipschitz map. 
	\end{abstract}
	
    \maketitle\thispagestyle{empty}
    
    \begingroup
    \hypersetup{hidelinks}
    \tableofcontents
    \endgroup
	
\section{Introduction}\label{sec:introduction}

    \subsection{Overview}

    Let $(X,d_X)$ and $(Y,d_Y)$ be metric spaces with locally finite Hausdorff $2$-measure. A homeomorphism $f\colon X \to Y$ is \emph{$K$-quasiconformal} if there exists $K \geq 1$ such that
	\begin{equation} \label{equ:qc_definition}
	K^{-1}\Mod \Gamma \leq \Mod f\Gamma \leq K \Mod \Gamma
	\end{equation}
	for all path families $\Gamma$ in $X$, where $\Mod \Gamma$ denotes the conformal modulus of $\Gamma$. The map $f$ is \emph{quasiconformal} if it is $K$-quasiconformal for some $K \geq 1$. 
	This definition is generally referred to as the {\it geometric definition} of quasiconformal mappings, and it is one of several possible generalizations of Euclidean quasiconformal maps to the setting of metric spaces. The definition of modulus, as well as other terms used in this introduction, is reviewed in \Cref{sec:preliminaries}.
	
	The \emph{quasiconformal uniformization problem} asks one to determine which metric spaces can be mapped onto a domain in the Euclidean plane or the 2-sphere by a mapping that is quasiconformal, according to one of the several definitions. This problem is based on the classical uniformization theorem, which states that every simply connected Riemannian 2-manifold is \textit{conformally} equivalent to either the Euclidean plane, the 2-sphere, or the hyperbolic plane. Outside the 2-dimensional Riemannian setting, conformality is a very strong property, and it is natural to require only quasiconformality.  Motivation comes from connections to neighboring fields such as complex dynamics \cite{BM:17} and geometric group theory \cite{Bon:06}.
	
	In the following, let $(X,d)$ be a metric space homeomorphic to a 2-dimensional manifold and having locally finite Hausdorff 2-measure. Such a space is referred to in this paper as a \emph{metric surface}. By \emph{quasiconformal surface}, we mean a metric surface $(X,d)$ that is quasiconformally equivalent to a smooth Riemannian 2-manifold. 

	
	The uniformization problem for metric surfaces has been studied recently using various axiomatic approaches. Rajala has proved that a metric surface $X$ homeomorphic to $\mathbb{R}^2$ is a quasiconformal surface if and only if it satisfies a condition called \emph{reciprocality} (\Cref{defi:reciprocal} below) \cite{Raj:17}. Roughly speaking, this condition says that $X$ does not have too many more rectifiable paths, as quantified by conformal modulus, than Euclidean space. In this case, as shown in \cite{Rom:19}, there exists a quasiconformal map $f\colon X \to \Omega \subset \mathbb{R}^2$ that satisfies the modulus inequality
	\[\frac{2}{\pi}\Mod \Gamma \leq \Mod f\Gamma \leq \frac{4}{\pi}\Mod \Gamma \]
	for all path families $\Gamma$ in $X$. This inequality is sharp, as can be shown by considering the plane equipped with either the $\norm{ \cdot }_{1}$- or $\norm{ \cdot }_{\infty}$-norm. These results are extended to arbitrary metric surfaces in \cite{Iko:19}. A different approach was taken in a series of papers of Lytchak and Wenger \cite{Lyt:Wen:17}, \cite{LW:18}, \cite{LW:20} based on the assumption that the space satisfies a \emph{quadratic isoperimetric inequality}.  
	
	The goal of the present paper is to understand the uniformization results described above in the context of concrete constructions of metric surfaces. We study a general scheme for constructing surfaces based on specifying a measurable Finsler structure on a planar domain that vanishes on some subset of the plane. The natural problem is to decide when this construction yields a quasiconformal surface. 
	
	We provide an answer by linking the uniformization problem for metric surfaces to a separate topic in complex analysis: removable sets for classes of holomorphic functions. There are several notions of removability; see \cite{You:15} for a recent survey. For us, the relevant definition is the following. A compact set $E \subset \mathbb{R}^2$ is \emph{removable for conformal mappings} if every conformal embedding $f\colon \mathbb{R}^2 \setminus E \to \widehat{\mathbb{R}}^2$ extends to a conformal mapping $\widetilde{f}\colon \widehat{\mathbb{R}}^2 \to \widehat{\mathbb{R}}^2$, that is, to a Möbius transformation. Here, $\widehat{ \mathbb{R} }^{2}$ denotes the extended plane, which can be identified with $\mathbb{S}^{2}$ via stereographic projection. There seems to be no standard terminology for sets satisfying this condition. This is referred to as \textit{$S$-removability} in the survey \cite{You:15}, while the terms \emph{set of absolute area zero} and \emph{neglible set for extremal distance} are also used. Note that this is different than the notion of \textit{conformal removability}, which requires that every \textit{homeomorphism} of $\widehat{\mathbb{R}}^2$ that is conformal on the set $\widehat{\mathbb{R}}^2 \setminus E$ be a Möbius transformation. 
	
	This connection to removable sets is natural in hindsight but does not appear to have been made before. On the other hand, removable sets are inherently connected to a different type of uniformization problem, namely of multiply connected planar domains onto some canonical class of domain, typically slit domains or circle domains. We recall that whether an arbitrary planar domain can be mapped conformally onto a circle domain is the well-known \emph{Koebe Kreisnormierungsproblem} \cite{HS:93}. We hope the present paper will add a new perspective on these various topics. 
	
	\subsection{Motivating examples} 
	
	A basic observation, made in Example 2.1 in \cite{Raj:17}, is that not every metric surface is a quasiconformal surface. A simple example is the following. Define a length pseudometric $d_\omega$ on $\mathbb{R}^2$ via the conformal weight $\omega = \chi_{\mathbb{R}^2 \setminus \mathbb{D}}$. More precisely, we define the $\omega$-length of an absolutely continuous path $\gamma$ to be $\ell_\omega(\gamma) = \int_\gamma \omega\,ds$, and let $d_\omega(x,y) = \inf \ell_\omega(\gamma)$, the infimum taken over all absolutely continuous paths $\gamma$ connecting $x$ and $y$. If we let $X$ be the quotient space of $\mathbb{R}^2$ formed by collapsing the unit disk to a single point, then $d_\omega$ induces a metric on $X$, denoted by $\widetilde{d}_\omega$, that is locally Euclidean outside the origin. The space $(X,\widetilde{d}_\omega)$, while being homeomorphic to $\mathbb{R}^2$, is not quasiconformally equivalent to a planar domain. This is because the family of paths in $X$ that intersect the collapsed point has positive modulus, while the modulus of the family of paths intersecting a single point in the Euclidean plane is zero. This example is included as Example 11.3 in \cite{LW:18}.
	
	A second example, and the one that comprises Example 2.1  in \cite{Raj:17}, is a continuous conformal weight $\omega$ that vanishes on a Cantor set $E$ of positive area. In this case, $d_\omega$ is a metric on $\mathbb{R}^2$, and the identity map $(\mathbb{R}^2,\|\cdot\|_2) \to (\mathbb{R}^2,d_\omega)$ is a homeomorphism. Nevertheless, the vanishing of the weight increases the conformal modulus of path families in $(\mathbb{R}^2,d_\omega)$ in a way incompatible with admitting a quasiconformal parametrization by $\mathbb{R}^2$. 
	
	At the other extreme, it is not hard to show that if the analogous construction is carried out for a set $E$ with Hausdorff dimension smaller than one, then the resulting space is quasiconformally equivalent to the plane. Indeed, the set $E$ is then negligible for length and so has no effect on modulus. What happens in the intermediate situation---when the Hausdorff dimension satisfies $1 \leq \dim_{\mathcal{H}} E < 2$ or when $\mathcal{H}^2(E) = 0$---is not \textit{a priori} clear and is one of the motivations of our work. 
    
    Similar constructions appear in a number of related contexts. One of these is the notion of \emph{strong $A_\infty$-weight} introduced by David and Semmes in \cite{Dav:Sem:90}. Such a weight determines a metric on $\mathbb{R}^2$ that is Ahlfors 2-regular and quasisymmetrically equivalent to the plane. Conversely, the Jacobian of a quasisymmetric mapping from $\mathbb{R}^2$ to an Ahlfors 2-regular metric space induces a strong $A_\infty$-weight on $\mathbb{R}^2$. We do not define this term here but refer the reader to \cite[Def. 1.5]{Sem:96}. Such weights appear naturally when trying to recognize metric spaces that are bi-Lipschitz embeddable in some Euclidean space. See \cite{Dav:Sem:90,Sem:93,Sem:96,Laa:02,Bis:07} for various contributions to this topic. A separate set of papers \cite{Bo:Kos:Roh:98,Bo:Hei:Roh:01} studies metrics on the unit disk defined by conformal weights satisfying a Harnack-type inequality and an area growth condition, and shows that a number of results of classical complex analysis have natural analogues in this setting. All of the metric surfaces constructed in these two sets of papers are quasiconformally equivalent to a planar domain.
    
    %
    %
    %
    %
    
    In the above examples, when a space fails to be a quasiconformal surface, this is due to the space ``collapsing'' on the set $E$ where the weight vanishes. In fact, it may be the case that this is essentially the only way that a metric surface can fail to admit a quasiconformal parametrization. This is made precise by the following question of Rajala and Wenger. 
    
    \begin{ques}
    Let $(X,d)$ be a metric space homeomorphic to $\mathbb{R}^2$ with locally finite Hausdorff 2-measure. Is there in general a domain $\Omega \subset \mathbb{R}^2$ and a surjective continuous monotone mapping $f\colon \Omega \to X$ such that $f$ is in the metric Sobolev space $N^{1,2}_{\loc}(\Omega, X)$ and satisfies the one-sided dilatation condition
        \[ g_f^2(x) \leq K J_f(x) \]
    for some constant $K \geq 1$ and almost every $x \in \Omega$?
    \end{ques}
    Here, $g_f$ is the minimal weak upper gradient of $f$ and $J_f$ is the Jacobian of $f$; see \Cref{sec:sobolev}. 
    We say that $f \colon \Omega \rightarrow X$ is \emph{monotone} if the preimage of every point $x \in X$ is a connected and compact subset of $\Omega$. 
    
    
    \subsection{Setting and main results} \label{sec:main_results}
    We continue with a description of our setting and main results. Let $\Omega$ be a planar domain and $E \subset \Omega$ be a compact set that does not separate $\Omega$. We consider a measurable seminorm field $N\colon \Omega \times \mathbb{R}^{2} \to [0, \infty)$ that vanishes exactly on the set $E$ and satisfies certain mild assumptions, namely lower semicontinuity, local boundedness, and having locally bounded distortion. The seminorm at the point $x \in \Omega$ is denoted throughout this paper by $N_x$. We think of $N$ as a Finsler structure on $\mathbb{R}^2$, determining a Finsler metric on $\mathbb{R}^2$, although requiring no regularity beyond the previous assumptions.
    
    For conciseness, and since $N_x$ is a norm for all $x \in \Omega \setminus E$, we use the term \emph{norm field} and not \emph{seminorm field} throughout this paper when referring to $N$.  A norm field $N$ satisfying the above hypotheses is said to be \emph{admissible} (Definition \ref{defi:seminorm_admissible} below). We define the \emph{$N$-length} of an absolutely continuous path $\gamma\colon I \to \Omega$ by
    \begin{equation} \label{equ:length}
      \ell_N(\gamma) = \int_I N \circ D\gamma(t) \,dt.  
    \end{equation}
    In interpreting \eqref{equ:length}, note that the base point of $N$ is understood to be $\gamma(t)$ even though this is omitted from the notation. 
    One then obtains a pseudometric $d_N$ on $\Omega$ by setting $d_N(x,y) = \inf \ell_N(\gamma)$, the infimum taken over all absolutely continuous paths $\gamma$ from $x$ to $y$ contained in $\Omega$. Let $\mathcal{E}_N$ denote the collection of equivalence classes of points in $\mathbb{R}^2$, declaring $x$ to be equivalent to $y$ if $d_N(x,y) = 0$. Then $d_N$ determines a metric on the quotient space $\mathbb{R}^2/\mathcal{E}_N$ denoted by $\widetilde{d}_N$. In Section \ref{sec:distances}, we describe this construction in more detail.

    
    We make the following definition.
    \begin{defi}
    The admissible norm field $N$ is \emph{reciprocal} if the corresponding space $(\Omega/\mathcal{E}_N, \widetilde{d}_N)$ is reciprocal (\Cref{defi:reciprocal}). 
    \end{defi}
    The natural problem is to characterize as best as possible those norm fields $N$ that are reciprocal. 
     Our first result is the following.
    

    \begin{thm} \label{thm:removable_implies_reciprocal}
    Let $\Omega \subset \mathbb{R}^2$ be a domain and $E \subset \Omega$ a compact set. If $E$ is removable for conformal mappings, then every admissible norm field $N \colon \Omega \times \mathbb{R}^{2} \rightarrow \left[0,  \infty\right)$ that vanishes exactly on $E$ is reciprocal.
    \end{thm}
    Recall that our definition of admissibility includes the statement that $N$ is locally bounded. It turns out that this assumption can be relaxed. In \Cref{prop:reciprocality:p-int}, we show that \Cref{thm:removable_implies_reciprocal} still holds provided there exists some $p > 2$ such that the maximal stretching $L( N )$ is in $L^{p}_{ \loc }( \Omega )$. This generalization follows fairly readily from \Cref{thm:removable_implies_reciprocal} by an approximation argument.

    Next, we consider whether some converse to \Cref{thm:removable_implies_reciprocal} holds. Observe first that the strongest possible converse to \Cref{thm:removable_implies_reciprocal} is false: a reciprocal norm field $N$ may vanish on a set $E$ that is not removable for conformal mappings. As a simple example, take $E \subset \mathbb{R}^2$ to be a snowflake arc and let $N = \chi_{\mathbb{R}^2 \setminus E}\|\cdot\|_2$. Since $\mathcal{H}_{\|\cdot\|_2}^1(|\gamma| \cap E) = 0$ for every absolutely continuous path $\gamma$, we see that $d_{N}$ actually coincides with the Euclidean metric. However, it is a basic fact that any set that is removable for conformal mappings is totally disconnected.
    
    
    On the other hand, if one requires that the norm field $N$ decays fast enough near $E$ and $N$ is reciprocal, then examples of the type just described are not possible. To illustrate this, consider two admissible norm fields $N_1$ and $N_2$ that satisfy $N_1 \leq N_2$. Every path that has finite $N_2$-length also has finite $N_1$-length, while the opposite may fail to be true for a large family of paths. In this sense, the space generated by the smaller norm field $N_1$ has more rectifiable paths and the reciprocality condition is harder to satisfy. This leads to the following partial converse to \Cref{thm:reciprocal_implies_removable}.
    \begin{thm} \label{thm:reciprocal_implies_removable}
    Let $\Omega \subset \mathbb{R}^2$ be a domain and $E \subset \Omega$ a compact set for which $\Omega \setminus E$ is connected, and let 
    $N_p(x) = \min\left\{ 1,  d_{\|\cdot\|_2}(x,E)^p \right\} \|\cdot\|_2$. If $N_p$ is reciprocal for some $p > \max\left\{ \dim_\mathcal{H} E - 1, 0 \right\}$, then the set $E$ is removable for conformal mappings.
    \end{thm}
    
    Our method of proof actually yields a stronger conclusion. The relevant property of the norm field $N_p$, verified in \Cref{lemm:good_snowflake} below, is that the quotient map $\pi_N$ maps $E$ onto a set of zero 1-dimensional Hausdorff measure with respect to the metric $\widetilde{d}_N$. Thus, for any reciprocal norm field $N$ with this property, the corresponding set $E$ on which $N$ vanishes is removable for conformal mappings. For example, one can show that, if $E$ is contained in a continuum $F$ satisfying $\mathcal{H}_{\|\cdot\|_2}^1(F) < \infty$, this condition on the quotient map $\pi_N$ is satisfied for all admissible norm fields $N$ vanishing on $E$. For such compact sets, the strongest converse to \Cref{thm:removable_implies_reciprocal} holds. That is, if any admissible norm field $N$ vanishing exactly on $E$ is reciprocal, then $E$ is removable for conformal mappings.
    
    The lower bound for $p$ in \Cref{thm:reciprocal_implies_removable} is sharp. Consider an arc $E \subset \mathbb{R}^{2}$ that is bi-Lipschitz equivalent to $( \left[0, 1\right], \abs{ \cdot }^{1/d} )$ for some $d \in (1,2)$. Then $E$ is a snowflake arc of Hausdorff dimension $d$. It follows from \cite[Theorem 6.3]{Sem:96} that the square of the weight $\omega_{d-1}(x) = \min\left\{ 1,  d_{\|\cdot\|_2}(x,E)^{d-1} \right\}$ is a strong $A_{\infty}$-weight, as defined in \cite[Definition 1.5]{Sem:96}, and hence the norm field $N_{d-1}$ is reciprocal. However, the arc $E$ is not removable for conformal mappings.
    
    Theorems \ref{thm:removable_implies_reciprocal} and \ref{thm:reciprocal_implies_removable} show that reciprocal norm fields are almost characterized by whether the set on which they vanish is removable for conformal mappings. We now mention a few facts about removable sets for conformal mappings that are known, many of them coming from a classic paper of Ahlfors--Beurling \cite{AB:50}. First, every compact set of positive Hausdorff 2-measure is not removable. Second, every compact set of zero Hausdorff 1-measure is removable. More intriguingly, for Cantor sets $E \subset \mathbb{R} \times \left\{ 0 \right\}$ of positive Hausdorff 1-measure, both outcomes are possible. In \cite{AB:50}, Ahlfors and Beurling construct Cantor sets in $\mathbb{R} \times \left\{0\right\}$ of positive $\mathcal{H}^{1}$-measure that are removable for conformal maps, as well as such Cantor sets that are not removable. A similar example in the related context of circle domain uniformization can be found as Theorem 11.1 of an early version of a paper of Schramm \cite{Sch:95}. Next, by Theorem 10 in \cite{AB:50} and Proposition 3.3 in \cite{Kal:Kov:Raj:19}, removable sets for conformal mappings are \emph{metrically removable}: for every $\varepsilon>0$, each pair of points $x,y \in \mathbb{R}^2$ can be connected by a curve disjoint from $E \setminus \left\{ x, y \right\}$ that has length at most $\|x-y\|_2 + \varepsilon$. See \cite{Hak:Her:08} and \cite{Kal:Kov:Raj:19} for more on the topic of metric removability. Removable sets for conformal mappings are also examples of the \emph{quasiextremal distance exceptional sets} considered in \cite{GM:85} and the related literature. Finally, an equivalent definition can be given by replacing the word ``conformal'' with ``quasiconformal'' in the definition \cite[Prop. 4.7]{You:15}. Thus the property of removability is invariant under quasiconformal mappings of the complementary domain. 
    
    This should be compared with the notion of \emph{removable sets for bounded analytic functions}. The problem of characterizing such sets is known as \emph{Painlev\'e's problem} and has received considerable attention, with a satisfactory resolution obtained by Tolsa in \cite{Tol:03}. We note here that this is a stronger notion of removability: every set that is removable for bounded analytic functions is removable for conformal mappings. See Proposition 4.3 of \cite{You:15} for a proof. For example, a removable set for bounded analytic functions must have Hausdorff dimension at most $1$. Moreover, according to David's resolution of Vitushkin's conjecture \cite{Dav:98}, a compact set $E$ with finite Hausdorff $1$-measure is removable for bounded analytic functions if and only if it is purely $1$-unrectifiable.
    
    Finally, we remark that the notion of uniformly disconnected sets provides a further class of examples to which these results apply. In \cite{Sem:96}, Semmes studies metrics of the form $d_{N_{p}}$, where $N_{p}$ is as in \Cref{thm:reciprocal_implies_removable}, with the additional assumption that the set $E$ is \emph{uniformly disconnected}, meaning that there exists $\varepsilon>0$ with the property that, for any two distinct points $x,y \in E$, there is no sequence of points $x=x_0, x_1, \ldots, x_m = y$ in $E$ satisfying $\|x_{j-1} - x_j\|_2 < \varepsilon \|x - y\|_2$ for all $j \in \{1, \ldots, m\}$. He proves that for such an $E$ and every $p > 0$, the square of the weight $\omega_p(x) = \min\{1, d_{\norm{\cdot}_2}(x,E)^{p}\}$ is a strong $A_\infty$-weight  
    and hence the norm field $N_{p}$ in \Cref{thm:reciprocal_implies_removable} is reciprocal. Therefore \Cref{thm:reciprocal_implies_removable} implies that uniformly disconnected Cantor sets are removable for conformal mappings. This removability can alternatively be deduced in many ways from the existing literature. Note in particular that a uniformly disconnected set $E$ can have Hausdorff dimension arbitrarily close to $2$. 
    
    \subsection{Factorization of quasiconformal mappings}
    
    This section is motivated by the following factorization problem. Consider a quasiconformal surface $(X,d)$ and corresponding isothermal parametrization $f \colon \Omega \to X$, where $\Omega$ is a smooth Riemannian surface. Following \cite{Iko:19}, a quasiconformal homeomorphism $f\colon \Omega \to X$ is \emph{isothermal} if it is distortion-minimizing at almost every point in a suitable sense. Roughly speaking, the \emph{pointwise distortion} of $f$ at $x$ is the aspect ratio of the image of a small ball centered at $x$. The existence of an isothermal parametrization for every quasiconformal surface is established in \cite[Corollary 6.3]{Iko:19}. See \Cref{sec:quasiconformal_mappings} for the precise definition of distortion and \Cref{sec:isothermal} for the definition of isothermal map. We ask: can one find a metric surface $(\widehat{X},\widehat{d})$ such that $f$ factors as $f = \widehat{f} \circ P$, where $\widehat{f} \colon \widehat{X} \to X$ is 1-quasiconformal and $P \colon \Omega \to \widehat{X}$ is bi-Lipschitz? In other words, can one find a ``conformal representative'' for the space $X$ within the class of bi-Lipschitz surfaces?
    
    If the metric is defined by a continuous reciprocal norm field of bounded distortion, then such a factorization can always be found. Recall that for every domain $\Omega \subset \mathbb{R}^{2}$ there exists a smooth Riemannian norm field $G = \omega\norm{ \cdot }_{2}$ on $\Omega$ such that $( \Omega, d_{G} )$ is complete and has Gaussian curvature $0$ or $-1$. We have the following result. 

    \begin{prop}\label{prop:pi:isothermal}
    Let $\Omega \subset \mathbb{R}^2$ be a domain and $N$ a reciprocal norm field with distortion $H$. If $N$ is continuous outside the set $E = \{ x\in \Omega: N_x = 0\}$, then there exists a distance $\widehat{d}$ on $\Omega$ such that:
    \begin{enumerate}[label=(\roman*)]
        \item The identity map $P\colon (\Omega, d_G) \to (\Omega, \widehat{d})$ satisfies
        \begin{equation} \label{equ:P_bi_Lipschitz}
          d_G(x,y) \leq \widehat{d}(P(x), P(y)) \leq H d_G(x,y)  
        \end{equation}
        for all $x,y \in \Omega$.
        \item The identity map $\widehat{\iota} \colon (\Omega, \widehat{d}) \to (\Omega, d_N)$ is 1-quasiconformal.
    \end{enumerate}
    \end{prop}
    If the identity map $\iota\colon \Omega \to (\Omega,d_N)$ is isothermal, then it has distortion at most $\sqrt{2}$ \cite[Cor. 4.7]{Iko:19}, and so \eqref{equ:P_bi_Lipschitz} holds with $H = \sqrt{2}$. The example of the $\ell^\infty$-norm on $\mathbb{R}^2$ shows that the value $H = \sqrt{2}$ in \eqref{equ:P_bi_Lipschitz} is sharp for the case of general isothermal maps. Since every quasiconformal surface has an isothermal parametrization, this raises the question of finding conditions on $N$ that guarantee that the conclusion of \Cref{prop:pi:isothermal} holds with $H = \sqrt{2}$. In turn, this question is related to the regularity of the Beltrami coefficient derived from distance ellipse field corresponding to $N$ and does not appear to have a straightforward answer. We briefly address this issue in \Cref{sec:positive_answer:regular}.

    In general, the conclusion of \Cref{prop:pi:isothermal} may fail if $N_x$ is discontinuous outside of $E$. In the final part of the paper, we present a lengthy construction giving a negative answer to the above factorization question in general. In fact, we obtain the stronger conclusion that no quasiconformal map $\widehat{f}$ in such a factorization can have distortion smaller than that of $f$.
    \begin{thm} \label{thm:example}
    There is a metric $d$ on $\mathbb{R}^2$ such that the identity map $\iota\colon (\mathbb{R}^2, \|\cdot\|_2) \to (\mathbb{R}^2, d)$ is an isothermal quasiconformal homeomorphism, but $\iota$ does not factor as $\iota = \widehat{\iota} \circ P$, where $(\widehat{X}, \widehat{d})$ is a metric surface, $\widehat{\iota} \colon (\widehat{X}, \widehat{d}) \to (\mathbb{R}^2, d)$ is quasiconformal with distortion $H(\widehat{\iota}) < \sqrt{2}$ and $P \colon (\mathbb{R}^2, \|\cdot\|_2) \to (\widehat{X}, \widehat{d})$ is bi-Lipschitz. 
    \end{thm}
    The identity map $\iota$ in our construction has distortion $H(\iota) = \sqrt{2}$, so the inequality $H(\widehat{\iota}) < \sqrt{2}$ is sharp.
    
    The metric $d$ in \Cref{thm:example} is defined via a lower semicontinuous norm field of the form
    \[
        N_x
        =
        \begin{cases}
            c_x \norm{\cdot}_1 & \text{if } x \in F
            \\
            c_x \norm{\cdot}_\infty & \text{if } x \notin F
        \end{cases}
    \]
    for some measurable set $F \subset \mathbb{R}^2$ and measurable function $x \mapsto c_x$, where $0\leq c_x \leq 1$ and $c_x$ vanishes at a single point. Note that this fits exactly into the construction scheme of this paper, and therefore $(\mathbb{R}^2, d)$ is a quasiconformal surface. 
    
    One might initially expect that the metric $\widehat{d}$ on $\mathbb{R}^2$ defined by 
    \[
        \widehat{N}_x
        =
        \begin{cases}
            \norm{\cdot}_1 & \text{if } x \in F
            \\
            \sqrt{2}\norm{\cdot}_\infty & \text{if } x \notin F
    \end{cases}
    \]
    with $\widehat{\iota}$ and $P$ the identity map on $\mathbb{R}^2$, or some variation on this, gives a factorization satisfying the properties given in \Cref{thm:example}. Observe that $\norm{ \cdot }_{2} \leq \widehat{N} \leq \sqrt{2} \norm{ \cdot }_{2}$ everywhere, so the map $P$ in this situation is bi-Lipschitz. However, the map $\widehat{\iota}$ may fail to be 1-quasiconformal. The reason for this is that the norm field $\widehat{N}$ corresponding to $F$ is typically not lower semicontinuous, in which case the metric tangents of $P$ need not coincide with $\widehat{N}_x$ almost everywhere. Indeed, we prove \Cref{thm:example} by specifying explicitly a set $F$ and coefficients $c_x$ for which this failure of 1-quasiconformality occurs for the norm field $\widehat{N}$ defined above, and in fact for any conformal rescaling of $\widehat{N}$ bi-Lipschitz equivalent to the Euclidean norm field.

    
    We now describe our construction in somewhat more detail. The basic idea is to construct a sequence of nested Cantor sets $K_i$ as the intersection of a collection of squares in the plane. This is done so that the odd-indexed Cantor sets are formed from squares in the standard (i.e., non-rotated) alignment, while the even-indexed Cantor sets are formed from squares aligned diagonally. Next, the norm field on $K_i\setminus K_{i+1}$ for odd values of $i$ is defined to be the supremum norm $\norm{ \cdot }_{ \infty }$, scaled by a constant $c_i$ satisfying $c_i \to 0$ as $i \to \infty$, while the norm field for even values of $i$ is defined to be the $\norm{ \cdot }_{1}$-norm, also scaled by a constant $c_i'$ satisfying $c_i' \to 0$ as $i \to \infty$. A consequence of the distortion inequality for $\widehat{\iota}$ is that the metric tangents of $P$ and $\iota$ cannot differ by more than a fixed amount, up to rescaling. With a suitable choice of constants $c_i, c_i'$, the alternating arrangement of the Cantor sets $K_i$ then forces the metric tangents of $P$ to be arbitrarily small at some points.

    Lytchak--Wenger \cite{LW:18} and Creutz--Soultanis \cite{CS:19} study similar types of factorizations for \emph{minimal disks} or \emph{solutions to Plateau's problem} with metric space target, though without trying to optimize the properties of $P$ in the way that we have proposed. Here, we simply remark that the map $\iota$ in our example is also an energy-minimizing map (for the Reshetnyak energy) in the sense of these papers on each closed disk. We refer the reader to the above papers for definitions of these terms.

    \subsection{Outline} Our paper is organized as follows. \Cref{sec:preliminaries} gives an overview of basic results and notation related to metric Sobolev spaces, quasiconformal mappings, and removable sets. In \Cref{sec:distances}, we give a detailed overview of the construction of metric spaces from a prescribed norm field under suitable assumptions. In \Cref{sec:reciprocality}, we prove the first of the main results, \Cref{thm:removable_implies_reciprocal}, stating that an admissible norm field is reciprocal if it vanishes exactly on a set that is removable for conformal mappings. In \Cref{sec:reciprocal_implies_removable}, we prove the partial converse, \Cref{thm:reciprocal_implies_removable}. \Cref{sec:linear_Cantor_sets} gives a pair of examples of spaces constructed from conformal weights that each vanish on a linear Cantor set of positive length, one of which is reciprocal and one of which is not. Thus this can be viewed as the borderline case. Finally, \Cref{sec:example} gives the proof of \Cref{prop:pi:isothermal} as well as the construction for \Cref{thm:example}.
    
    \subsection*{Acknowledgments} 
    
    We are thankful to Alexander Lytchak, Kai Rajala, and Atte Lohvansuu for discussions about this project and feedback on a draft of this paper. We also thank Dimitrios Ntalampekos and Malik Younsi for discussions related to removable sets for conformal mappings, and in particular for Malik Younsi bringing the paper \cite{AB:50} to our attention. Finally, we thank Jarmo Jääskeläinen for discussions related to the Beltrami equation.
    
\section{Preliminaries}\label{sec:preliminaries}

\subsection{Notation} \label{sec:notation} 

    In this paper, we frequently consider several metrics in close proximity to one another. For this reason, we will consistently use subscripts to denote the metric being referred to. Let $(X,d)$ be a metric space. The open ball centered at a point $x \in X$ of radius $r>0$ with respect to the metric $d$ is denoted by $B_d(x,r)$. 

    The Euclidean metric is denoted by $\|\cdot\|_2$. Thus, for example, we write $B_{\|\cdot\|_2}(x,r)$ for an open ball with respect to this metric, and $ds_{\|\cdot\|_2}$ for the Euclidean length element.
    
    We recall the definition of Hausdorff measure. Let $(X,d)$ be a metric space. For all $p \geq 0$, the \emph{$p$-dimensional Hausdorff measure}, or \emph{Hausdorff $p$-measure}, is defined by 
	\[
	    \mathcal{H}_X^p(A)
	    =
	    \sup_{ \delta > 0 }
	    \inf
	        \left\{
	            \frac{\alpha(p)}{2^p}
	            \sum_{i=1}^\infty (\diam A_i)^p
	            :
	            A \subset \bigcup_{i=1}^\infty A_i, 
	            \diam A_i < \delta
	        \right\} 
	\]
	for all sets $A \subset X$, where $\alpha(p) = \pi^{ \frac{p}{2} } \left( \Gamma\left( \frac{p}{2} + 1 \right) \right)^{-1}$. The constant $\alpha(p)$ is chosen in such a way that $\mathcal{H}^{n}_{\mathbb{R}^{n}}$ coincides with the Lebesgue measure $\mathcal{L}^{n}$ for all positive integers.
	
	If the space $X$ is understood but not the metric $d$, then we use the notation $\mathcal{H}_d^p$ instead of $\mathcal{H}_X^p$. The Hausdorff dimension of a set $E \subset X$ is the infimal value of $p$ for which $\mathcal{H}_X^p(E)=0$ and is denoted by $\dim_{\mathcal{H}_d} E$. For the basics of Hausdorff measure, see for example \cite[Chapter 2]{Amb:Til:04}.
	
	Unless otherwise noted, in this paper a metric surface $(X,d)$ is always equipped with the Hausdorff 2-measure generated by the metric $d$. For example, the phrase \emph{almost every} refers to the Hausdorff 2-measure. Similarly, an interval in $\mathbb{R}$ is equipped with the Lebesgue measure $\mathcal{L}^1$.
    
    A \emph{path} is a continuous function from an interval into a metric space. A path in $X$ will typically be denoted by $\gamma$. The image of $\gamma$ is denoted by $|\gamma|$. The \emph{length} of the path $\gamma\colon [a,b] \to X$ is defined as 
    \[ \ell_d(\gamma) = \sup \sum_{j=1}^n d(\gamma(t_{i-1}), \gamma(t_i)),\]
    the supremum taken over all finite sequences $a = t_0 \leq t_1 \leq \cdots \leq t_n = b$. A path is \emph{rectifiable} if it has finite length. 
    
    The \emph{metric speed} of a path $\gamma \colon \left[a,  b\right] \rightarrow X$ at the point $t \in \left[a,  b\right]$ is defined as
    \begin{equation*}
    \label{eq:metric:speed:path}
        v_{\gamma}(t)
        =
        \lim_{ h \rightarrow 0 }
            \frac{ d( \gamma( t + h ),  \gamma( t ) ) }{ t }
    \end{equation*}
    whenever this limit exists. If $\gamma$ is rectifiable, its metric speed exists at $\mathcal{L}^{1}$-almost every $t \in \left[a,  b\right]$; see Theorem 2.1 of \cite{Dud:07}.

    A rectifiable path $\gamma\colon [a,b] \to X$ is \emph{absolutely continuous} if for all $a \leq s \leq t \leq b$,
    \begin{equation*}
    \label{eq:AC:characterization}
        d( \gamma(t),  \gamma(s) )
        \leq
        \int_{ s }^{ t }
            v_{ \gamma }( u )
        \,d\mathcal{L}^{1}( u )
    \end{equation*}
    with $v_{ \gamma } \in L^{1}( \left[a,  b\right] )$ and $\mathcal{L}^{1}$ the Lebesgue measure on the real line. Equivalently, $\gamma$ is absolutely continuous if it maps sets of $\mathcal{L}^{1}$-measure zero to sets of $\mathcal{H}_X^{1}$-measure zero in its image; see Section 3 of \cite{Dud:07}.
    
    A path $\widetilde{\gamma} \colon \left[c,  d\right] \rightarrow X$ is a \emph{reparametrization} of $\gamma$ if there exists a map $\psi \colon \left[a,  b\right] \rightarrow \left[c,  d\right]$ that is surjective, non-decreasing, and continuous such that $\gamma = \widetilde{\gamma} \circ \psi$. If $\psi$ is absolutely continuous, we say that $\widetilde{\gamma}$ is an \emph{absolutely continuous reparametrization} of $\gamma$. Note that this is different from $\widetilde{\gamma}$ itself being an absolutely continuous path.
    
    Every rectifiable path $\gamma$ has a reparametrization $\widetilde{\gamma} \colon \left[ 0,  \ell_d( \gamma ) \right] \rightarrow X$ such that the metric speed of $\widetilde{\gamma}$ equals one $\mathcal{L}^{1}$-almost everywhere. In this case, we write $\gamma_{s} = \widetilde{\gamma}$, and refer to $\gamma_{s}$ as the \emph{unit speed parametrization} of $\gamma$. See Chapter 5 of \cite{HKST:15} for details.
    
    If $\gamma$ is rectifiable, the unit speed parametrization $\gamma_{s}$ is $1$-Lipschitz and hence absolutely continuous \cite[Proposition 5.1.8]{HKST:15}.
    
    Let $\gamma$ be a rectifiable path. Then the \emph{path integral} of a Borel function $\rho \colon X \rightarrow \left[0,  \infty\right]$ \emph{over $\gamma$} is
    \begin{equation}
    \label{eq:path:integral}
        \int_{ \gamma }
            \rho
        \,ds
        =
        \int_{ 0 }^{ \ell_d( \gamma ) }
            \rho \circ \gamma_{s}
        \,d\mathcal{L}^{1},
    \end{equation}
    where $\mathcal{L}^{1}$ is the Lebesgue measure on the real line. 
    
    If $\gamma$ is absolutely continuous and $\widetilde{\gamma}$ is an absolutely continuous reparametrization of $\gamma$, the chain rule for metric speeds \cite[Theorem 3.16 and Remark 3.4]{Dud:07} states that
    \begin{equation*}
        v_{ \gamma }
        =
        (v_{ \widetilde{\gamma} } \circ \psi)
        \psi'
        \in
        L^{1}( \left[c,  d\right] ),
    \end{equation*}
    where the right-hand side is understood to be zero whenever the derivative $\psi' = 0$ (even if $v_{ \widetilde{\gamma} } \circ \psi$ is not defined or is infinite at such a point).
    
    Moreover, for absolutely continuous $\gamma$, the unit speed parametrization $\gamma_{s}$ is an absolutely continuous reparametrization of $\gamma$. Therefore \eqref{eq:path:integral} can be restated for absolutely continuous $\gamma \colon \left[a,  b\right] \rightarrow X$ as follows:
    \begin{equation*}
    \label{eq:path:integral:AC}
        \int_{ \gamma }
            \rho
        \,ds
        =
        \int_{ a }^{ b }
            (\rho \circ \gamma)
            v_{ \gamma }
        \,d\mathcal{L}^{1}.
    \end{equation*}
    Given a Borel set $A \subset X$, the length of a path $\gamma \colon \left[a, b\right] \rightarrow X$ in $A$ is defined as $\int_{ X } \#( A \cap \gamma^{-1}(x) ) \,d\mathcal{H}^{1}_{X}(x)$, where $\#( A \cap \gamma^{-1}(x) )$ is the multiplicity of $\gamma$ in $A$. This formula makes sense for paths that are not necessarily rectifiable; see Theorem 2.10.13 \cite{Fed:69}. If $\gamma$ is rectifiable, the number coincides with the path integral of $\chi_{A}$ over $\gamma$.

\subsection{Metric Sobolev spaces}\label{sec:sobolev}

    In this section, we give an overview of the theory of Sobolev spaces in the metric space setting. We refer the reader to the book \cite{HKST:15} for a comprehensive introduction to this topic. Throughout this section, assume that $(X,d_X)$ and $(Y, d_Y)$ are each a metric surface: a metric space homeomorphic to a $2$-dimensional manifold with locally finite Hausdorff 2-measure.

    The conformal modulus provides a basic way of measuring the size of a family of paths. It is a conformal invariant in the Euclidean case, which accounts for both its nomenclature and its usefulness. Let $\Gamma$ be a family of paths in $X$. A Borel function $\rho\colon X \to [0, \infty]$ is \emph{admissible} for the path family $\Gamma$ if the path integral $\int_\gamma \rho\,ds \geq 1$ for all locally rectifiable paths $\gamma \in \Gamma$. The \emph{conformal modulus}, or simply \emph{modulus}, of $\Gamma$ is
        \[\Mod \Gamma = \inf \int_X \rho^2\,d\mathcal{H}_X^2,\]
    where the infimum is taken over all admissible functions $\rho$. 

    If $\rho$ is admissible for a path family $\Gamma' \subset \Gamma$ such that $\Gamma \setminus \Gamma'$ has modulus zero, then $\rho$ is said to be \emph{weakly admissible} for $\Gamma$. If a property holds for every path $\gamma \in \Gamma$ except in a subfamily of modulus zero, then this property is said to hold \emph{on almost every path}. If $\Mod \Gamma < \infty$, then there exists a weakly admissible Borel function $\rho \in L^{2}( X )$ such that
    \begin{equation*}
        \Mod \Gamma
        =
        \int_{ X }
            \rho^{2}
        \,d\mathcal{H}^{2}_{X}.
    \end{equation*}
    Such a $\rho$ is called a \emph{minimizer} of $\Gamma$. Such a minimizer is unique $\mathcal{H}^{2}_{X}$-almost everywhere.

    Let $f\colon (X,d_X) \to (Y,d_Y)$ be a mapping between metric surfaces $X$ and $Y$. A function $g\colon X \to [0, \infty]$ is an \emph{upper gradient of $f$} if
        \[d_Y(f(x),f(y)) \leq \int_\gamma g\,ds\]
    for every rectifiable path $\gamma\colon [0,1] \to X$ connecting $x$ to $y$. The function $g$ is a \emph{weak upper gradient of $f$} if the same holds for almost every rectifiable path.
    
    The weak upper gradient $g \in L^{2}_{\loc}( X )$ is \emph{minimal} if it satisfies $g \leq \widetilde{g}$ almost everywhere for all weak upper gradients $\widetilde{g} \in L^{2}_{\loc}( X )$ of $f$. If $f$ has a weak upper gradient $g \in L^{2}_{\loc}( X )$, then $f$ has a minimal weak upper gradient, which we denote by $g_f$. The existence of $g_{f}$ follows from the fact that the weak upper gradients of $f$ form a lattice. This also implies that $g_{f}$ is unique up to measure zero; see Section 6 of \cite{HKST:15} and Section 3 of \cite{Wil:12} for details. In general, $g_{f}$ is only a weak upper gradient.

    Proposition 6.3.3 of \cite{HKST:15} and countable subadditivity of modulus (see also Lemmas 3.2 and 3.3 of \cite{Wil:12}) establish that a Borel function $\rho \colon X \rightarrow \left[0,  \infty\right]$ belonging to $L^{2}_{\loc}(X)$ is a weak upper gradient of $f$ if and only if for almost every absolutely continuous path $\gamma \colon \left[a,  b\right] \rightarrow X$, the composition $f \circ \gamma$ is an absolutely continuous path for which the metric speeds $v_{ f \circ \gamma }$ and $v_{ \gamma }$ satisfy
    \begin{equation}
        \label{eq:metric:speed:characterization}
        v_{ f \circ \gamma }
        \leq
        (\rho \circ \gamma)
        v_{ \gamma }
    \end{equation}
    $\mathcal{L}^{1}$-almost everywhere on $\left[a,  b\right]$. Since $\rho \in L^{2}_{\loc}( X )$ the right-hand side of \eqref{eq:metric:speed:characterization} is integrable on its domain for almost every $\gamma$.

    Let $Z$ be a metric space such that $\mathcal{H}_{d_Z}^2(Z) < \infty$. Choose a point $y \in Y$, and let $d_y = d_Y(\cdot,y)$. The space $L^{2}( Z,  Y )$ is defined as the set of measurable mappings $f\colon Z \to Y$ such that $d_y \circ f$ is in $L^2(Z)$. One can check that this definition is independent of the choice of $y$.  

    We define $L^{2}_{\loc}( X,  Y )$ to consist of those measurable mappings $f\colon X \to Y$ for which, for all $x \in X$, there is an open set $U \subset X$ containing $x$ such that $f|_U$ is in $L^2(U,Y)$. 
    
    The metric Sobolev space $N^{1,  2}_{\loc}( X,  Y )$ consists of those mappings $f \colon X \rightarrow Y$ in $L^{2}_{\loc}( X,  Y )$ that have a minimal weak upper gradient $g_{f} \in L^{2}_{\loc}( X )$.
    
    For open $U \subset X$ with $\mathcal{H}^{2}_{X}( U ) < \infty$, we say that $f \in N^{1,  2}( U,  Y )$ if $f|_{U} \in N^{1,  2}_{\loc}( U,   Y )$ in such a way that $g_{f}|_{U} \in L^{2}( U )$ and for some $y \in Y$, $f_{y}( x ) = d_{y} \circ f|_{U} \in L^{2}( U )$.

    Next we define the Jacobian of $f$ for continuous $f \colon X \rightarrow Y$. The \emph{pullback measure} $f^{*}\mathcal{H}^{2}_{Y}$ is defined for Borel sets $A \subset X$ by the formula
    \begin{equation*}
    \label{eq:pullback_measure}
        f^{*}\mathcal{H}^{2}_{Y}( A )
        =
        \int_{Y}
            \#( A \cap f^{-1}(y) )
        \,d\mathcal{H}^{2}_{Y},
    \end{equation*}
    where $\#( A \cap f^{-1}(y) )$ is the multiplicity function of $f$ relative to $A$. The measure $f^{*}\mathcal{H}^{2}_Y$ can be defined equivalently using a suitable Carathéodory construction; see \cite[2.10.10]{Fed:69}. In fact, $f^{*}\mathcal{H}^{2}_{Y}$ is a Borel regular outer measure.
    
    If the pullback measure $f^{*}\mathcal{H}^{2}_{Y}$ is locally finite, we say that the \emph{Jacobian} of $f$ is the Radon--Nikodym derivative of $f^{*}\mathcal{H}^{2}_{Y}$ with respect to $\mathcal{H}^{2}_{X}$. The Jacobian is denoted by $J_{f}$. The local finiteness of $f^{*}\mathcal{H}^{2}_{Y}$ and $\mathcal{H}^{2}_{X}$ imply that $J_{f}$ is locally integrable. See Sections 3.1-3.2 in Volume I of \cite{Bog:07} for details on the Radon--Nikodym derivative of a measure.

\subsection{Seminorms}\label{sec:seminorms}
    We introduce the terminology and notation we use for seminorms. Recall that a \emph{seminorm} $S$ on $\mathbb{R}^{2}$ is a function $S \colon \mathbb{R}^2 \to [0, \infty)$ satisfying the following conditions for all $v,w \in \mathbb{R}^2$ and $\lambda \in \mathbb{R}$:
    \begin{itemize}
        \item[(i)] (absolute homogeneity) $S( \lambda v ) = |\lambda| S( v )$ whenever $\lambda \in \mathbb{R}$ and $v \in \mathbb{R}^{2}$;
        \item[(ii)] (triangle inequality) $S( v + w ) \leq S( v ) + S( w )$.
    \end{itemize}
    The seminorm $S$ is a \emph{norm} if it has the additional property that $S( v ) = 0$ only if $v = 0$.
    The \emph{maximal stretching} of $S$ is
    \begin{equation}
        \label{eq:upper_gradient}
        L( S )
        =
        \sup\left\{
            S(v)
            \colon
            \norm{ v }_{2} \leq 1
        \right\}.
    \end{equation}
    The \emph{minimal stretching} of $S$ is
    \begin{equation}
        \label{eq:minimal_gradient}
        \omega(S)
        =
        \inf\left\{
            S(v)
            \colon
            \norm{ v }_{2} \geq 1
        \right\}.
    \end{equation}
    The \emph{Jacobian} of the seminorm $S$ is
    \begin{equation*}
	\label{eq:area_jacobian}
	    J_{2}(S)
	    =
	    \frac{ \pi }{ \mathcal{L}^{2}\left( \left\{ v : S(v) \leq 1 \right\} \right) }.
    \end{equation*}
    Observe that $J_2(S) = 0$ in the case that $N$ is only a seminorm.
    The \emph{distortion} of $S$ is 
    \begin{equation}
    \label{eq:distortion}
        H(S) = \frac{L(S)}{\omega(S)}
    \end{equation}
    if $\omega(S)>0$ and $H(S) = \infty$ otherwise. The latter case occurs if $S$ is a non-zero seminorm that is not a norm. 
    The \emph{outer dilatation} and \emph{inner dilatation} of $S$ are defined by, respectively,
    \begin{align*}
    \label{eq:outer_dilatation}
        K_O(S) = \frac{L(S)^2}{J_2(S)}, \quad K_I(S) = \frac{J_2(S)}{\omega(S)^2}
    \end{align*}
    if $J_{2}(S)>0$, and $K_O(S) = K_I(S) = \infty$ otherwise. The \emph{maximal dilatation} of $S$ is $K(S) = \max\{K_O(S), K_I(S)\}$. Observe that $K_O(S) \geq 1$ and $K_I(S) \geq 1$.
    
    The seminorm $S$ induces a pseudometric $d_S$ on $\mathbb{R}^2$ by the formula $d_S(x,y) = S(x-y)$. The identity map
    $   \iota_{S}
        \colon
        ( \mathbb{R}^{2},  \norm{ \cdot }_{2} )
        \to
        ( \mathbb{R}^{2},  d_S )$
    has the constant function $L( S )$ as its minimal weak upper gradient and $J_2(S)$ as its Jacobian. Its inverse $\iota_S^{-1}$ has the constant function $\omega(S)^{-1} $ as its minimal weak upper gradient.

    The following lemma gives a relationship between the maximal dilatation and distortion.
    \begin{lemm}\label{lemm:distortion_vs_dilatation}
    The distortion $H(S)$ and maximal dilatation $K(S)$ of $S$ satisfy $H( S ) \leq K( S ) \leq H( S )^{2}$.
    \end{lemm}
    \begin{proof} 
    If $\omega(S) = 0$, then $H(S) = K(S) = \infty$. Otherwise, $H(S)$ and $K(S)$ are both finite. Observe the relationship $H(S)^{2} = K_O(S)K_I(S) \leq K(S)^{2}$. On the other hand, the relationships $K_O(S) \geq 1$ and $K_I(S) \geq 1$ imply respectively that $H(S)^{2} \geq K_I(S)$ and $H(S)^{2} \geq K_O(S)$. We conclude that $H(S)^2 \geq K(S)$.
    \end{proof}

\subsection{Metric derivatives of Lipschitz mappings}\label{sec:lipschitz}
    Throughout this section, we let $\Omega$ denote a domain in $\mathbb{R}^2$ and $(X,d)$ denote a metric space.
    We refer to \Cref{sec:seminorms} for basic terminology about seminorms.
	\begin{defi}\label{defi:metric_derivative_Ivanov}
	Let $f \colon (\Omega,\|\cdot\|_2) \rightarrow (X,d)$ be a Lipschitz map. For all $x \in \Omega$ and $v \in \mathbb{R}^{2}$, the \emph{metric derivative of $f$ at $x$ in the direction $v$} is
	\begin{equation}
	\label{eq:metric_differential}
	    \apmd[f]{x}(v)
	    =
	    \limsup_{ t \rightarrow 0^{+} }
	        \frac{ d( f(x), f(x + tv) ) }{ t }.
	\end{equation}
	\end{defi}
    A result by Ivanov \cite{Iva:08} states the following. Similar results are proved in \cite{Kir:94,DP:90,DP:91,DP:95}.
	\begin{thm}\label{thm:legit_differential}
	Let $f \colon (\Omega,\|\cdot\|_2) \rightarrow (X,d)$ be a Lipschitz map. There exists a Borel set $N_{0} \subset \Omega$ of zero Lebesgue measure such that, for all $x \in \Omega \setminus N_{0}$ and all $v \in \mathbb{R}^{2}$, the limit superior in \eqref{eq:metric_differential} is an actual limit, and $v \mapsto \apmd[f]{x}(v)$ is a seminorm for every $x \in \mathbb{R}^{2} \setminus N_{0}$.
	\end{thm}
	As a consequence of \Cref{thm:legit_differential}, the metric derivative of a Lipschitz map defines a seminorm field on $\Omega$.
	
	\begin{prop}\label{prop:lipschitz}
	Let $f \colon \Omega \rightarrow X$ be a Lipschitz function and $\apmd{ f }$ its metric derivative. The maximal stretching $x \mapsto L( \apmd{f}(x) )$ is a minimal weak upper gradient of $f$, and $f$ satisfies the change of variables formula 
	\begin{equation}
	\label{eq:change_of_variables}
        \int_{ \Omega }
            \rho(z)
            J_{2}( \apmd[f]{z} )\,
        d\mathcal{L}^{2}(z)
        =
        \int_{ X }
        \int_{ f^{-1}(x) }
            \rho(y)
        d\mathcal{H}^{0}(y)\,
        d\mathcal{H}^{2}_{d}( x )
    \end{equation}
    for all Borel functions $\rho \colon \Omega \rightarrow \left[0,  \infty\right]$.
	\end{prop}
	\begin{proof}
    \Cref{thm:legit_differential} implies that the metric derivative, as defined in \Cref{defi:metric_derivative_Ivanov}, coincides with the metric derivative of Kirchheim \cite{Kir:94} $\mathcal{L}^{2}$-almost everywhere in $\Omega$. Kirchheim proves the change of variables formula \eqref{eq:change_of_variables} as Corollary 8 in \cite{Kir:94}.
    That $L( \apmd{f} )$ is a minimal weak upper gradient of $f$ is proved in Section 4 of \cite{Lyt:Wen:17}.
    \end{proof}
    The metric differential can be used to compute the metric speed of an absolutely continuous path.
    \begin{lemm}\label{lemm:metric_speed_Lips}
    If $\gamma \colon \left[a,  b\right] \rightarrow \Omega$ is an absolutely continuous path, then for almost every $t \in [a,b]$, the metric speed $v_{ f \circ \gamma }(t)$ of $f \circ \gamma$ exists and satisfies
    \begin{equation*}
        v_{ f \circ \gamma }(t)
        =
        \apmd{f} \circ D\gamma(t),
    \end{equation*}
    where $D\gamma(t)$ is the derivative of $\gamma$ at $t$.
    \end{lemm}
    \begin{proof}
	Ivanov proves in Proposition 2.7 of \cite{Iva:08} that $\ell_{d}( f \circ \gamma ) = \ell_{ \apmd{f} }( \gamma )$ for every Lipschitz path $\gamma \colon \left[a,  b\right] \rightarrow \mathbb{R}^{2}$. Since every absolutely continuous path has a Lipschitz parametrization, the same result holds for absolutely continuous paths $\gamma \colon \left[a,  b\right] \rightarrow \mathbb{R}^{2}$. The lemma now follows from the Lebesgue differentiation theorem.
    \end{proof}

\subsection{Quasiconformal mappings} \label{sec:quasiconformal_mappings}

    Recall the geometric definition of quasiconformal mapping given in \eqref{equ:qc_definition}. A result of Williams is that this geometric definition is equivalent to an analytic definition based on metric Sobolev spaces. We state the two-dimensional case of this result, or rather a generalization to the case of continuous monotone maps.
    \begin{thm}[cf. \cite{Wil:12}] \label{prop:williams:L-Wversion}
    Let $X$ and $Y$ be metric surfaces with locally finite Hausdorff $2$-measure. Let $f \colon X \rightarrow Y$ be continuous and monotone and suppose that the pullback measure $f^{*}\mathcal{H}^{2}_{Y}$ is locally finite.
    The following are equivalent for the same constant $K \geq 1$:
    \begin{enumerate}
        \item[(i)] $\Mod \Gamma \leq K \Mod f\Gamma$ for all path families $\Gamma$ in $X$.
        \item[(ii)] $f \in N_{\loc}^{1,2}(X,Y)$ and satisfies
            \[ g_f^2(x) \leq KJ_f(x) \]
        for $\mathcal{H}_X^2$-almost every $x \in X$.
    \end{enumerate}
    \end{thm}
    \Cref{prop:williams:L-Wversion} can be established using the original proof in \cite{Wil:12} with slight modifications which deal with the multiplicity of $f$. This is omitted here. A similar result can be found as Proposition 3.5 of \cite{LW:20}.
    
    The \emph{outer dilatation} of $f$ is the smallest constant $K \geq 1$ for which the modulus inequality $\Mod \Gamma \leq K \Mod f \Gamma$ holds for all $\Gamma$ in $X$. The \emph{inner dilatation} of $f$ is the smallest constant $K \geq 1$ for which $\Mod f \Gamma \leq K \Mod \Gamma$ holds for all $\Gamma$ in $X$. These are denoted respectively by $K_O(f)$ and $K_I(f)$. Thus a quasiconformal map is a homeomorphism with finite outer and inner dilatation. The \emph{pointwise distortion} of $f$ at $x \in X$ is
    \begin{equation}
    \label{eq:QC:pw:distortion}
        H_{f}( x )
        =
        g_{f}( x ) g_{ f^{-1} }( f(x) ).
    \end{equation}
    We interpret \eqref{eq:QC:pw:distortion} as $H_{f}( x ) = 1$ whenever $g_{f}( x ) = 0$ or $g_{ f^{-1} }( f(x) ) = 0$. A consequence of Proposition 3.4 and Corollary 3.8 in \cite{Iko:19} is that $H_{f}( x )$ is independent of the representatives of $g_{f}$ and $g_{ f^{-1} }$ $\mathcal{H}^{2}_{X}$-almost everywhere.
    Moreover, Corollary 3.12 of \cite{Iko:19} implies that $H_{f}( x ) \leq \sqrt{ K_{O}(f) K_{I}(f) }$ for $\mathcal{H}^{2}_{X}$-almost every $x \in X$. The smallest constant $H \geq 1$ for which $H_{f}( x ) \leq H$ for $\mathcal{H}^{2}_{X}$-almost every $x \in X$ is called the \emph{distortion} of $f$ and denoted by $H(f)$.
    
    Consider now a quasiconformal map $f\colon \Omega \subset \mathbb{R}^2 \to X$ that is also Lipschitz. Then the equalities $g_f(x) = L(\N_{f,x})$ and $g_{ f^{-1} } \circ f(x) = ( \omega( \N_{f,x} ) )^{-1}$ hold for $\mathcal{L}^2$-almost every $x \in \Omega$ \cite[Proposition 4.8]{Iko:19}. Consequently, we have the equality $H_f(x) = H(\N_{f,x})$ for $\mathcal{L}^2$-almost every $x \in \Omega$. 

    In general, a quasiconformal map $f\colon \Omega \subset \mathbb{R}^2 \to X$ must satisfy \emph{Lusin's Condition ($N^{-1}$)}: for every Borel set $E \subset \Omega$ of positive Lebesgue measure, $f(E)$ has positive Hausdorff 2-measure. This is essentially proved in Remark 8.3 or Section 17 of \cite{Raj:17}. On the other hand, $f$ need not satisfy \emph{Lusin's Condition ($N$)}: for every Borel set $E \subset \Omega$ of zero Lebesgue measure, $f(E)$ has zero Hausdorff 2-measure. An example of this is given as Proposition 17.1 of \cite{Raj:17}. Generalizations of these results are considered in Section 3 of \cite{Iko:19}.

    A uniformization theorem for quasiconformal mappings was proved by Rajala based on the notion of \emph{reciprocality} \cite{Raj:17}. Let $X$ be a metric surface. For a set $G \subset X$ and disjoint sets $F_1, F_2 \subset G$, let $\Gamma(F_1,F_2; G)$ denote the family of paths whose images are contained in $G$ that start from $F_1$ and end in $F_2$. A \emph{quadrilateral} is a set $Q$ homeomorphic to $[0,1]^2$ with boundary consisting of four nonoverlapping boundary arcs, labelled $\xi_1, \xi_2, \xi_3, \xi_4$ in cyclic order.
    
    \begin{defi}\label{defi:reciprocal}
    A metric surface $X$ is \emph{reciprocal} if there exists a constant $\kappa  \geq 1 $ such that
    \begin{align}
    	\label{upper:bound}
    	\Mod \Gamma\left( \xi_{1},  \xi_{3};  Q \right)
    	\Mod \Gamma\left( \xi_{2},  \xi_{4};  Q \right)
    	\leq
    	\kappa
    \end{align}
    for every quadrilateral $Q \subset X$, and 
    \begin{equation}
    	\label{point:zero:modulus}
    	\lim_{ r \rightarrow 0^{+} }
    	\Mod \Gamma\left( \overline{B}( x,  r ),  X \setminus B( x,  R );  \overline{B}( x,  R ) \right)
    	=
    	0
    \end{equation}
    for all $x \in X$ and $R > 0$ such that $X \setminus B( x, R ) \neq \emptyset$.
    \end{defi}
    Note that, for all metric surfaces, the product in \eqref{upper:bound} is bounded from below by a universal constant $\widetilde{\kappa} >0$ \cite{RR:19}. We say that a reciprocal surface is \emph{$\kappa$-reciprocal} if \eqref{upper:bound} and the corresponding lower bound hold for the constant $\kappa$.
    
    Theorem 1.4 in \cite{Raj:17} states that a metric surface $X$ homeomorphic to $\mathbb{R}^2$ is reciprocal if and only if there exists a quasiconformal homeomorphism onto the disk or the Euclidean plane. This result is extended to arbitrary metric surfaces in \cite{Iko:19}. More precisely, Theorem 1.3 in \cite{Iko:19} states that a metric surface $X$ is locally reciprocal (that is, every point in $X$ has a neighborhood that is reciprocal) if and only if $X$ is quasiconformally equivalent to a smooth Riemannian 2-manifold. In particular, a metric surface that is locally reciprocal is also globally reciprocal.

\subsection{Removable sets for conformal mappings}
    
    We collect some background on removable sets for conformal mappings. 
    Recall from the introduction that the compact set $E \subset \mathbb{R}^{2}$ is \emph{removable for conformal mappings} if every conformal embedding $f \colon \mathbb{R}^2 \setminus E \rightarrow \widehat{\mathbb{R}}^2$ extends to a conformal mapping $F \colon \widehat{\mathbb{R}}^2 \rightarrow \widehat{\mathbb{R}}^2$. Thus $f$ is the restriction of a Möbius transformation.
    
    This notion exists under several names, including \emph{sets of absolute area zero} and \emph{negligible sets for extremal distance}. This nomenclature reflects the following characterization.
    
    \begin{prop} \label{prop:removable_set_characterizations}
    Let $E \subset \mathbb{R}^2$ be compact. The following are equivalent.
    \begin{enumerate}[label=(\roman*)]
        \item $E$ is removable for conformal mappings. \label{item:removable_i}
        \item $E$ has absolute area zero: for every conformal embedding $f\colon \mathbb{R}^{2} \setminus E \to \widehat{\mathbb{R}}^2$, the complementary set $\widehat{\mathbb{R}}^2 \setminus f(\mathbb{R}^{2} \setminus E)$ has Lebesgue measure zero. \label{item:removable_ii}
        \item $E$ is negligible for modulus: for every domain $\Omega \subset \mathbb{R}^{2}$ and pair of disjoint compact sets $F, G \subset \Omega \setminus E$, $\Mod \Gamma( F, G; \Omega ) = \Mod \Gamma( F, G; \Omega \setminus E )$. \label{item:removable_iii}
        \item Any quasiconformal embedding $f \colon \mathbb{R}^{2} \setminus E \rightarrow \widehat{\mathbb{R}}^2$ has an extension to a quasiconformal mapping $F\colon \widehat{\mathbb{R}}^2 \to \widehat{\mathbb{R}}^2$. \label{item:removable_iv}
        \item For any open set $U \subset \mathbb{R}^{2}$, every quasiconformal mapping on $U \setminus E$ extends quasiconformally to the whole open set $U$. \label{item:removable_v}
    \end{enumerate}
    \end{prop}
    The equivalence of \ref{item:removable_i}, \ref{item:removable_ii} and \ref{item:removable_iii} is proved in \cite{AB:50}. The equivalence of \ref{item:removable_i} and \ref{item:removable_iv} is a consequence of the measurable Riemann mapping theorem. See Proposition 4.7 in \cite{You:15} for a proof. The equivalence of \ref{item:removable_i} and \ref{item:removable_v} can also be found in \cite{You:15} as Proposition 4.6. We see from \ref{item:removable_iv} and \ref{item:removable_v} that removability for conformal mappings is a local property and a quasiconformal invariant. If $E$ contains a nontrivial connected component $E_0$, then there is a non-Möbius conformal map $f\colon \mathbb{R}^2 \setminus E_0 \to \mathbb{R}^2$ such that $\mathbb{R}^2 \setminus f(\mathbb{R}^2 \setminus E_0)$ is the closed unit disk. Thus Property \ref{item:removable_ii} implies that a removable set for conformal mappings is totally disconnected.
    
    Property \ref{item:removable_iii} in \Cref{prop:removable_set_characterizations} indicates the connection between quasiconformal uniformization and removable sets. Observe that for each triple $F$, $G$, and $\Omega$, $\Gamma( F, G; \Omega \setminus E )$ is a subset of $\Gamma( F, G; \Omega )$ and thus satisfies $\Mod \Gamma( F, G; \Omega \setminus E ) \leq \Mod \Gamma( F, G; \Omega )$. In contrast, the metric space constructions in our paper collapse a domain at the set $E$ and hence \emph{increase} the modulus of a path family, up to a factor related to the dilatation bound of the norm field. Thus \Cref{thm:removable_implies_reciprocal} and \Cref{thm:reciprocal_implies_removable} can be summarized roughly by saying that \emph{removing the set $E$ does not decrease the modulus of any path family if and only if collapsing the plane at $E$ does not increase the modulus of any path family}.

\section{Constructing a metric from a norm field}\label{sec:distances}

    In this section, we give a description of the metric spaces considered in this paper and develop their basic properties. These spaces are constructed from measurable Finsler structures satisfying additional assumptions. The precise definition is given in \Cref{sec:admissible_norm_field}. 
    
    There is a vast literature on Riemannian and Finsler geometry, typically requiring smoothness or at least continuity of the metrics. The idea of constructing metrics from Finsler structures with less regularity has been considered by various previous authors, and so the material in this section is more-or-less standard. In \Cref{sec:remarks:admissible}, we include a brief comparison with the existing literature. 

    We consider here seminorm fields $N$ such that either $N_x$ is a norm or $N_x = 0$ for all $x \in \Omega$. Recall from the introduction that, slightly abusing terminology, we use the term \emph{norm field} to refer to an object of this type. Since a vector $v \in \mathbb{R}^2$ often comes with an implicit basepoint $x$, we will sometimes write $N(v)$ in place of $N_x(v)$, such as in the expression $N \circ D\gamma$. 
    \subsection{Definition of the metric} \label{sec:admissible_norm_field}
    
    Let $\Omega \subset \mathbb{R}^{2}$ be a domain.
    \begin{defi}\label{defi:seminorm_admissible}
    A norm field $N \colon \Omega \times \mathbb{R}^{2} \rightarrow \left[0,  \infty\right)$ is \emph{admissible} if it satisfies the following:
    \begin{itemize}
        \item[(i)] (lower semicontinuous) For all vectors $v \in \mathbb{R}^2$ and points $x \in \Omega$, we have $N_x(v) \leq \liminf_{y \to x} N_y(v)$.
        \item[(ii)] (locally bounded) For all $x \in \Omega$, there is a neighborhood $U$ of $x$  and $M >0$ such that $L(N_y) \leq M$ for all $y \in U$.  
        \item[(iii)] (locally bounded distortion) For all $x \in \Omega$, there is a neighborhood $U$ of $x$ and $H >0$ such that $L(N_y) \leq H \omega( N_y) )$ for all $y \in U$.
        \item[(iv)] (nonseparating) The set $E = \{x \in \Omega: N_x = 0\}$ is compact and $\Omega \setminus E$ is connected. 
    \end{itemize}
    \end{defi}
    An immediate consequence of having locally bounded distortion is that $N_{x}(v) = 0$ for some $v \in \mathbb{R}^{2} \setminus \left\{ 0 \right\}$ if and only if $N_{x}$ is identically zero. 
    
    We use the norm field $N$ to measure the length of an absolutely continuous path $\gamma \colon \left[a,  b\right] \rightarrow \Omega$ in the following way. We define the \emph{$N$-length} of $\gamma$ to be
    \begin{equation*}
    \label{eq:length_norm_field}
        \ell_{N}( \gamma )
        =
        \int_a^b
            N\circ D\gamma(t)\, 
        dt,
    \end{equation*}
    where $D\gamma \colon [a,b] \rightarrow \Omega \times \mathbb{R}^{2}$ is a Borel representative of the differential of $\gamma$.
    \begin{defi}\label{eq:seminorm_admissible_distance}
    Let $N$ be an admissible seminorm field and $x,  y \in \Omega$. The \emph{$N$-distance} between $x$ and $y$ is defined as
    \begin{equation*}
        d_{N}(x,y)
        =
        \inf \ell_{N}( \gamma ),
    \end{equation*}
    where the infimum is taken over absolutely continuous paths $\gamma$ joining $x$ to $y$ in $\Omega$.
    \end{defi}
    
    The function $d_{N}$ is locally finite and satisfies the triangle inequality, but it may happen that $d_N(x,y) = 0$ for distinct points $x,y \in \Omega$. Thus, in general, $d_N$ is only a pseudodistance. Let $\mathcal{E}_N$ be the partition of $\Omega$ into equivalence classes of points, where $x,y \in \Omega$ belong to the same equivalence class if $d_N(x,y) = 0$. This yields the quotient space $\Omega / \mathcal{E}_N$ and the natural quotient map $\pi_N \colon \Omega \to \Omega / \mathcal{E}_N$. The space $\Omega / \mathcal{E}_N$ comes equipped with the metric that is the pushforward of $d_N$ under $\pi_N$, which we denote by $\widetilde{d}_N$.
    
    A consequence of the local boundedness of $N$ is that the quotient map $\pi_N$ is locally Lipschitz. In particular, the results described in \Cref{sec:lipschitz} apply to the map $\pi_N$.
    
    \subsection{Remarks on definition of admissible norm fields}\label{sec:remarks:admissible}
    
    We offer a few remarks about  \Cref{defi:seminorm_admissible} and give a comparison to the previous literature. 
    
    The lower semicontinuity assumption guarantees that the metric tangents of $\Omega / \mathcal{E}_{N}$ coincide with $N$ almost everywhere. This implies, for example, that two conformally equivalent norm fields generate metrics that are 1-quasiconformally equivalent. In general, the metric tangents are not so well-behaved. For example, let $F \subset [0,1]$ be a Cantor set of positive linear measure, and let $E = F \times F \subset \mathbb{R}^2$. The norm field $N$ defined by
    \[
        N_x
        = 
        \begin{cases}
            2\|\cdot\|_\infty & \text{ if } x \in E
            \\
            \|\cdot\|_1 & \text{ if } x \notin E
        \end{cases} 
    \]
    generates the same metric as the norm field $\|\cdot\|_1$, despite the fact that they differ on a positive measure set. Indeed, the inequality $\|x-y\|_1 \leq d_N(x,y)$ is immediate for all $x,y \in \mathbb{R}^2$, since $\|\cdot\|_1 \leq 2\|\cdot\|_\infty$. On the other hand, for all $x,y \in \mathbb{R}^2 \setminus E$, there is an $\ell_1$-geodesic from $x$ to $y$ lying in $\mathbb{R}^2 \setminus E$. Thus $d_N(x,y) \leq \|x-y\|_1$ for such $x,y$. Since $E$ has empty interior, we obtain the inequality $d_N(x,y) \leq \|x-y\|_1$ for all $x,y \in \mathbb{R}^2$. The lower semicontinuity assumption allows us to avoid this type of behaviour; see \Cref{lemm:metric_derivative_Jacobian} below.
    
    The fact that $E = \left\{ x \in \Omega \colon N_{x} = 0 \right\}$ is non-separating guarantees that the quotient space is homeomorphic to $\Omega$ (\Cref{lemm:topology}). For example, if $E$ is the Euclidean unit circle and $N = \chi_{ \mathbb{R}^{2} \setminus E }\norm{ \cdot }_{2}$, the resulting quotient space is not a $2$-manifold.

    Now we discuss some of the related literature on non-smooth Finsler metrics. Perhaps the first investigations into this topic were carried out by Busemann--Mayer in \cite{BM:41}. Beginning in the 1940s, the Russian school led by Alexandrov developed a theory of \emph{surfaces of bounded curvature}, also now known as \emph{Alexandrov surfaces}, as a generalization of two-dimensional Riemannian geometry. See \cite{AZ:67} and \cite{Res:93} for an overview.

    Finsler metrics on Lipschitz manifolds were systematically studied by De Cecco--Palmieri in the series of papers \cite{DP:88, DP:90, DP:91,DP:95}. Note that they take a different approach to defining the distance $d_N$ from a norm field $N$. The idea is to make the distance more robust by making the definition insensitive to changes in $N$ on a set of measure zero. In particular, the norm field $N$ need only be defined on a full measure subset. This is achieved as follows. For a set $F \subset \mathbb{R}^2$ of measure zero, let $\Gamma_F$ be the family of absolutely continuous paths that intersect $F$ in a set of length zero. Then one defines the metric $d_{N,F}$ as in \Cref{eq:seminorm_admissible_distance} but restricting to paths in $\Gamma_F$. Next one defines $D_N(x,y) = \sup d_{N,F}(x,y)$, the supremum taken over all measure zero sets $F$. This is called the \emph{intrinsic distance} in \cite{DP:95,GPP:06} and \emph{essential metric} in \cite{AHPS:18} and further investigated in \cite{CS:19}. Observe that if $N$ is continuous, then the essential metric coincides with the metric considered in this paper. However, we do not take this approach, since the norm fields we have in mind typically vanish on a set of measure zero, and we prefer the additional flexibility of only requiring $N$ to be lower semicontinuous.
    

    \subsection{Properties of length} 
    
    In the remainder of this section, we establish properties of admissible norm fields and their corresponding metric. Our first lemma states that the property of lower semicontinuity of $N$ in each direction $v$ can be promoted to lower semicontinuity at a point in all directions uniformly. 
    
    \begin{lemm}\label{lemm:strong_LSC}
    Let $N$ be an admissible norm field and $x \in \Omega$. For every $\varepsilon > 0$, there exists $r > 0$ such that
    \begin{equation*}
        N_{y}(v)
        \geq
        ( 1 - \varepsilon )N_{x}(v).
    \end{equation*}
    for all $y \in B( x,  r )$ and $v \in \mathbb{R}^{2}$.
    \end{lemm}
    \begin{proof}
    If $N_x$ is the zero seminorm, then the conclusion follows immediately. Thus we may assume that $N_x$ is a norm. By the positive homogeneity of $N$, we need only consider vectors $v \in \mathbb{S}^1$. Let $\varepsilon >0$ and let $\delta = \varepsilon \omega(N_x)$, so that $N_x(v) - \delta \geq (1- \varepsilon)N_x(v)$ for all $v \in \mathbb{R}^2$. Thus it suffices to show that there exists a radius $r > 0$ such that
    \begin{equation*}
        N_{y}( v )
        \geq
        N_{x}( v ) - \delta
    \end{equation*}
    for all $y \in B( x,  r )$ and $v \in \mathbb{S}^{1}$.
    
    Assume to the contrary that no such $r$ exists. Then there exist sequences $(y_n) \subset \Omega$ and $(v_n) \subset \mathbb{S}^1$ for which
    \begin{equation} \label{equ:contradiction}
        N_{ y_{n} }( v_{n} )
        <
        N_{x}( v_{n} )
        -
        \delta
    \end{equation}
    for all $n \in \mathbb{N}$. By passing to a subsequence, we have that $v_{n}$ converges to some vector $v \in \mathbb{S}^{1}$.
    
    Let $M >0$ be such that $L(N_y) \leq M$ for all $y$ in a neighborhood of $x$. Then for every sufficiently large $n \in \mathbb{N}$,
    \begin{align*}
        N_{x}( v_{n} )
        -
        M
        \norm{ v - v_{n} }_2
        &\leq
        N_{x}( v )
    \end{align*}
    and
    \begin{align*}
        N_{ y_{n} }( v )
        &\leq
        N_{ y_{n} }( v_{n} )
        +
        M \norm{ v - v_{n} }_{2}.
    \end{align*}
    Moreover, the lower semicontinuity of $N$ implies that
    \[N_{x}( v )
        -
        \frac{ \delta }{ 2 }
        \leq N_{y_n}(v)\]
    for all sufficiently large $n \in \mathbb{N}$. Combining these inequalities yields
    \begin{equation*}
        N_{x}( v_{n} )
        -
        \left(
            2 M \norm{ v - v_{n} }_{2}
            +
            \frac{ \delta }{ 2 }
        \right)
        \leq
        N_{ y_{n} }( v_{n} ).
    \end{equation*}
    Let $n$ be sufficiently large so that $\|v-v_n\|_2 < \delta(4M)^{-1}$. Then the preceeding inequality contradicts \eqref{equ:contradiction}, and the result follows.
    \end{proof}
    
    The next lemma shows that the metric $d_N$ is locally well-behaved outside of the set $E$. 
    
    \begin{lemm}\label{lemm:quotient_map_invertible}
    Let $N$ be an admissible norm field. For all $x \in \Omega \setminus E$, there exists $r>0$ such that $\overline{B}(x,r) \subset \Omega \setminus E$ and the quotient map $\pi_{N}$ is bi-Lipschitz in the neighborhood $B(x,r)$.
    \end{lemm}
    \begin{proof}
    We let $\omega( x ) = \omega( N_{x} )$ denote the minimal stretching of $N$. \Cref{lemm:strong_LSC} implies that $\omega$ is lower semicontinuous. Also $\omega( x ) = 0$ if and only if $N_{x}$ is a seminorm.
    
    Let $x \in \Omega \setminus E$. Let $R>0$ be such that the closed ball $\overline{B}(x,R)$ is contained in $\Omega \setminus E$ and satisfies $\omega(z) \geq \omega(x)/2$ for all $z \in B(x,R)$. Such an $R>0$ exists by the lower semicontinuity of the map $z \mapsto \omega(z)$. Moreover, the local boundedness of $N$ implies that the maximal stretching $L( N_z )$ is bounded from above by $M$ for all $z \in B(x,R)$. We conclude that 
    \[\frac{\omega(x)}{2}\norm{ v }_{2} \leq N_{z}(v) \leq M\norm{ v }_{2} \]
    for all $z \in B(x,R)$ and all $v \in \mathbb{R}^2$.
    
    Let $r = R/2$. We claim that \[\frac{\omega(x)\norm{ y-z }_{2}}{4} \leq d_N(y,z) \leq M \norm{ y-z }_{2}\] for all $y,z \in B(x,r)$. Clearly, the line segment from $y$ to $z$ has $N$-length at most $M\norm{ y-z }_{2}$. For the lower bound, consider an arbitrary absolutely continuous path $\gamma$ from $y$ to $z$. If $|\gamma| \subset B(x,R)$, then we have the lower bound $\ell_N(\gamma) \geq \omega(x)\norm{ y-z }_{2}/2$. If $|\gamma|$ is not contained in $B(x,R)$, then its length is at least \[\omega(x)(R-r) = \frac{\omega(x)R}{2} \geq \frac{\omega(x)\norm{ y-z }_{2}}{4}.\] 
    Since our path is arbitrary, we obtain $d_N(y,z) \geq \omega(x)\norm{ y-z }_{2}/4$. We conclude that $d_N$ is bi-Lipschitz equivalent to the Euclidean distance on $B(x,r)$.
    \end{proof}

    \begin{lemm}\label{lemm:derivative_standard}
    For $\mathcal{L}^{2}$-almost every $x \in \Omega$, the metric derivative $\apmd{\pi_{N}}$ of $\pi_{N}$ at $x$ satisfies
    \begin{equation*}
        \apmd[\pi_{N}]{x}
        \leq
        N_x.
    \end{equation*}
    Moreover, for every $x \in \Omega$, 
    \begin{equation*}
        N_x
        \leq
        \apmd[\pi_{N}]{x}.
    \end{equation*}
    In particular, the metric derivative $\apmd{ \pi_{N} }$ equals $N$ $\mathcal{L}^{2}$-almost everywhere in $x \in \Omega$.
    \end{lemm}
    \begin{proof}
    First, we show that the upper bound $\apmd[ \pi_{N} ]{x} \leq N_x$ holds $\mathcal{L}^{2}$-almost everywhere in $\Omega$. Consider a fixed $v \in \mathbb{R}^{2} \setminus \left\{0\right\}$. The local boundedness of $N$ implies that the function $x \mapsto N_x(v)$ is locally integrable. Consider a rectangle $R \subset \Omega$ with one side parallel to $v$. There is a family of parallel line segments $\gamma_t \colon [0,h_0] \to R$, $\gamma_t(s) = x_t + vs$, that foliate $R$. Observe that for all $t$ and $s$, $D\gamma_{t}(s) = v$. The definition of $d_{N}$ implies that
    \begin{equation*}
        \apmd[ \pi_{N} ]{ \gamma_{t}(s) }( v )
        \leq
        \limsup_{ h \rightarrow 0^{+} }
            \frac{1}{h}
            \int_{ \left[s,  s+h\right] }
                N_{ \gamma_{t}(a) }( v )
            \, d\mathcal{L}^{1}(a).
    \end{equation*}
    According to Lebesgue's differentiation theorem, the $\limsup$ on the right-hand side equals $N_{ \gamma_{t}(s) }( v )$ for $\mathcal{L}^{1}$-almost every $s \in [0, h_0]$. Fubini's theorem implies that
    \begin{equation*}
        \apmd[ \pi_{N} ]{ x }( v )
        \leq
        N_x( v )
    \end{equation*}
    holds $\mathcal{L}^{2}$-almost everywhere in $R$. Since $R$ is arbitrary, the same conclusion holds for almost every point in $\Omega$. The first inequality follows.
    
    Next, we show that the inequality $N_x \leq \apmd[ \pi_{N} ]{x}$ holds for all $x \in \Omega$. In the case that $x \in E$, the conclusion is immediate since then $N_x = 0$. We consider now the case that $x \in \Omega \setminus E$. Let $v \in \mathbb{R}^{2} \setminus \left\{ 0 \right\}$ and let $\varepsilon > 0$. 
    
    Let $r>0$ be such that the conclusions of \Cref{lemm:strong_LSC} and \Cref{lemm:quotient_map_invertible} hold for the point $x$ and the given value of $\varepsilon$. In particular, \Cref{lemm:quotient_map_invertible} implies that there exists $\alpha>0$ such that 
        \[\alpha^{-1}d_N(y,z) \leq \|y-z\|_2 \leq \alpha d_N(y,z)\] 
    for all $y,z \in B(x,r)$. Moreover, the local boundedness of $N$ implies that there exists $M >0$ such that the maximal stretching $L(N_y) \leq M$ for all $y \in B(x,r)$. Let
    \[
        t_0
        =
        \frac{ 1 }{ \norm{ v }_{2} }
        \frac{r}{2\alpha}
        \min\left\{
            \frac{1}{\alpha},
            \frac{1}{\varepsilon M}
        \right\}.
    \]
    For all $t \in (0, t_0)$, consider an absolutely continuous path $\gamma_t\colon [0,1] \to \Omega$ joining $x$ to $x+tv$ that satisfies
    \begin{equation}
    \label{eq:choice_lower_bound}
        \int_{ 0 }^{ 1 }
            N \circ D\gamma_{t}
        \,d\mathcal{L}^{1}
        \leq
        d_{N}( x,  x + t v ) +
        \varepsilon t N_{x}(v).
    \end{equation}
    The right-hand side of \eqref{eq:choice_lower_bound} is bounded above by $\alpha t \norm{ v }_{2} + \varepsilon M t \norm{ v }_{2} < r/\alpha$. In particular, this implies that 
    \begin{equation}
    \label{eq:almost_minimizer}
        |\gamma_{t}|
        \subset
        B_{\|\cdot \|}( x,  r).
    \end{equation}
    Next, observe that
    \begin{equation}
    \label{eq:good_estimate_basepoint_norm}
        tN_x(v) 
        = 
        N_{x}(tv)
        \leq
        M \|tv\|_2
        \leq
        \alpha M d_{N}( x,  x + tv ).
    \end{equation}
    Applying now the conclusion of \Cref{lemm:strong_LSC} along $\gamma_{t}$, which is allowed due to \eqref{eq:almost_minimizer}, we have
    \begin{equation}
    \label{eq:good_estimate_along_N}
        ( 1 - \varepsilon )N_x( D\gamma_{t}(s) )
        \leq
        N \circ D\gamma_{t}(s),
    \end{equation}
    for almost every $s \in [0,1]$. Note that the norm field $N$ on the left-hand side has a fixed basepoint.
   
   Since straight line segments are geodesics with respect to the norm $N_{x}$, by integrating both sides of \eqref{eq:good_estimate_along_N} and applying \eqref{eq:choice_lower_bound} and \eqref{eq:good_estimate_basepoint_norm}, we obtain
    \begin{equation}
    \label{eq:the_final_point}
        ( 1 - \varepsilon )
        N_{x}( t v )
        \leq
         ( 1 + \varepsilon \alpha M)d_{N}( x,  x + tv ).
    \end{equation}
    We divide both sides of \eqref{eq:the_final_point} by $t$ and let $t \rightarrow 0$. We have
    \begin{equation*}
       ( 1 - \varepsilon )N_{x}(v) 
        \leq
        ( 1 + \varepsilon \alpha M )
        \liminf_{ t \rightarrow 0 }
            \frac{ d_{N}( x,  x + tv ) }{ t }.
    \end{equation*}
    The $\liminf$ on the right-hand side is bounded from above by the metric derivative $\apmd[ \pi_{N} ]{x}(v)$. The result follows by letting $\varepsilon \rightarrow 0$.
    \end{proof}
    
    \begin{lemm}\label{lemm:metric_derivative_Jacobian}
    For every Borel function $\rho \colon \Omega/\mathcal{E}_N \rightarrow \left[0,  \infty\right]$, we have the change of variables formula
    \begin{equation*}
        \int_{ \Omega }
            \rho \circ \pi_{N}
            J_{2}( N )
        \,d\mathcal{L}^{2}
        =
        \int_{ \Omega / \mathcal{E}_{N} }
            \rho
        \,d\mathcal{H}^{2}_{ \widetilde{d}_{N} }.
    \end{equation*}
    \end{lemm}
    \begin{proof}
    We proved in \Cref{lemm:derivative_standard} the fact that the metric derivative of $\pi_{N}$ equals $N$ $\mathcal{L}^{2}$-almost everywhere. Then the change of variables formula \Cref{prop:lipschitz} implies that $\mathcal{H}^{2}_{ \widetilde{d}_{N} }( \pi_{N}(E) ) = 0$. The fact that $\pi_{N}$ is injective in the complement of $E$ implies that the multiplicity term from \Cref{prop:lipschitz} can be omitted.
    \end{proof}
    
    \begin{lemm}\label{lemm:metric_speed_ABS}
    For every absolutely continuous path $\gamma$ in $\Omega$, $\ell_N(\gamma) = \ell_{d_N}( \pi_N \circ \gamma )$. In particular, the equality
    \begin{equation} \label{eq:metric_speed}
        v_{ \pi_{N} \circ \gamma }
        =
        N \circ D\gamma
    \end{equation}
    holds almost everywhere in the domain of $\gamma$.
    \end{lemm}
    \begin{proof}
    An immediate consequence of the definitions is that $\ell_{ \widetilde{d}_{N} }( \pi_{N} \circ \gamma ) \leq \ell_{N}( \gamma )$ for every absolutely continuous $\gamma$ in $\Omega$. For the other direction, let $L$ denote the $\apmd{\pi_N}$-length of $\gamma$:
    \begin{equation*}
        L
        =
        \int_I \apmd{\pi_N} \circ D\gamma(t)
        \,d\mathcal{L}^{1}(t).
    \end{equation*}
    Since $N(x) \leq \apmd{\pi_N}(x)$ for all $x \in \Omega$ by \Cref{lemm:derivative_standard}, we see that $\ell_N(\gamma) \leq L$. By \Cref{lemm:metric_speed_Lips}, the equality $L = \ell_{ \widetilde{d}_{N} }(\pi_N \circ \gamma)$ holds for all absolutely continuous $\gamma$. The equality $\ell_N(\gamma) = \ell_{ \widetilde{d}_{N} }(\pi_N \circ \gamma)$ now follows. The metric speed identity \eqref{eq:metric_speed} follows from the Lebesgue differentiation theorem.
    \end{proof}
 
    As a consequence of the previous lemma, whenever $\gamma\colon I \to \Omega$ is an absolutely continuous path, we have the integral formula
    \[
        \int_{\pi_N \circ \gamma} \rho\, ds_N
        =
        \int_{I} (\rho \circ \pi_{N}) (N \circ D\gamma)\,d\mathcal{L}^{1}
    \]
    for all Borel measurable functions $\rho: \Omega/\mathcal{E}_N \to [0, \infty]$.

\subsection{The quotient map} 
    
    \begin{prop}\label{prop:main_section}
    The quotient map $\pi_{N} \colon \Omega \rightarrow \Omega / \mathcal{E}_{N}$ is locally Lipschitz, locally bi-Lipschitz in the complement of $E$, and its restriction to $\Omega \setminus E$ is injective.
    
    Moreover, the map $\pi_{N}$ is closed and, for all $x \in \pi_N(E)$, the preimage $\pi^{-1}_{N}( x )$ is a connected and compact subset of $E$.
    \end{prop}
    \begin{proof}
    We already proved in \Cref{lemm:quotient_map_invertible} that $\pi_N$ is locally bi-Lipschitz outside of $E$. Moreover, since $N$ is locally bounded, $\pi_{N}$ is locally Lipschitz at all points in $\Omega$.
    
    Next, let $x \in \Omega \setminus E$ and $U \subset \Omega \setminus E$ be a neighborhood of $x$ such that $\pi_{N}|_{U}$ is bi-Lipschitz. The bi-Lipschitz property implies that $d_N(x,y)>0$ for all $y \in U$. Next, let $r>0$ be small enough so that $\overline{B}_{\norm{\cdot}}(x,r) \subset U$, and let $c = \inf\{ d_N(x,y) \colon y \in S_{\norm{\cdot}}(x,r)\}>0$. If $y \in \Omega \setminus U$, then any path from $x$ to $y$ must intersect $S_{\norm{\cdot}}(x,r)$, which gives $d_N(x,y) \geq c >0$. We conclude that $\pi_{N}$ is injective in the complement of $E$.
    
    Next, we prove that $\pi_N^{-1}( \widetilde{x} )$ is a connected compact subset of $E$ for all $\widetilde{x} \in \pi_N(E)$. Let $x \in \pi_{N}^{-1}( \widetilde{x} )$ and let $K$ be the component of $E$ containing $x$.
    
    Let $\gamma$ be a closed Jordan path in $\Omega \setminus E$ that separates $K$ and the boundary of $\partial \Omega$. See \cite[Section III.3]{Why:64} for the existence of such a path $\gamma$. Let $U$ be the complementary component of $|\gamma|$ containing $K$. 
    
    Let $c = \inf\left\{ d_N(x,z) \colon z \in \Omega \setminus \overline{U} \right\}$. The image $\abs{ \gamma }$ has a small neighborhood $V$ compactly contained in $\Omega \setminus E$. Every path joining the point $x$ to $\Omega \setminus \overline{U}$ must pass through $\overline{V}$. The lower semicontinuity of $N$ implies that $N \geq \alpha \norm{ \cdot }_{2}$ in $\overline{V}$ for some $\alpha > 0$ and hence that $c > 0$.
    
    Let $y \in \pi_N^{-1}( \widetilde{x} )$. Let $(\gamma_{n})$ be a sequence of Lipschitz paths joining $x$ to $y$ satisfying
    \begin{equation*}
        \ell_{N}( \gamma_{n} )
        \leq
        2^{-n} c
    \end{equation*}
    for all $n \in \mathbb{N}$. Observe that the image of each path $\gamma_n$ is contained in $\overline{U}$. Moreover, for every $z_{n} \in \abs{ \gamma_{n} }$, we have that $d_N(x, z_{n} ) \leq 2^{-n} c$. This implies that a subsequence of the sets $(|\gamma_n|)$ converges with respect to the Hausdorff distance to a connected subset of $\pi_N^{-1}( \widetilde{x} ) \cap \overline{U}$. This is a consequence of general properties of Hausdorff convergence in metric spaces; see Proposition 4.4.14 and Theorems 4.4.15 and 4.4.17 in \cite{Amb:Til:04}. 
    Since $y$ is arbitrary, we conclude that $\pi_N^{-1}( \widetilde{x} )$ is connected.
    
    Next, we check that $\pi_N^{-1}( \widetilde{x} ) \subset K$. Observe that, for a given point $z \in \Omega \setminus K$, the Jordan path $\gamma$ above can be chosen so that $K$ and $z$ are contained in different complementary components. An argument similar to that above shows that $d_N(K,z)>0$.
    
    The final step is to show that $\pi_{N}$ is closed. Let $F \subset \Omega$ be a closed set and let $\widetilde{x}$ be a limit point of $\pi_N(F)$. Since $\pi_N^{-1}( \widetilde{x} )$ is a singleton or contained in a component of $E$, there is a Jordan domain $U \subset \Omega$ such that $\partial U$ is contained in $\Omega \setminus E$ and separates $\pi_{N}^{-1}( \widetilde{x} )$ and $\partial \Omega$. Arguing as in the first part of the proof, we deduce that there is a constant $c>0$ such that $d_N( \pi_N^{-1}( \widetilde{x} ), z) \geq c$ for all $z \in \Omega \setminus U$. This implies that $\widetilde{x}$ is a limit point of $\pi_N(F \cap \overline{U})$. Let $( \widetilde{x}_j )$ be a sequence in $\pi_N(F \cap \overline{U})$ with limit $\widetilde{x}$. Let $(y_j)$ be a sequence in $F \cap \overline{U}$ such that $\pi_N( y_j ) = \widetilde{x}_j$. The compactness of $\overline{U}$ implies that there is a subsequence $(y_{j_k})$ that converges to a point $y \in F$. Since $\pi_N|_{\overline{U}}$ is Lipschitz, it follows that $(x_{j_k})$ converges to $\pi_N(y)$, and moreover that $\widetilde{x} = \pi_N(y)$. We conclude that $\widetilde{x} \in \pi_N(F)$, and hence that $\pi_N(F)$ is closed. 
    \end{proof}
    \begin{cor}\label{lemm:topology}
    The space $\Omega / \mathcal{E}_{N}$ is homeomorphic to $\Omega$.
    \end{cor}
    \begin{proof}
    By \Cref{prop:main_section}, $\pi_{N}$ is a closed and monotone map. Thus each element of the decomposition $\mathcal{E}_{N}$ is a planar continuum. Since the components of $E$ are non-separating, so are the elements of $\mathcal{E}_{N}$. It follows now from the classical theorem of Moore that $\Omega / \mathcal{E}_{N}$ is homeomorphic to $\Omega$. See, for instance, Theorem 25.1 in \cite{Dav:86}.
    \end{proof}
    
    Next we study the analytic properties of $\pi_{N}$. A consequence of \Cref{prop:lipschitz} and \Cref{lemm:derivative_standard} is that $x \mapsto L( N_{x} )$ is a minimal weak upper gradient of $\pi_{N}$. The following lemma identifies the minimal weak upper gradient of the inverse of $\pi_{N}$.
    \begin{lemm}\label{lemm:metric_derivative_minimal_upper_gradient}
    If $U \subset \Omega$ is an open set such that $\pi_{N}|_{U}$ is injective and its inverse $h$ is an element of $N^{1,2}_{\loc}( \pi_{N}( U ), \mathbb{R}^2)$, the function
    \begin{equation*}
        g
        =
        \left( 
            \frac{ 1 }{ \omega( N ) }
            \chi_{ U \setminus E }
        \right) 
        \circ 
        h
    \end{equation*}
    is a minimal weak upper gradient of $h$.
    \end{lemm}
    We use the convention $\frac{ 1 }{ 0 } \cdot 0 = 0$ in \Cref{lemm:metric_derivative_minimal_upper_gradient}.
    \begin{proof}
    We show that the function $g$ as in the claim is a weak upper gradient of $h$. First, the change of variables formula \Cref{lemm:metric_derivative_Jacobian} implies that $\mathcal{H}^{2}_{ \widetilde{d}_{N} }( \pi_{N}( E ) ) = 0$. Therefore the paths that have positive $\widetilde{d}_{N}$-length on $\pi_{N}( E )$ have zero modulus. Moreover, since $h$ is an element of $N^{1,  2}_{\loc}(\pi_{N}( U ), \mathbb{R}^2)$, $h$ maps almost every absolutely continuous path in $\pi_{N}( U )$ to an absolutely continuous path in $U$. Thus it suffices to check the upper gradient inequality for a path $\widetilde{\gamma} \colon [0,1] \to \pi_{N}( U )$ that intersects $\pi_{N}( E )$ in a set of $\widetilde{d}_{N}$-length zero and along which $h$ is absolutely continuous.  
    
    Let $\widetilde{\gamma} \colon [0,1] \to \pi_{N}( U )$ be such a path, and let $x = \widetilde{\gamma}(0)$ and $y = \widetilde{\gamma}(1)$. Let $\gamma = h \circ \widetilde{\gamma}$. Note that the absolute continuity of $h$ along $\widetilde{\gamma}$ implies that $\gamma$ intersects $E$ in a set of Euclidean length zero. Therefore, by reparametrizing, we can assume that the set $J = \gamma^{-1}( \Omega \setminus E )$ has full length in $[0,1]$.

    By \Cref{lemm:metric_speed_Lips}, the metric speed identity $v_{ \widetilde{\gamma} } = N \circ D\gamma$ holds $\mathcal{L}^1$-almost everywhere for $\gamma$. Also for almost every $t \in \left[0, 1\right] \setminus \gamma^{-1}( E )$,
    \begin{equation}
    \label{eq:speed:inequality}
        v_{ \gamma }( t )
        =
        \norm{ D\gamma(t) }_{2}
        \leq
        \frac{ 1 }{ \omega( N_{\gamma}(t) ) }
        N \circ D\gamma(t)
        =
        \frac{ 1 }{ \omega( N_{\gamma}(t) ) }
        v_{ \widetilde{\gamma} }(t),
    \end{equation}
    where $\omega( N_{\gamma(t)} )$ is the minimal stretching of $N$ at $\gamma(t)$. Since $\gamma^{-1}( E )$ has zero measure, we conclude from \eqref{eq:speed:inequality} that
    for almost every $t \in \left[0, 1\right]$,
    \begin{equation}
    \label{eq:speed:inequality:almost:complete}
        v_{ \gamma }( t )
        \leq
        \left(
            \frac{ 1 }{ \omega( N ) }
            \chi_{ \Omega \setminus E }
        \right)
        \circ \gamma(t)
        v_{ \widetilde{\gamma} }(t).
    \end{equation}
    The right-hand side in \eqref{eq:speed:inequality:almost:complete} equals $g \circ ( \pi_{N} \circ \gamma(t) ) v_{ \widetilde{\gamma} }(t)$. Therefore, integrating both sides of \eqref{eq:speed:inequality:almost:complete} implies that
    \begin{equation*}
        \|h(x) - h(y)\|_2
        \leq
        \int_{ \widetilde{\gamma} } g \,ds_{ \widetilde{d}_{N} }.
    \end{equation*}
    The local $L^{2}$-integrability of $g$ follows from the fact that $N$ has locally bounded distortion (\Cref{lemm:distortion_vs_dilatation}) and the change of variables formula (\Cref{lemm:metric_derivative_Jacobian}).
    
    We are left to check that $g$ is a minimal weak upper gradient. Let $\rho \in L^{2}_{\loc}( \pi_{N}( U ) )$ be a weak upper gradient of $h$. We want to show that $g(x) \leq \rho(x)$ for $\mathcal{H}^{2}_{d_{N}}$-almost every $x \in \pi_{N}( U )$. The set $\pi_{N}( E )$ is negligible, so it is sufficient to check this in the complement of $\pi_{N}( E )$. As $h$ is locally bi-Lipschitz in the complement of $\pi_{N}( E )$, it suffices to check that
    \begin{equation}
    \label{eq:almost}
        g \circ \pi_{N}(x)
        =
        \sup_{ v \in \mathbb{S}^{1} }
            \frac{ 1 }{ N_{x}( v ) }
        \leq
        \rho \circ \pi_{N}( x )
    \end{equation}
    $\mathcal{L}^{2}$-almost every $x \in U \setminus E$.
    
    Consider a square $R \subset U \setminus E$ and the accompanying foliation given by
    \begin{equation*}
        \gamma_{t}(s)
        =
        x_{0} + s v + t w,
    \end{equation*}
    where $v,w \in \mathbb{R}^2$ are orthogonal vectors and $s,t \in [-1,1]$.
    
    The metric speed characterization $v_{ \pi_{N} \circ \gamma } = N \circ D\gamma$ implies that
    \begin{equation*}
        v_{ \gamma_{t} }( s )
        =
        \norm{ v }_{2}
        \leq
        \rho \circ ( \pi_{N} \circ \gamma_{t}(s) )
        N_{ \gamma_{t}(s) }(v)
    \end{equation*}
    holds almost everywhere along the domain of $\gamma_{t}$ for almost every $t$. Fubini's theorem implies that
    \begin{equation*}
        \norm{ v }_{2}
        \leq
        \rho \circ \pi_{N}( x )
        N_{ x }( v )
    \end{equation*}
    holds $\mathcal{L}^{2}$-almost everywhere on $R$. Equivalently,
    \begin{equation*}
        1
        \leq
        \rho \circ \pi_{N}( x )
        N_{x}\left( \frac{ v }{ \norm{ v }_{2} } \right)
    \end{equation*}
    $\mathcal{L}^{2}$-almost everywhere on $R$. We can cover $U \setminus E$ by squares whose sides are parallel to $v$ and $w$, so we deduce that
    \begin{equation}
    \label{eq:essential:inequality}
        \frac{ 1 }{ N_{x}\left( \frac{ v }{ \norm{ v }_{2} } \right) }
        \leq
        \rho \circ \pi_{N}( x )
    \end{equation}
    $\mathcal{L}^{2}$-almost everywhere on $U \setminus E$.
    
    Let $D$ be a countable dense subset of $\mathbb{S}^{1}$. We have shown that, for $\mathcal{L}^{2}$-almost every $x \in U \setminus E$, \eqref{eq:essential:inequality} holds for every $v \in D$. Consequently, \eqref{eq:almost} holds for $\mathcal{L}^{2}$-almost every $x \in U \setminus E$.
    \end{proof}
    
    \subsection{Local quasiconformality} Let $U$ be a subdomain of $\Omega$ such that $\overline{U} \subset \Omega$ is compact. Since the norm field $N$ has locally bounded distortion, there exists $K( U ) < \infty$ such that 
    \begin{equation*}
    \label{eq:outer_distortion}
        L( N )^2
        \leq
        K(U) J_{2}( N )
    \end{equation*}
    for the maximal stretching $L( N )$ and the Jacobian $J_{2}( N )$. Recall that $L(N)$ is a weak upper gradient of $\pi_N$, $J_{2}( N )$ is the Jacobian of $\pi_{N}$, and that the pullback measure $\pi_N^*\mathcal{H}_{\widetilde{d}_N}^2$ is locally finite. Thus \Cref{prop:williams:L-Wversion} implies the following.
    \begin{prop}\label{prop:modulus_distortion}
    For every path family $\Gamma$ in $U$, we have that
    \begin{equation*}
    \label{eq:distortion_inequality}
        \Mod \Gamma
        \leq
        K(U) \Mod \pi_{N} \Gamma.
    \end{equation*}
    \end{prop}
    
    If $(\Omega/\mathcal{E}_N, \widetilde{d}_N)$ is reciprocal, then it admits some quasiconformal parametrization from a domain in Euclidean space. We show here that the map $\pi_N$ itself is a quasiconformal parametrization, at least locally.
    
    \begin{prop}\label{lemm:injectivity_obstruction}
    The metric surface $( \Omega / \mathcal{E}_{N}, \widetilde{d}_{N} )$ is reciprocal if and only if $\pi_{N}$ is a homeomorphism that is locally quasiconformal.
    \end{prop}
    Here, a map $\psi \colon X \rightarrow Y$ is \emph{locally quasiconformal} if every point $x \in X$ has a neighborhood $U$ such that the restriction of $\psi$ to $U$ is $K$-quasiconformal for some $K \geq 1$, where the $K$ is allowed to depend on $x$.
    \begin{proof}

    If $\pi_{N}$ is a locally quasiconformal homeomorphism, every point in $Y = ( \Omega / \mathcal{E}_{N},  \widetilde{d}_{N} )$ has a neighborhood that is reciprocal. By Corollary 1.4 of \cite{Iko:19}, this implies that $Y$ is reciprocal.
    
    Conversely, suppose that $Y$ is reciprocal. It suffices to fix an arbitrary Jordan domain $Q \subset \Omega$ with $\partial Q \cap E = \emptyset$ and check that $\pi_{N}|_{\interior Q}$ is quasiconformal.
    
    The reciprocality of $Y$ implies the existence of a homeomorphism $f \colon \pi_{N}( Q ) \rightarrow \overline{\mathbb{D}}$ that is $\frac{ \pi }{ 2 }$-quasiconformal in $\pi_{N}( Q )$. Let $h = f \circ \pi_{N}|_{ \interior Q }$. The claim follows if we can check that $h$ is quasiconformal.
    
    The map $h$ satisfies the assumptions of \Cref{prop:williams:L-Wversion} and the property (i). Therefore the quasiconformality of $h$ follows from its injectivity; see Section 3.1 of \cite{Ast:Iwa:Mar:09}. We check that $h$ is injective.
    
    Let $y \in \mathbb{D}$ and $C = h^{-1}( y )$. The set $C$ is connected and compact. The modulus of paths joining $y$ to the boundary $\mathbb{S}^{1}$ is zero by the $2$-Loewner property of $\overline{ \mathbb{D} }$ \cite[Example 8.24]{Hei:01}, hence (i) from \Cref{prop:williams:L-Wversion} implies that the modulus of paths joining $C$ to $\partial Q$ is zero. This happens only when $C$ is a singleton by the $2$-Loewner property of $\overline{ \mathbb{D} }$. Therefore $h$ is injective.
    \end{proof}

    \begin{rem} \label{rem:positive_measure}
    \Cref{lemm:injectivity_obstruction} gives two simple criteria for $( \Omega / \mathcal{E}_{N}, \widetilde{d}_{N} )$ to fail to be reciprocal. First, if $\mathcal{L}^{2}( E ) > 0$, then $\pi_{N}$ is not locally quasiconformal since Lusin's Condition ($N^{-1}$) is violated. Second, if $\pi_{N}$ is not injective, then $( \Omega / \mathcal{E}_{N}, \widetilde{d}_{N} )$ is not reciprocal.
    \end{rem} 
    
\section{Removable implies reciprocal}\label{sec:reciprocality}
    The objective of this section is to prove \Cref{thm:removable_implies_reciprocal}. An outline of the proof is as follows. First, we give a pair of reductions, \Cref{lemm:general:to:normal} and \Cref{lemm:subdomain:to:space}, showing that it suffices to consider only the case of admissible norm fields of the form $N = \omega \|\cdot\|_2$ defined on all of $\mathbb{R}^2$, for some bounded function $\omega\colon \mathbb{R}^2 \to [0, \infty)$. Next, \Cref{lemm:reciprocal_quadrilaterals_mod} gives a criterion for the mapping $\pi_N$ in our situation to be quasiconformal: it suffices to show that $\pi_N$ preserves the modulus of the path families $\Gamma(\xi_1, \xi_3;R)$ and $\Gamma(\xi_2, \xi_4;R)$ for a single rectangle $R$ containing $E$ with boundary edges $\xi_1, \xi_2, \xi_3, \xi_4$.
    
    We complete the proof by verifying the modulus condition of \Cref{lemm:reciprocal_quadrilaterals_mod}. This part is an application of the classical theorem of uniformization onto slit domains. This argument is based on the proof of Theorem 9 in \cite{AB:50}. In \Cref{sec:generalization}, we extend \Cref{thm:removable_implies_reciprocal} by relaxing the assumption that $L( N ) \in L^{\infty}_{ \loc } ( \Omega )$ to the assumption that $L( N ) \in L^{p}_{\loc}( \Omega )$ for some $p \in ( 2, \infty )$.
    
    \begin{lemm}\label{lemm:general:to:normal}
    An admissible norm field $N$ on $\Omega$ is reciprocal if and only if the norm field $\widehat{N} = \omega( N ) \norm{ \cdot }_{2}$ induced by the minimal stretching $\omega( N )$ is reciprocal.
    \end{lemm}
    \begin{proof}
    As we see from \Cref{lemm:injectivity_obstruction}, it suffices to show that the metrics generated by $N$ and $\widehat{N}$ are locally quasiconformally equivalent. Observe first that it follows directly from the definition that $\widehat{N} \leq N$. Since $N$ has locally bounded distortion, every point has a neighborhood $U$ such that $N_{x} \leq H \widehat{N}_{x}$ for some $H > 0$ and every $x \in U$. These facts imply that the corresponding distances are locally bi-Lipschitz equivalent.
    \end{proof}
    For the following lemma, fix a subdomain $\Omega' \subset \Omega$ that contains $E$ and is compactly contained in $\Omega$. Let $K = \overline{ \Omega' }$. Since $N = \omega \norm{ \cdot }_{2}$ is locally bounded, there exists $\alpha > 0$ such that $\omega < \alpha$ everywhere on $K$. We define
    \begin{equation*}
    \label{eq:norm:field:extension}
        \widehat{N}
        =
        \left(
            \omega \chi_{ K }
            +
            \alpha \chi_{ \mathbb{R}^{2} \setminus K }
        \right)
        \norm{ \cdot }_{2}.
    \end{equation*}
    The choice of $\alpha$ implies that $\widehat{N}$ is admissible on $\mathbb{R}^{2}$ vanishing exactly on $E$. Also, $\widehat{N}$ coincides with $N$ in $\Omega'$.
    \begin{lemm}\label{lemm:subdomain:to:space}
    The norm field $N = \omega \norm{ \cdot }_{2}$ is reciprocal in $\Omega$ if and only if the extension $\widehat{N}$ is reciprocal in $\mathbb{R}^{2}$. Moreover, in either one of these cases the quotient maps $\pi_{ \widehat{N} }$ and $\pi_{N}$ are 1-quasiconformal homeomorphisms.
    \end{lemm}
    \begin{proof}
    First of all, since $N$ and $\widehat{N}$ are equal in $\Omega'$, there exists a homeomorphism
    \begin{equation*}
        f
        \colon
        \pi_{N}( \Omega' )
        \rightarrow
        \pi_{ \widehat{N} }( \Omega' )
    \end{equation*}
    for which $\pi_{ \widehat{N} } = \pi_{N} \circ f$ on $\Omega'$. In fact, the map $f$ is a local isometry and hence 1-quasiconformal.
    
    Since the restrictions of $\pi_{N}$ and $\pi_{ \widehat{N} }$ to the complement of $E$ are locally bi-Lipschitz, we deduce that they are locally quasiconformal if and only if their restrictions to $\Omega'$ are locally quasiconformal. These two conditions are equivalent for the maps since $f$ is quasiconformal. We conclude from \Cref{lemm:injectivity_obstruction} that $N$ is reciprocal if and only if $\widehat{N}$ is reciprocal.
    
    We are left to check that if $\pi_{N}$ is locally quasiconformal, then it is actually 1-quasiconformal. Combining \Cref{prop:williams:L-Wversion} with the local quasiconformality of $\pi_{N}$, we conclude that $h = \pi_{N}^{-1}$ has the Sobolev regularity required for \Cref{lemm:metric_derivative_minimal_upper_gradient}. Therefore
    \begin{equation*}
        \rho
        =
        \left(
            \frac{ 1 }{ \omega }
            \chi_{ \Omega \setminus E  }
        \right)
        \circ
        h
    \end{equation*}
    is a minimal weak upper gradient of $h$. The change of variables formula \Cref{lemm:metric_derivative_Jacobian} and \Cref{prop:williams:L-Wversion} imply that the outer dilatation of $h$ is bounded from above by one. The outer dilatation bound for $\pi_N$ follows from \Cref{prop:modulus_distortion}. We conclude that $\pi_{N}$ is 1-quasiconformal. The 1-quasiconformality of $\pi_{ \widehat{N} }$ is argued in a similar manner.
    \end{proof}
    
   \subsection{A criterion for quasiconformality}\label{sec:criterions} 
    
    We prove \Cref{lemm:reciprocal_quadrilaterals_mod} in this section. We consider an admissible norm field $N = \omega \norm{ \cdot }_{2}$ defined on a domain $\Omega \subset \mathbb{R}^{2}$ vanishing exactly on a non-separating compact set $E \subset \Omega$.
    
    We consider a quadrilateral $Q \subset \Omega$ whose boundary $\partial Q$ does not intersect the set $E$. Let $( \xi_{1},  \xi_{2},  \xi_{3},  \xi_{4} )$ be a decomposition of $\partial Q$ into four arcs labelled in counterclockwise order.
    
    Since $\partial Q$ does not intersect $E$, the restriction of $\pi_{N}$ to $\partial Q$ is a homeomorphism (\Cref{prop:main_section}). As a consequence of \Cref{lemm:topology}, the image $\pi_{N} Q$ is a Jordan domain with boundary $\pi_{N} \partial Q$ and the arcs $( \pi_{N} \xi_{1},  \pi_{N} \xi_{2},  \pi_{N} \xi_{3},  \pi_{N} \xi_{4} )$ decompose $\pi_{N}( \partial Q )$.
    
    We fix some notation for the following proof. Let
    \begin{align*}
        \Gamma_{1}
        &=
        \Gamma\left(
            \xi_{1},  \xi_{3};  Q
        \right)
        \quad \text{and} \quad
        \widetilde{\Gamma}_{1}
        =
        \Gamma\left(
            \pi_{N}\xi_{1},  \pi_{N}\xi_{3};  \pi_{N}Q
        \right);
        \\
        \Gamma_{2}
        &=
        \Gamma\left(
            \xi_{2},  \xi_{4};  Q
        \right)
        \quad \text{and} \quad
        \widetilde{\Gamma}_{2}
        =
        \Gamma\left(
            \pi_{N}\xi_{1},  \pi_{N}\xi_{3};  \pi_{N}Q
        \right).
    \end{align*}
    We defined $\Gamma(F_1, F_2; G)$ in \Cref{sec:quasiconformal_mappings}. Notice that $\pi_{N} \Gamma_{1} \subset \widetilde{\Gamma}_{1}$ and $\pi_{N} \Gamma_{2} \subset \widetilde{\Gamma}_{2}$.
    
    \begin{prop}\label{lemm:reciprocal_quadrilaterals_mod}
    Let $N = \omega \norm{ \cdot }_{2}$ be admissible. If $\Mod \Gamma_{1} = \Mod \widetilde{\Gamma}_{1}$ and $\Mod \Gamma_{2} = \Mod \widetilde{\Gamma}_{2}$, then the restriction of $\pi_{N}$ to $Q$ is a homeomorphism and 1-quasiconformal.
    \end{prop}
    \begin{proof}
    \Cref{prop:modulus_distortion} and the special form of $N$ imply that the outer dilatation of $\pi_{N}$ is one so we only need to check that $\pi_{N}|_{Q}$ is injective and that its inverse has its outer dilatation bounded from above by one.
    
    It was proved in \cite{RR:19} that there exists a continuous function
    \begin{equation*}
        \widetilde{u}_{1}
        \colon
        \pi_{N} Q
        \rightarrow
        \left[0,  1\right] 
    \end{equation*}
    in the Sobolev space $N^{1,  2}( \pi_{N}Q )$ whose minimal weak upper gradient $\widetilde{\rho}_{1}$ is a minimizer for $\Mod \widetilde{\Gamma}_{1}$. The function $\widetilde{u}_{1}$ satisfies the boundary conditions $\widetilde{u}_{1}( \pi_{N} \xi_{1} ) = 0$ and $\widetilde{u}_{1}( \pi_{N} \xi_{3} ) = 1$.
    
    Consider $u_{1} = \widetilde{u}_{1} \circ \pi_{N}$. Since $N = \omega \norm{ \cdot }_{2}$ and $\pi_{N}$ has bounded outer dilatation, it is readily verified that $\rho_{1} = (\widetilde{\rho}_{1} \circ \pi_{N}) \omega \in L^{2}( Q )$ is a weak upper gradient of $u_{1}$ with $L^{2}$-norm $\Mod \widetilde{ \Gamma }_{1} = \Mod \Gamma_{1}$. Therefore $u_{1} \in N^{1, 2}( Q )$.

    A consequence of Weyl's lemma \cite[A.6.10]{Ast:Iwa:Mar:09} and continuity of $u_{1}$ is that $u_{1}$ is harmonic in the interior of $Q$; it minimizes the Dirichlet energy among continuous Sobolev maps $u \colon Q \rightarrow \left[0, 1\right]$ with boundary values $u( \xi_{1} ) = 0$ and $u( \xi_{3} ) = 1$.
    
    We repeat the above construction for the path families $\Gamma_{2}$ and $\widetilde{\Gamma}_{2}$. Let $u_{2}$ and $\widetilde{u}_{2}$ denote the corresponding functions, where $u_{2}( \xi_{2} ) = 0$ and $u_{2}( \xi_{4} ) = 1$.
    
    Let $M$ denote the modulus of $\Gamma_{1}$. A consequence of Riemann mapping theorem is that $M u_{2}$ is a harmonic conjugate of $u_{1}$ and the restriction of $f = ( u_{1}, M u_{2}) $ to the interior of $Q$ is conformal. The map extends as a homeomorphism to the boundary $\partial Q$.
    
    Let $\widetilde{f} = ( \widetilde{u}_{1},  M \widetilde{u}_{2} )$. Then $f = \widetilde{f} \circ \pi_{N}$ by construction. This identity and bijectivity of $f$ imply that the restriction of $\pi_{N}$ to $Q$ is a homeomorphism.
    
    Since the restriction of $\pi_{N}$ to $Q$ is a homeomorphism and $\partial Q$ does not intersect $E$, we find a Jordan neighborhood $U \supset Q$ for which $\pi_{N}|_{U}$ is a homeomorphism and that $U \cap E = Q \cap E$ (\Cref{prop:main_section}). Let $h$ denote the inverse of $\pi_{N}|_{ U }$.
    
    We claim that $h \in N^{1,  2}_{\loc}( \pi_{N}(U),  U )$. Since $\pi_{N}$ is locally bi-Lipschitz in the complement of $E$ and $E \cap U \subset \mathrm{int}( Q )$, it suffices to verify that $h|_{ V }$ is an element of $N^{1, 2}_{\loc}( V, U )$, where $V = \pi_{N}( \mathrm{int}(Q))$. This regularity follows readily since the restriction of $f$ to the interior of $Q$ is locally bi-Lipschitz, $\widetilde{f}$ is an element of $N^{1,  2}\left( \pi_{N}Q,  \left[0,  1\right] \times \left[0,  M\right] \right)$, and $h = f^{-1} \circ \widetilde{f}$ in $V$. Now the outer dilatation bound $K_O(h) \leq 1$ follows from \Cref{lemm:metric_derivative_minimal_upper_gradient} and the change of variables formula of $\pi_{N}$.
    \end{proof}
    
    \begin{rem}
    \Cref{lemm:reciprocal_quadrilaterals_mod} is related to a question posed by Rajala in \cite{Raj:17}. Rajala asks whether the reciprocal upper bound \eqref{upper:bound} implies that points have zero modulus in the sense of \eqref{point:zero:modulus}. This is the case whenever $N = \omega \norm{ \cdot }_{2}$ is admissible and satisfies the sharp modulus upper bound of one in \eqref{upper:bound}. The proof relies on the existence and Sobolev regularity of energy minimizers of quadrilaterals on metric spaces studied on \cite{Raj:17,RR:19}. The key observation is that if $\pi_{N}$ preserves the modulus of opposite sides of $\partial Q$, then the quotient map $\pi_{N}$ pulls back these minimizers to harmonic minimizers on the Euclidean space.
    \end{rem}
    
\subsection{Proof of \Cref{thm:removable_implies_reciprocal}}\label{sec:S_implies_R}

    Let $E \subset \mathbb{R}^{2}$ be removable for conformal mappings. We want to prove that for any domain $\Omega \supset E$ and admissible norm field $N \colon \Omega \times \mathbb{R}^{2} \rightarrow \left[0, \infty\right)$ vanishing exactly on $E$, the quotient space $( \Omega / \mathcal{E}_{N}, \widetilde{d}_{N} )$ is reciprocal.
    
    As shown in \Cref{lemm:general:to:normal} and \Cref{lemm:subdomain:to:space}, we only need to consider the case where $\Omega = \mathbb{R}^{2}$ and $N = \omega \norm{ \cdot }_{2}$.
    
    Let $R = [a,b] \times [c,d]$ be a rectangle whose interior contains $E$. Let $\xi_{1} = \{a\} \times [c,d]$, $\xi_{2} = [a,b] \times \{c\}$, $\xi_{3} = \{b\} \times [c,d]$, and $\xi_{4} = [a,b] \times \{d\}$. Let $\Gamma_{1} = \Gamma( \xi_{1},  \xi_{3};  R )$ and $\Gamma_{2} = \Gamma( \xi_{2},  \xi_{4};  R )$.
    
    Let $\widetilde{ \Gamma }_{1}$ denote the family of paths joining $\pi_{N} \xi_{1}$ to $\pi_{N} \xi_{3}$ in $\pi_{N} R$ and $\widetilde{ \Gamma }_{2}$ the family of paths joining $\pi_{N} \xi_{2}$ to $\pi_{N} \xi_{4}$ in $\pi_{N} R$. We claim that $\Mod \widetilde{ \Gamma }_{1} = \Mod \Gamma_{1}$ and $\Mod \widetilde{ \Gamma }_{2} = \Mod \Gamma_{2}$. \Cref{lemm:reciprocal_quadrilaterals_mod} then implies that $\pi_{N}$ is $1$-quasiconformal in the interior of $R$. Since $R$ is an arbitrary rectangle containing $E$, it then follows that $\pi_N$ is globally 1-quasiconformal.
    
    Observe that the inequalities $\Mod \widetilde{ \Gamma }_{1} \geq \Mod \Gamma_{1}$ and $\Mod \widetilde{ \Gamma }_{2} \geq \Mod \Gamma_{2}$ hold in general by \Cref{prop:modulus_distortion}. Thus we only need to verify the opposite inequalities.
    
    A standard fact is that there is a sequence of finitely connected domains $\Omega_k \subset \mathbb{R}^2 \setminus E$ such that $\overline{\Omega}_k \subset \Omega_{k+1}$ for all $k \in \mathbb{N}$, each component of $\partial \Omega_k$ is a closed analytic Jordan path, and $\bigcup_{k=1}^\infty \Omega_k = \mathbb{R}^2 \setminus E$. We assume without loss of generality that $\partial R \subset \Omega_{1}$.
    
    For each $n \in \mathbb{N}$, there exists a conformal embedding $\varphi_{n} \colon \Omega_{n} \rightarrow \mathbb{R}^{2}$ normalized as
    \begin{equation*}
        \varphi_{n}(z)
        =
        z
        +
        \frac{ a_{1,n} }{ z }
        +
        \cdots
    \end{equation*}
    near $\infty$ such that the real part of $a_{1,n}$ is the smallest among all conformal embeddings $\psi \colon \Omega_{n} \rightarrow \mathbb{R}^{2}$ of the form
    \begin{equation}
    \label{eq:normalization}
        \psi(z)
        =
        z
        +
        \frac{ a }{ z }
        +
        \cdots .
    \end{equation}
    See for example Section V.2 of \cite{Gol:69}.
    
    For each $n \in \mathbb{N}$, the minimizer $\varphi_{n}$ is unique and its image is a domain $U_{n} \subset \mathbb{R}^{2}$ whose complement consists of finitely many arc segments parallel to the vertical axis.
    
    Fix $n$ and consider $k \geq n$. By the minimality of the real part of $a_{1,  n}$, we have that $0 \geq \re( a_{1, k} ) \geq \re( a_{1,  n} )$. Hence the mappings $\varphi_k|_{\Omega_n}\colon \Omega_n \to \mathbb{R}^2$ form a normal family. See for example the proof of Theorem 1 of \cite[Section V.2]{Gol:69} for details. A diagonal argument then implies that $( \varphi_{n} )$ is a normal family. Thus every subsequence of $(\varphi_n)$ has a further subsequence converging uniformly on compact sets to a conformal map $f \colon \mathbb{R}^{2} \setminus E \rightarrow \mathbb{R}^{2}$ satisfying the normalization \eqref{eq:normalization} around $\infty$. By the removability of $E$, the map extends to a Möbius transformation, and thus \eqref{eq:normalization} implies that $f(z) = z$ for all $z \in \mathbb{R}^{2}$. Hence the sequence $( \varphi_{k} )$ itself must converge to the identity map uniformly on compact sets of $\mathbb{R}^{2} \setminus E$.
    
    Let $Q_{n}$ denote the quadrilateral bounded by the Jordan curve $\varphi_{n}( \partial R )$. The quadrilateral $Q_{n}$ converges to $R$ with respect to Hausdorff distance as $n \rightarrow \infty$. Let $\pi_{1}$ and $\pi_{2}$ denote projection onto the $x$-axis and $y$-axis, respectively, and let $a_{n} = \sup \pi_1(\varphi_n(\xi_1))$, $b_{n} = \inf \pi_1(\varphi_n(\xi_3))$, $c_{n} = \inf \pi_2(\varphi_n(\xi_2))$ and $d_{n} = \sup \pi_2(\varphi_n(\xi_4))$.
    
    Let $R_n = [a_n, b_n] \times [c_n,d_n]$ and $\widehat{E}_n = \mathbb{R}^2 \setminus \varphi_n(\Omega_n)$. Observe that $\widehat{E}_n$ consists of finitely many vertical slits $S_1, \ldots, S_m$.
    
    There exists $n_{0}$ such that for all $n \geq n_{0}$, the slits $\widehat{E}_{n}$ are contained in the interior of $R_{n}$, and that $0 < b_{n} - a_{n}$ and $0 < d_{n} - c_{n}$. Fix such an $n$. We claim that
    \begin{equation}
    \label{eq:slit_modulus}
        \Mod \widetilde{ \Gamma }_{1}
        \leq
        \frac{ d_{n} - c_{n} }{ b_{n} - a_{n} }.
    \end{equation}
    Consider the function $\rho_n \colon \mathbb{R}^2 / \mathcal{E}_{N} \to [0, \infty]$ defined as zero in the complement of $\pi_{N}( \Omega_{n} )$, and otherwise by
    \[
        \rho_n
        =
        \left(
            \frac{
                \chi_{
                    R_n \setminus \widehat{E}_n
                }
            }
            { b_{n} - a_{n} }
            \circ \varphi_{n}
            \frac{ J_{ \varphi_{n} }^{-1/2} }{ \omega }
        \right)
        \circ ( \pi_N|_{ \Omega_{n} } )^{-1}.
    \]
    We claim that $\rho_n$ is admissible for $\widetilde{ \Gamma }_{1}$. Let $\gamma \in \widetilde{ \Gamma }_{1}$ be locally rectifiable with respect to $\widetilde{d}_N$.
    
    We consider the restriction of $\gamma$ to the set $I = \gamma^{-1}\left( \pi_{N} \varphi_{n}^{-1}( R_{n} \setminus \widehat{E}_{n} ) \right)$. We have
    \begin{equation*}
        \int_{ \gamma } \rho_{n} \,ds_{ \widetilde{d}_{N} }
        \geq
        \int_{ I }
            \rho_{n} \circ \gamma
            v_{ \gamma }
        \,d\mathcal{L}^{1}.
    \end{equation*}
    Here $\theta = \varphi_{n} \circ ( \pi_{N}|_{ \Omega_{n+1} } )^{-1} \circ \gamma|_{I}$ is well-defined and
    \begin{equation*}
        \int_{ I }
            \rho_{n} \circ \gamma
            v_{ \gamma }
        \,d\mathcal{L}^{1}
        =
        \int_{ I }
            \frac{
                \chi_{
                    R_n \setminus \widehat{E}_n
                }
            }
            { b_{n} - a_{n} }
            \circ \theta
            v_{ \theta }
        \,d\mathcal{L}^{1}.
    \end{equation*}
    Since $\widehat{E}_{n}$ consists of finitely many vertical slits, we conclude using the area formula for paths and the projection onto the $x$-axis that
    \begin{equation*}
        \int_{ I }
            \frac{
                \chi_{
                    R_n \setminus \widehat{E}_n
                }
            }
            { b_{n} - a_{n} }
            \circ \theta
            v_{ \theta }
        \,d\mathcal{L}^{1}
        \geq
        \frac{ 1 }{ b_{n} - a_{n} }
        \mathcal{L}^{1}\left( \abs{ \pi_{1} \circ \theta } \right)
        \geq
        1.
    \end{equation*}
    Therefore
    \[
        \int_{\gamma} \rho_n\,ds_{ \widetilde{d}_N }
        \geq
        1,
    \]
    and we conclude that $\rho_n$ is admissible. The change of variables formulas for $\pi_{N}$ and $\varphi_{n}$ yield that
    \begin{align*}
        \int \rho_n^2\, d\mathcal{H}_{ \widetilde{d}_N }^2
        =
        \frac{ d_{n} - c_{n} }{ b_{n} - a_{n} }.
    \end{align*}
    This verifies \eqref{eq:slit_modulus}. Finally, observe that $d_n - c_n \to d-c$ and $b_n-a_n \to b-a$ as $n \to \infty$. This shows that
    \[
        \Mod \widetilde{ \Gamma }_{1}
        \leq
        \frac{ d - c }{ b - a }
        =
        \Mod \Gamma_1.
    \]
    A similar argument, using conformal mappings onto horizontal slit domains, shows that $\Mod \widetilde{ \Gamma }_{2} \leq \Mod \Gamma_2$. This completes the proof.

\subsection{An extension of \Cref{thm:removable_implies_reciprocal} to integrable norm fields}\label{sec:generalization}
    In this section we extend \Cref{thm:removable_implies_reciprocal} to the case of lower semicontinuous norm fields $N$ with locally bounded distortion such that $L(N) \in L_{\loc}^p(\Omega)$ for some $p \in (2, \infty)$. We assume that $N$ vanishes exactly on a compact set $E \subset \Omega$ that is removable for conformal mappings.
    
    For this section, we allow the possibility for $N_x$ to be infinite at some points $x \in \Omega$. To say this more precisely, in the definition of seminorm in \Cref{sec:seminorms}, we consider a seminorm to be a function $S \colon \mathbb{R}^2 \to [0, \infty]$ satisfying the same assumptions listed there, following the convention that $0 \cdot \infty = 0$. An admissible norm field is now a function $N \colon \Omega \times \mathbb{R}^2 \to [0, \infty]$ satisfying the conditions of \Cref{defi:seminorm_admissible}, except that local boundedness of $N$ is now replaced by the assumption that $L(N) \in L_{\loc}^p(\Omega)$.  Observe that the local boundedness of the distortion then implies that if $N_x(v) = \infty$ for some $v \in \mathbb{R}^2\setminus \{0\}$, then $N_x$ must have the form
    \[
        N_x(v) = \begin{cases} \infty & \text{if } v \neq 0 \\ 0 & \text{if } v = 0 \end{cases} .
    \] 
    In particular, $\omega(N_x) = L(N_x) = \infty$. Note also that the minimal stretching $\omega(N)$ is lower semicontinuous, and that \Cref{lemm:strong_LSC} remains true for $x \in \Omega$ with $\omega( N_{x} ) < \infty$.
    
    We define the pseudodistance $d_{N}$ exactly as in \Cref{eq:seminorm_admissible_distance}. It is no longer immediately clear that $d_{N}( x, y ) < \infty$ for any given pair of points $x, y \in \Omega$, but this can be shown from the local $L^p$-integrability of $L(N_x)$, since $p > 2$. See \cite[Thm. 1.1]{LW:18} and \cite[Prop. 5.12]{CS:19} for proofs of the analogous fact in related settings.

    As before, we identify $x, y \in \Omega$ if $d_{N}( x, y ) = 0$ and let $X$ denote the corresponding quotient space. Let $\pi \colon \Omega \rightarrow X$ denote the associated quotient map. The quotient distance $d_X$ on $X$ is defined as follows: for every $x,y \in X$, we set $d_X(x,y) = d_N(\pi^{-1}(x), \pi^{-1}(y))$, observing that this is independent of the choice of element in $\pi^{-1}(x)$ and $\pi^{-1}(y)$ and hence well-defined.
    
    Next, we establish that $\pi$ is continuous and an element of $N^{1,p}_{\loc}( \Omega, X )$. To this end, for every $x \in \Omega$, $f_{x}( z ) = d_{N}( x, z )$ is measurable as a consequence of \cite[Theorem 9.3.1]{HKST:15}. Since $L(N)$ is an upper gradient of $\pi$, it is also an upper gradient of $f_{x}$. Then Morrey's embedding theorem \cite[Theorem 9.2.14]{HKST:15} implies that $f_{x}$ is locally Hölder continuous with Hölder exponent and Hölder constant independent of $x$. This implies that $\pi$ is locally Hölder continuous. Thus $\pi$ is continuous with upper gradient $L( N ) \in L_{\loc}^{p}( \Omega )$. We are now ready for the main result of this section. We recall that $N$ is assumed to vanish on a compact set $E$ removable for conformal mappings.
    \begin{prop}\label{prop:reciprocality:p-int}
    The metric space $X$ has locally finite Hausdorff $2$-measure, and the quotient map $\pi$ is a locally quasiconformal homeomorphism. In particular, $X$ is a quasiconformal surface.
    \end{prop}

    \begin{proof}
    We first prove that $\pi$ is a homeomorphism. To this end, let $\omega(z) = \omega( N_{z} )$ and $\widehat{ N }_{z} = \omega( z )\norm{ \cdot }_{2}$ for every $z \in \Omega$. For each $k \in \mathbb{N}$, we define the function $\omega_k \colon \Omega \to [0, \infty)$ by
    \begin{equation*}
        \omega_{k}(z)
        =
        \min\left\{
            \omega(z), k
        \right\}.
    \end{equation*}
    Each function $\omega_{k}$ is bounded and lower semicontinuous in $\Omega$, and $\omega_{k}(z) = 0$ if and only if $\omega( z ) = 0$. For every $z \in \Omega$, the sequence $( \omega_{k}(z) )_{ k = 1 }^{ \infty }$ is non-decreasing and converges to $\omega(z)$.

    Let $N_{k} = \omega_{k} \norm{ \cdot }_{2}$ and $d_{k} = d_{ N_{k} }$. Since $N_k$ is bounded and lower semicontinuous, \Cref{thm:removable_implies_reciprocal} implies that $\pi_{k} \colon \Omega \rightarrow ( \Omega, d_{k} )$ defined by $\pi_{k}(z) = z$ is a $1$-quasiconformal homeomorphism. Since $N_{k} \leq \omega(N)\norm{ \cdot }_{2} \leq N$ everywhere, we see that
    \begin{equation*}
        d_{k}( \pi_{k}(x), \pi_{k}(y) ) \leq d_{X}( \pi(x), \pi(y) )
    \end{equation*}
    for all $x, y \in \Omega$. Since $\pi_{k}$ is a homeomorphism, we see that $\pi$ is injective. Now the map $\psi_{k} \colon X \rightarrow ( \Omega, d_{k} )$ defined by $\psi_k = \pi_{k} \circ \pi^{-1}$ is $1$-Lipschitz, hence $\pi^{-1} = \pi_{k}^{-1} \circ \psi_{k}$ is continuous. Therefore $\pi$ is a homeomorphism.
    
    Recall that $N$ has locally bounded distortion. 
    From this and the fact that, for every $x \in X$, $\pi^{-1}( \overline{B}_{X}( x, r ) )$ is compact for sufficiently small $r > 0$, we see that the induced distances $d_{N}$ and $d_{ \widehat{N} }$ are locally bi-Lipschitz equivalent. We assume from this point onwards, without loss of generality, that $N = \widehat{N} = \omega \norm{ \cdot }_{2}$.
    
    Let $\Gamma_{0}$ denote the family of paths along which $\omega = L( N )$ fails to be integrable. Since $\omega \in L^{p}_{\loc}( \Omega ) \subset L^{2}_{ \loc }( \Omega )$, the family $\Gamma_{0}$ has zero modulus. Recall from \Cref{lemm:metric_speed_ABS} that, for any absolutely continuous path $\theta$ in $\Omega$,
    \begin{equation*}
        \ell_{ d_{X} }( \pi \circ \theta )
        \geq
        \lim_{ k \rightarrow \infty }
        \ell_{ d_{k} }( \psi_{k} \circ \theta )
        =
        \lim_{ k \rightarrow \infty }
        \ell_{ N_{k} }( \theta )
        =
        \ell_{ N }( \theta ),
    \end{equation*}
    where the latter equality follows from monotone convergence. If $\theta \not\in \Gamma_{0}$, we have $\ell_{ N }( \theta ) < \infty$ and the definition of $d_{ X }$ implies $\ell_{ d_{X} }( \pi \circ \theta ) \leq \ell_{ N }( \theta )$. So $\ell_{ d_{X} }( \pi \circ \theta ) = \ell_{ N }( \theta )$. Since the equality holds for all absolutely continuous paths outside the negligible family $\Gamma_{0}$, we conclude from Sections 3.3 and 3.4 of \cite{LW:18} that $L( N ) = \omega$ is a minimal weak upper gradient of $\pi \in N^{1, 2}_{ \loc }( \Omega, X )$ and the Jacobian of $\pi$ equals $J_{2}( N ) = \omega^{2}$ $\mathcal{L}^{2}$-almost everywhere. Since we also have that $\pi \in N^{1, p}_{ \loc }( \Omega, X )$ for $p > 2$, it satisfies Lusin's Condition ($N$) \cite[Theorem 7.1]{Vod:00}. Therefore, for each compact set $K \subset \Omega$,
    \begin{equation*}
        \mathcal{H}^{2}_{X}( \pi( K ) )
        =
        \int_{ K }
            \omega^{2}
        \,d\mathcal{L}_{2}
        <
        \infty.
    \end{equation*}
    We conclude that $X$ has locally finite Hausdorff $2$-measure. An application of \Cref{prop:williams:L-Wversion} yields that $K_{O}( \pi ) = 1$.
    
    The proof is complete after we verify $K_{O}( \pi^{-1} ) = 1$. Since $\pi_{k}$ is $1$-quasiconformal for every $k$, it suffices to verify $K_{O}( \psi_{k} ) = 1$ for some $k$. To this end, we fix $k \in \mathbb{N}$ and recall that $\psi_{k}$ is $1$-Lipschitz.
    
    Since $\pi_{ k }$ is a quasiconformal homeomorphism, it satisfies Lusin's Condition ($N^{-1}$). This implies that the map $\pi^{-1}$ satisfies Lusin's Condition ($N$). As a consequence, the Jacobian of $\psi_{k}$ coincides with $\rho_{k}^{2}$ for $\rho_{k} = ( (\omega_{k}/\omega) \chi_{ \Omega \setminus E } ) \circ \pi^{-1}$.
     
    Since $\psi_{k}$ is Lipschitz, we have $\psi_{k} \in N^{1, 2}_{ \loc }( X, ( \Omega, d_{k} ) )$. We claim that any minimal weak upper gradient of $\psi_{k}$ coincides with $\rho_{k}$ almost everywhere in $X$. If we verify this, then $K_{O}( \psi_{k} ) = 1$ follows from \Cref{prop:williams:L-Wversion}.
    
    Consider an absolutely continuous path $\gamma \colon \left[0, 1\right] \rightarrow X$ with $\abs{ \gamma } \subset X \setminus \pi( E )$. Then $\psi_{k} \circ \gamma$ is absolutely continuous, and since $d_{k}$ and $\norm{ \cdot }_{2}$ are locally bi-Lipschitz equivalent in a neighborhood of the image of $\theta = \pi_{ k }^{-1} \circ \psi_{k} \circ \gamma$, the path $\theta$ is absolutely continuous with respect to $\norm{ \cdot }_{2}$. Then by monotone convergence and \Cref{lemm:metric_speed_ABS},
    \begin{equation*}
        \ell_{ N }( \theta )
        =
        \lim_{ n \rightarrow \infty }
        \ell_{ N_{n} }( \theta )
        =
        \lim_{ n \rightarrow \infty }
        \ell_{ d_{n} }( \psi_{n} \circ \gamma ).
    \end{equation*}
    Since every $\psi_{n}$ is $1$-Lipschitz,
    \begin{equation*}
        \lim_{ n \rightarrow \infty }
        \ell_{ d_{n} }( \psi_{n} \circ \gamma )
        \leq
        \ell_{ d_{X} }( \gamma ).
    \end{equation*}
    Therefore $\ell_{N}( \theta ) \leq \ell_{ d_{X} }( \gamma ) < \infty$, and, by the construction of $d_{X}$, $\ell_{ d_{X} }( \gamma ) \leq \ell_{ N }( \theta )$.
    
    Since $\ell_{N}( \theta ) = \ell_{ d_{X} }( \gamma )$ holds for every subpath of $\gamma$, we see that $v_{ \gamma } = N \circ D\theta = \omega \circ \theta v_{\theta}$ and $v_{ \psi_{k} \circ \gamma } = \omega_{k} \circ \theta v_{\theta}$ almost everywhere in the domain of $\gamma$. We conclude from this that
    \begin{equation}
    \label{eq:identity}
        v_{ \psi_{k} \circ \gamma }
        =
        \rho_{k} \circ \gamma
        v_{ \gamma }
    \end{equation}
    almost everywhere with respect to the length measure of $\gamma$. If $\widetilde{\Gamma}_{0}$ denotes the family of absolutely continuous paths in $X$ that have positive length on the set $\pi( E )$, then $\mathcal{H}^{2}_{X}( \pi(E) ) = 0$ implies $\Mod \widetilde{\Gamma}_{0} = 0$. The equality \eqref{eq:identity} remains valid for every absolutely continuous path $\gamma \not\in \widetilde{ \Gamma }_{0}$, which implies that $\rho_{k}$ is a weak upper gradient of $\psi_{k}$. The minimality of $\rho_{k}$ is immediate from \eqref{eq:identity}. So any minimal weak upper gradient of $\psi_{k}$ coincides with $\rho_{k}$ $\mathcal{H}^{2}_{X}$-almost everywhere.
    \end{proof}
    
    \begin{rem}\label{rem:exten:sharp} 
    The norm field $N = \omega \norm{ \cdot }_{2}$ defined by the weight $\omega( x ) = \norm{ x }_{2}^{-1} ( 1 - \log \norm{ x }_{2} )^{-1} \in L^{2}( \mathbb{D} )$ induces a complete hyperbolic metric on the punctured unit disk. In particular, the origin is at infinite distance from any other point. Consequently, the assumption $p > 2$ in \Cref{prop:reciprocality:p-int} cannot be relaxed.
    \end{rem}

\section{Reciprocal implies removable} \label{sec:reciprocal_implies_removable}

    This section is dedicated to a proof of \Cref{thm:reciprocal_implies_removable}. Recall that we consider a compact set $E \subset \Omega$ for which $\Omega \setminus E$ is connected, together with the norm field $N$ defined by $N_x = \min\left\{ 1,  d_{\norm{\cdot}_{2}}( E, x )^p \right\}\norm{ \cdot }_{2}$ for some $p > \max\left\{ \dim_{\mathcal{H}} E- 1,  0 \right\}$. The norm field $N$ induces a decomposition $\mathcal{E}_{ N }$ of $\Omega$, a metric $\widetilde{d}_N$ on $\Omega/ \mathcal{E}_N$, and a quotient map $\pi \colon \Omega \to (\Omega/\mathcal{E}_{N},  \widetilde{d}_{ N })$, as described in \Cref{sec:distances}.
    
\subsection{Decay of the norm field near $E$} 
    The following lemma states that if $N$ decays to zero sufficently fast near $E$, then each component of $E$ collapses to a point under the quotient map $\pi_{N}$.
    
    \begin{lemm}\label{lemm:good_snowflake}
    Let $N_x = \min\{1, d_{\|\cdot\|_2}(x,E)^p\}\norm{ \cdot }_{2}$. For all $p > \max\left\{ \dim_{\mathcal{H}} E- 1,  0 \right\}$, $\mathcal{H}_{ \widetilde{d}_{N} }^1( \pi_{N}(E) ) = 0$. Consequently, the preimage of every $x \in \pi_{N}( E )$ is a connected component of $E$.
    \end{lemm}
    \begin{proof}
    Let $p > \dim_{\mathcal{H}} E- 1$ and let $\varepsilon>0$. By the definition of Hausdorff dimension, there exists $\delta > 0$ and a countable collection of sets $\mathcal{A} = \{A_j\}$ such that $E \subset \bigcup_j A_j$, $\diam_{\|\cdot\|_2} A_j \leq \delta$  for all $j$, and 
        \[\sum_j (\diam_{\|\cdot\|_2} A_j)^{p+1} < \varepsilon.\]
    Without loss of generality, we may assume that $A_j \cap E \neq \emptyset$ for all $j$. Let $d_j = \diam_{\|\cdot\|_2} A_j$. Thus $A_j \subset \overline{B}_{\|\cdot\|_2}(y, d_j)$ for some $y \in E$. By integrating $N$ over the straight-line path from $y$ to a point $z \in A_j$, it follows that
    \[d_N(y,z) \leq \int_0^{d_j} t^p\,dt = \frac{d_j^{p+1}}{p+1}. \]
    Thus $\diam_{d_N} A_j \leq 2(p+1)^{-1} d_j^{p+1} < 2(p+1)^{-1} \delta^{p+1}$, and we get $\sum_j \diam_{d_N} A_j < 2(p+1)^{-1}\varepsilon.$ This is sufficient to show that $\mathcal{H}_{ \widetilde{d}_{N} }^1( \pi_{N}(E) ) = 0$.
    
    
    Next, let $x \in \pi_{N}( E )$. \Cref{prop:main_section} implies that $\pi_{N}^{-1}( x )$ is a subset of a connected component $F$ of $E$. Since $\pi_{N}( F )$ is connected and compact subset of $\pi_{N}( E )$, we have that
    \begin{equation*}
        \diam \pi_{N}( F )
        \leq
        \mathcal{H}_{ \widetilde{d}_{N} }^1( \pi_{N}( F ) )
        \leq
        \mathcal{H}_{ \widetilde{d}_{N} }^1( \pi_{N}( E ) )
        =
        0.
    \end{equation*}
    Hence $\pi_{N}( F ) = x$ and we must have $F = \pi_{N}^{-1}( x )$.
    \end{proof}

\subsection{Proof of \Cref{thm:reciprocal_implies_removable}}

    We first observe that if $(\Omega/\mathcal{E}_{N},  \widetilde{d}_{ N })$ is reciprocal, then the space formed by taking the same set $E$ and the same definition for $N$, but applied to all points $x \in \mathbb{R}^2$, is also reciprocal. Thus the choice of domain $\Omega$ is not relevant for the proof, and we assume for the remainder of the section that $\Omega = \mathbb{R}^2$.
    
    We prove the contrapositive: if $E$ is not removable for conformal mappings, then $( \mathbb{R}^{2} / \mathcal{E}_{N},  \widetilde{d}_{N} )$ is not reciprocal.
    
    Let $E \subset \mathbb{R}^2$ be a set that is not removable for conformal mappings. As a consequence of \Cref{prop:removable_set_characterizations}, there is a compact set $\widehat{E} \subset \mathbb{R}^{2}$ of positive measure and a conformal map $f \colon \mathbb{R}^{2} \setminus \widehat{E} \to \mathbb{R}^{2} \setminus E$. Let $\widehat{N} = \chi_{ \mathbb{R}^{2} \setminus \widehat{E} }\norm{ \cdot }_{2}$ and let $\widehat{\pi} \colon \mathbb{R}^{2} \to (\mathbb{R}^{2}/\mathcal{E}_{ \widehat{N} }, \widetilde{d}_{\widehat{N}})$ be the associated quotient map. Observe that $\widehat{N}$ is an admissible norm field vanishing on the set $\widehat{E}$.
    
    The following lemma states that $f$ extends to a mapping of the respective quotient spaces. For brevity, let $\widehat{Y} = \mathbb{R}^{2} / \mathcal{E}_{ \widehat{N} }$ and $Y = \mathbb{R}^{2} / \mathcal{E}_{ N }$.
    
    \begin{lemm}\label{lemm:conformal_extension}
    The map $f\colon \mathbb{R}^2 \setminus \widehat{E} \to \mathbb{R}^2 \setminus E$ induces a continuous monotone map $\widehat{f} \colon \widehat{Y} \rightarrow Y$. That is, there is a monotone map $\widehat{f}\colon \widehat{Y} \to Y$ satisfying $\widehat{f}\circ \widehat{\pi}(x) = \pi_N \circ f(x)$ for all $x \in \mathbb{R}^2 \setminus \widehat{E}$. 
    \end{lemm}
    \begin{proof}
    Let $y \in \widehat{Y}$, and let $\widehat{F}$ denote its preimage under $\widehat{\pi}$. If $\widehat{F} = \{x\}$ for some point $x \notin \widehat{E}$, then we set $\widehat{f}(y) = \pi_N \circ f(x)$. 
    
    Otherwise, $\widehat{F}$ is a subset of some component $\widehat{A}$ of $\widehat{E}$. For all $m \in \mathbb{N}$, let $\widehat{\gamma}_m$ be a Jordan path with image contained in $B_{\|\cdot\|_2}( \widehat{A},  1/m) \setminus \widehat{E}$ that separates $\widehat{A}$ and infinity. The curve $|\widehat{\gamma}_m|$ is the boundary of a closed region $\widehat{A}_m$ containing $\widehat{A}$. We assume without loss of generality that $| \widehat{\gamma}_{m+1} | \subset \widehat{A}_{m}$ for all $m$.
    
    By assumption, $\gamma_m = f \circ \widehat{\gamma}_m$ is a Jordan loop whose image bounds a compactly contained domain $A_m$. Let $A = \bigcap_m \overline{A}_m$. It is immediate that $A$ is nonempty and compact. The intersection is also connected; see for example Section 28 of \cite{Wil:70}. This implies that $A$ is a connected component of $E$. Therefore $\pi_{N}( A )$ is a point by \Cref{lemm:good_snowflake}. We define $\widehat{f}( y ) = \pi_{N}( A )$.
    
    We now check that $\widehat{f}$ is continuous. Let $y \in \widehat{Y}$ and let $(y_n)$ be a sequence in $\widehat{Y}$ converging to $y$. Let $\widehat{F}_n = \widehat{\pi}^{-1}(y_n)$. In the case that $\widehat{F} = \{x\}$ for some $x \notin \widehat{E}$, the continuity is obvious. Otherwise, we proceed as follows. For each fixed $m \in \mathbb{N}$, $F_n \subset \widehat{A}_m$ for sufficiently large $n$. This implies that $\widehat{f}(y_n) \subset \pi_N(A_m)$. Therefore the accumulation points of $\widehat{f}(y_n)$ are in the intersection of $\pi_{N}( A_{m} )$. Since the intersection equals $\pi_{N}( A )$, the sequence $\widehat{f}( y_{n} )$ converges to $\pi_{N}( A ) = \widehat{f}(y)$. The continuity follows.
    
    By construction, the preimage of a point in $\mathbb{R}^2 / \mathcal{E}$ under $\widehat{\pi} \circ \widehat{f}$ is either a single-point set or a component of $\widehat{E}$. We conclude that $\widehat{f}$ is monotone.
    \end{proof}
    
    Let $R = \left[a,  b\right] \times \left[c,  d\right] \subset \mathbb{R}^{2}$ be a rectangle whose interior contains $\widehat{E}$. Let $\Gamma_{1}$ denote the family of paths $\gamma_{t}\colon [a,b] \to \mathbb{R}^2$, where $t \in [c,d]$, defined by $\gamma_t(s) = (s,t)$. Thus $\Gamma_{1}$ is a foliation of $R$ by horizontal paths. Let $\Gamma_{2}$ denote the corresponding foliation of $R$ by vertical paths. 
    
    Next, let $Q$ be the Jordan domain bounded by $f( \partial R )$, and let $\xi_{1}, \xi_{2}, \xi_{3}, \xi_{4}$ denote, respectively, the image of the left, bottom, right, and top side of $R$.
    Let $\widetilde{\Gamma}_{1}$ denote the family paths joining $\pi_N\xi_{1}$ to $\pi_N\xi_{3}$ in $\pi_NQ$ and $\widetilde{\Gamma}_{2}$ the family of paths joining $\pi_N\xi_{2}$ to $\pi_N\xi_{4}$.
    
   	By \Cref{lemm:subdomain:to:space}, it suffices to show that $Y$ is not 1-reciprocal. Thus the proof is complete after we verify the inequalities
   	\begin{equation}
    \label{eq:estimate:1}
        1
        <
        \Mod \widehat{\pi} \Gamma_{1}
        \Mod \widehat{\pi} \Gamma_{2}
    \end{equation} and 
    \begin{equation}
    \label{eq:estimate:3}
        \Mod \widehat{\pi} \Gamma_{1}
        \Mod \widehat{\pi} \Gamma_{2}
        \leq
        \Mod \widetilde{\Gamma}_{1}
        \Mod \widetilde{\Gamma}_{2}.
    \end{equation}

	Define the function $P \colon \mathbb{R}^2 \to [0,\infty]$ by
	\begin{equation*}
	    P(x)
	    =
	    \begin{cases}
	        L(N_{f(x)}) \norm{ D_{x}f } & \text{ if } x \not\in \widehat{E}
	        \\
	        0 & \text{ if } x \in \widehat{E}
	    \end{cases}
	     .
	\end{equation*}
    Since $N$ is a weighted Euclidean norm and $f$ is conformal in the complement of $\widehat{E}$, it follows that $N \circ D_{x}f(v) = P(x)\norm{v}_{2}$ for all $v \in \mathbb{R}^{2}$ and all $x \in \mathbb{R}^2\setminus \widehat{E}$.

    We consider the function $\widehat{ P } \colon \mathbb{R}^{2} / \widehat{ \mathcal{E} } \to [0, \infty]$ defined by taking $\widehat{ P }(x) = P(\widehat{\pi}^{-1}(x))$. Observe that $\widehat{ P }$ is well-defined since $\widehat{\pi}$ is injective outside of $\widehat{E}$. Loosely speaking, $\widehat{ P }$ is a weak upper gradient of $\widehat{ f }$.

	Let $\rho \colon \widehat{f}( \widehat{R} ) \rightarrow \left[0, \infty\right]$ be an admissible function for $\widetilde{\Gamma}_{1}$, and let $\widehat{\rho} = ( \rho \circ \widehat{f} )\widehat{P}$. We first observe that
	\begin{equation}
	\label{eq:change:of:variables:widehat}
	    \int_{ \widehat{R} }
	        \widehat{\rho}^{2}
	    \,d\mathcal{H}^{2}_{ \widehat{d} }
	    =
	    \int_{ \widehat{f}(\widehat{R}) }
	        \rho^{2}
	    \,d\mathcal{H}^{2}_{ \widetilde{d}_{N} }.
	\end{equation}
	Indeed, the integrals are left unchanged by the removal of $\pi_{N}( E )$ and $\widehat{ \pi }(\widehat{ E })$ from both sides. With this reduction, the identity \eqref{eq:change:of:variables:widehat} follows from the Jacobian identities $J_{f} \equiv \norm{ Df }^{2}$, $J_{ \widehat{\pi } } = \chi_{ \mathbb{R}^{2} \setminus E }$, and $J_{ \pi_{N} } = L^{2}( N )$.
	
	Next, we claim that $\widehat{\rho}$ is weakly admissible for $\widehat{\pi} \Gamma_{1}$. Let $\widehat{ \gamma_{t} }$ denote the image under $\widehat{\pi}$ of the horizontal path $\gamma_{t}$ in the quotient space $\mathbb{R}^{2} / \widehat{ \mathcal{E} }$.
	\Cref{lemm:metric_speed_ABS} implies that 
	\begin{equation}
	\label{eq:metric_speed_characterization}
	    v_{ \widehat{f} \circ \widehat{\gamma_{t}} }( s )
	    =
	    N \circ Df \circ D\gamma_{t}( s )
	    =
	    \widehat{P} \circ \widehat{\gamma_{t}}(s) v_{ \widehat{\gamma_{t}} }(s)
	\end{equation}
	for $\mathcal{L}^{1}$-almost every $s \in [a,b] \setminus \gamma_{t}^{-1}( \widehat{E} )$ and that the total variation of $\widehat{ \gamma_{t} }$ in $\widehat{\pi} \widehat{E}$ is zero. Similarly, since $\mathcal{H}^{1}_{ \widetilde{d}_{N} }( \pi_{N}(E) ) = 0$ by \Cref{lemm:good_snowflake}, the area formula \cite[Theorem 2.10.13]{Fed:69} for paths implies that the total variation of $\widehat{f} \circ \widehat{\gamma_{t}}$ in $\pi_{N}(E)$ is zero. We conclude that $\widehat{f} \circ \widehat{ \gamma_{t} }$ is absolutely continuous as long as the right-hand side of \eqref{eq:metric_speed_characterization} is integrable.

	Observe that \eqref{eq:change:of:variables:widehat} holds with the characteristic function $\chi_{ \widehat{f} \widehat{R} }$ in place of $\rho$ and $\widehat{P}$ in place of $\widehat{\rho}$. Then an application of Fubini's theorem implies that the function in the right-hand side of \eqref{eq:metric_speed_characterization} is integrable for $\mathcal{L}^{1}$-almost every $t$. For such $t$, we conclude from \eqref{eq:metric_speed_characterization} that
	\begin{equation*}
        1
	    \leq
        \int_{ \widehat{ f } \circ \widehat{ \gamma_{t} } }
            \rho
        \,ds
        =
        \int_{ \widehat{ \gamma_{t} } }
            \widehat{ \rho }
        \,ds.
    \end{equation*}
    Therefore $\widehat{\rho}$ is weakly admissible for $\widehat{\pi} \Gamma_{1}$, and the equality \eqref{eq:change:of:variables:widehat} implies that
	\begin{equation*}
	    \Mod \widehat{\pi} \Gamma_{1}
	    \leq
	    \Mod \widetilde{\Gamma}_{1}.
	\end{equation*}
	A similar argument applied to the path family $\widehat{\pi} \Gamma_{2}$ gives $\Mod \widehat{\pi} \Gamma_{2} \leq \Mod \widetilde{\Gamma}_{2}$. The inequality \eqref{eq:estimate:3} now follows.

	To conclude the proof, we prove \eqref{eq:estimate:1}. Let $\rho$ be admissible for $\widehat{ \pi } \Gamma_{1}$. Then for all $t \in [c,d]$, we have $1 \leq \int_c^d \rho \circ \widehat{ \pi } \chi_{\mathbb{R}^2 \setminus \widehat{E} }(s,t)\,dt$. Applying Fubini's theorem and Hölder's inequality gives
	\begin{align*}
    d-c &
        \leq
        \int_R 
            \rho \circ \widehat{\pi} \chi_{\mathbb{R}^2 \setminus \widehat{E} }
	    \, d\mathcal{L}^{2} 
	    \leq 
	    \left(
	        \int_R 
	            \rho \circ \widehat{\pi}^2 \chi_{ \mathbb{R}^{2} \setminus \widehat{E} }
	        \, d\mathcal{L}^{2}
	   \right)^{1/2}
	   \mathcal{L}^{2}( R \setminus \widehat{E} )^{1/2}.
	\end{align*}
	After rearranging and taking the infimum over admissible $\rho$, we find that
	\[
	    (d-c)^2/\mathcal{L}^{2}(R \setminus \widehat{E})
	    \leq
	    \Mod \widehat{ \pi } \Gamma_{1}.
	\]
	The analogous argument gives $(b-a)^2/\mathcal{L}^{2}(R \setminus \widehat{E}) \leq \Mod \widehat{ \pi } \Gamma_{2}.$ Thus
    \[
	    1
	    <
	    \frac{
	        (b-a)^2(d-c)^2
	  }{
	        \mathcal{L}^{2}( R \setminus \widehat{E} )^2
	  }
	    \leq
	    \Mod \widehat{ \pi } \Gamma_{1}
	    \Mod \widehat{ \pi } \Gamma_{2}.
    \]

\section{Linear Cantor sets: two examples} \label{sec:linear_Cantor_sets}
    
    We call a Cantor set $E \subset \mathbb{R} \times \left\{0\right\}$ a \emph{linear Cantor set}. As remarked in \Cref{sec:main_results}, a norm field vanishing on a linear Cantor set $E$ of positive length may or may not be reciprocal.
    For completeness, we include here two explicit examples to illustrate both of these cases. Recall from the discussion following the statement of \Cref{thm:reciprocal_implies_removable} that a compact set $E \subset [0,1] \times \{0\}$ is removable for conformal mappings if there exists an admissible norm field $N$ vanishing on $E$ that is reciprocal. Such an $E$ is necessarily a linear Cantor set by \Cref{lemm:injectivity_obstruction}. Conversely, if there exists an admissible norm field vanishing on a linear Cantor set $E$ that is not reciprocal, then $E$ is not removable for conformal mappings. Versions of these examples are already present in \cite[Section 7]{AB:50}. A closely related construction, and the one that we directly based \Cref{ex:weight_not_reciprocal} on, is found in Section 11 of an early version of the paper \cite{Sch:95}.
    
    \begin{exm}\label{ex:weight_not_reciprocal}
    We construct a lower semicontinuous weight $\omega \colon \mathbb{R}^2 \to [0, \infty]$ that vanishes on a Cantor set $E \subset [0,1] \times \{0\}$ of positive length such that the space $(\mathbb{R}^2, d_\omega)$ is not reciprocal. The idea is to make $E$ sufficiently large so that the modulus of the path family joining $(0,  0)$ to $(0,  1)$ in $( \mathbb{R}^{2},  d_{\omega} )$ is positive. 
    
    Identify $[0,1]$ with the set $[0,1] \times \{0\} \subset \mathbb{R}^2$. Let $a_1 = 1/2$, and now define inductively sequences $(a_j)$, $(b_j)$ by the rules $b_j = a_j/\exp(4^j)$ and $a_{j+1} = (a_j - b_j)/2$. Let $I_1$ be a closed interval centered at $t_1 = 1/2$ of length $2b_1$. Define next intervals $I_j$ inductively as follows. Assume that we have a collection of disjoint intervals $I_1, \ldots, I_{j-1}$. From the complement of $(0,1) \setminus \bigcup_{k=1}^{j-1} I_k$, choose an open interval $J$ of largest length. Let $t_j$ be the midpoint of $J$, and let $I_j$ be the closed interval centered at $t_j$ of length $2b_j$. We record the observation that $d_{\|\cdot\|_2}(t_j, \{0,1\}) = \min\{t_j, 1-t_j\} \geq a_j$. Let $E = [0,1] \setminus \bigcup_j I_j$, and let $\omega = \chi_{\mathbb{R}^2 \setminus E}$.
    
    Consider now an interval $I_j$. Assume in the first case that $t_j \leq 1/2$. For all $t \in (0, t_j-b_j)$, let $\gamma_{j,t}$ be the path that connects $t$ to $2t_j -t$ along the upper semicircle of the circle centered at $t_j$ with radius $t_j-t$. Let $\Gamma_j$ be the family of all such paths $\gamma_{j,t}$. Observe that $\Gamma_j$ is a full-modulus subfamily of the family of paths in the upper half-plane $H$ that separate the sets $\overline{B}((t_j,0),b_j)$ and $H \setminus B((t_j,0), t_j)$.
    
    Since the metric speed of $\gamma \in \Gamma_{j}$ with respect to Euclidean distance and with respect to $d_{\omega}$ coincide almost everywhere along $\gamma$, the modulus of $\Gamma_j$ with respect to the metric $d_\omega$ equals the Euclidean modulus: $\Mod_{ d_{\omega} } \Gamma_j = \log(t_j/b_j)/\pi$. See for example \cite[Lemma 7.18]{Hei:01}. Next, we consider the case that $t_j > 1/2$. For all $t \in (t_j+b_j,1)$, let $\gamma_{j,t}$ be the path connecting $t$ to $2t_j-t$ along the upper semicircle of the circle centered at $t_j$ with radius $t-t_j$. In this case, we have $\Mod_{ d_{\omega} } \Gamma_j = \log((1-t_j)/b_j)/\pi$.
    
    We claim that the metric $d_\omega$ violates reciprocality condition \eqref{point:zero:modulus}. Let $F_1 = \{(0,0)\}$ and $F_2 = \{(1,0)\}$ and let $\Gamma = \Gamma(F_1,F_2;\mathbb{R}^2)$. Recall the notation $\Gamma(F_1,F_2;G)$ defined in \Cref{sec:quasiconformal_mappings}.
    
    Observe that $\Gamma$ is a subfamily of $\Gamma(F_1, \mathbb{R}^2\setminus \mathbb{D};\mathbb{R}^2)$, which is majorized by the annular path families $\Gamma(B_{\norm{\cdot}_2}(0,\varepsilon), \mathbb{R}^2 \setminus \mathbb{D};\mathbb{R}^2)$ for all $\varepsilon>0$. In particular,
    \begin{equation*}
        \Mod_{d_\omega} \Gamma(B_{\norm{\cdot}_2}(0,\varepsilon), \mathbb{R}^2 \setminus \mathbb{D};\mathbb{R}^2) \geq
        \Mod_{d_\omega} \Gamma
    \end{equation*}
    for all $\varepsilon>0$. Thus it is sufficient to show that $\Mod_{d_\omega} \Gamma > 0$.  

    Let $\rho$ be an admissible function for $\Gamma$ for the metric $d_\omega$. For each $j \in \mathbb{N}$, let $m_j = \inf \{ \int_\gamma \rho\,ds_\omega: \gamma \in \Gamma_j\}$. If $m_j > 0$, this implies that $\rho/m_j$ is admissible for the path family $\Gamma_j$, and thus that 
    \begin{equation}
        \label{eq:ex.6.1:modulus:lowerbound}
        \int_{A_j} \frac{\rho^2}{m_j^2}\, d\mathcal{H}_\omega^2
        \geq
        \Mod \Gamma_j
        =
        \frac{\log(\min\{t_j,1-t_j\}/b_j)}{\pi}
        \geq
        \frac{\log(a_j/b_j)}{\pi}.
    \end{equation} 
    Let $\gamma_j$ be a path in $\Gamma_j$ such that $\int_{\gamma_j} \rho\,ds_\omega \leq \max\{2m_j,2^{-j-1}\}$. We can define a path $\gamma\colon [0,1] \to \mathbb{R}^2$ by
    \[
        \gamma(t)
        =
        (t,
        \sup_{j\in \mathbb{N}} \pi_2( |\gamma_j| \cap (\{t\}\times \mathbb{R}) )
        ).
    \] 
    Here, $\pi_2$ denotes projection onto the vertical axis. We see that the path $\gamma$ is $d_\omega$-rectifiable as follows. For every $t \in [0,1]$ except on a countable set, there is $\varepsilon > 0$ such that $\gamma((t-\varepsilon, t+\varepsilon))$ is a subpath of a single path $\gamma_j$. At such points $t$, $\gamma$ is locally rectifiable. Otherwise, there is $\varepsilon>0$ such that the image of $\gamma((t-\varepsilon, t+\varepsilon))$ is contained in the union of the image of two paths $\gamma_j$, $\gamma_k$. Again, it follows that $\gamma$ is locally rectifiable at these points. We conclude that $\gamma$ itself is rectifiable. Observe that
    \begin{align*}
        1
        \leq
        \int_\gamma \rho\,ds
        & \leq 
        \sum_{j=1}^\infty \int_{\gamma_j} \rho\,ds 
        \leq
        \sum_{j=1}^\infty \max\{2m_j,2^{-j-1}\}.
    \end{align*}
    This implies that the relationship $m_j \geq 1/(2\cdot 2^j)$ must hold for some $j \in \mathbb{N}$. This together with \eqref{eq:ex.6.1:modulus:lowerbound} gives 
    \begin{align*}
        \frac{1}{2\cdot 2^j}
        \leq
        m_j
        &\leq
        \left(
            \frac{\pi}{\log(a_j/b_j)}
        \right)^{1/2}
        \left(
            \int_{\mathbb{R}^2}
                \rho^2
            \,d\mathcal{H}_{\norm{\cdot}_2}^2
        \right)^{1/2}
        \\
        &=
        \left(
            \frac{\pi}{4^j}
        \right)^{1/2}
        \left(
            \int_{\mathbb{R}^2}
                \rho^2
            \,d\mathcal{H}_{\norm{\cdot}_2}^2
        \right)^{1/2}.
    \end{align*}
    This yields the lower bound
    \[
        \frac{1}{4\pi}
        \leq
        \int_{\mathbb{R}^n}\rho^2\,d\mathcal{H}^2.
    \]
    We conclude that $(\mathbb{R}^2, d_\omega)$ is not reciprocal.
    \end{exm}

    \begin{exm}
    We construct a lower semicontinuous weight $\omega \colon \mathbb{R}^2 \to [0, \infty]$ that vanishes on a Cantor set $E \subset [0,1] \times \{0\}$ of positive length such that the space $(\mathbb{R}^2, d_\omega)$ is reciprocal.
    
    As before, identify $[0,1]$ with the set $[0,1] \times \{0\} \subset \mathbb{R}^2$. Consider the quadrilateral $Q = [0,1] \times [-1,1]$. Let $\Gamma$ be the family of paths in $Q$ connecting the left and right edges of $Q$. 
    
    Fix for the time being a value $t \in (0,1/2)$. Let $I=[t,1-t] \subset (0,1)$ and let $\omega_1 = \chi_{\mathbb{R}^2 \setminus I}$, noting that $\omega_1$ vanishes on the set $I$. Let $\mathcal{E}_1$ denote the decomposition of $\mathbb{R}^2$ corresponding to $I$. The weight $\omega_1$ determines a metric $\widetilde{d}_{\omega_1}$ on $\mathbb{R}^2/\mathcal{E}_1$ that is not reciprocal. Let $\pi_{\omega_1}$ denote the associated quotient map. Note that the metric $\widetilde{d}_{\omega_1}$, like all other metrics in this example, agrees with the Euclidean metric locally outside of $\pi_{\omega_1}(I)$. Thus the Hausdorff 2-measure relative to the metric $\widetilde{d}_{\omega_1}$ coincides with Lebesgue measure.
    
    Let $\widetilde{ \rho }$ be an admissible function for $\pi_{ \omega_{1} } \Gamma$ with respect to the metric $\widetilde{d}_{ \omega_{1} }$ satisfying $\int \widetilde{ \rho }^2\, d\mathcal{H}_{d_{\widetilde{\omega}}}^2 \leq 2 \Mod \pi_{\omega_1} \Gamma$. Since the function 
    \[
        \widetilde{g}
        =
        \frac{
            \chi_{ [0, t) \times \left[-1, 1\right] }
            +
            \chi_{ (1-t, 1] \times \left[-1, 1\right] }
        }
        {2t}
    \] 
    is admissible for $\pi_{\omega_1} \Gamma$, it follows that
    \begin{equation}\label{equ:exm_2_upper_bound:prelim}
        \int \widetilde{ \rho }^2\, d\mathcal{H}_{ \widetilde{d}_{ \omega_{1} } }^{2}
        \leq
        2\int \widetilde{g}^2\,d\mathcal{H}_{ \widetilde{d}_{ \omega_{1} } }^{2}
        =
        \frac{2}{t}.
    \end{equation}
    Let $\rho = \chi_{Q} + \widetilde{\rho}\circ \pi_{\omega_1}$.
    
    For all $n \in \mathbb{N}$, $i \in \{1, \ldots, n\}$, let $\varphi_i^n$ denote the similarity mapping of $\mathbb{R}^2$ taking $I$ to the interval $[(i-1+t)/n, (i-t)/n ]$. Explicitly, $\varphi_i^n(x) = x/n+ ((i-1)/n,0)$. Let $I_n = \bigcup_{i=1}^n \varphi_i^n(I)$ and let $\mathcal{E}_n$ denote the corresponding decomposition of $\mathbb{R}^2$. Let $\omega_n = \chi_{\mathbb{R}^2 \setminus I_n}$, and $\widetilde{d}_{\omega_n}$ the resulting metric on $\mathbb{R}^2 / \mathcal{E}_n$.

    Let $\rho_i^n = \rho \circ (\varphi_i^n)^{-1}$. Define now the function $\rho_n\colon Q \to [0, \infty]$ by
    \[
        \rho_n(x)
        =
        \left\{ 
        \begin{array}{ll}
            \rho_i^n(x) 
            & \text{ if } x \in \varphi_i^n((0,1) \times [-1,1]) \text{ for some } i \in \{1, \ldots, n\}
            \\
            1
            & \text{ otherwise }
        \end{array} 
        \right. .
    \]
    For all $x \in \pi_{ \omega_{n} }( Q )$, we define $\widetilde{\rho}_{n}(x) = \rho_{n}( \pi_{ \omega_{n} }^{-1}( x ) )$. We claim that $\widetilde{\rho}_n$ is admissible for $\pi_{ \omega_{n} }\Gamma$ with respect to the metric $\widetilde{d}_{\omega_n}$.
    
    Let $Q_i^n = [(i-1)/n,i/n]\times [-1,1]$, and let $\gamma_i^n$ be a subpath of $\gamma$ that traverses $Q_i^n$ horizontally. It suffices to show that $\int_{\gamma_i^n} \rho_n \omega_{n} \,ds_{ \norm{ \cdot }_{2} } \geq 1/n$. If $\gamma_i^n$ does not intersect $I_n$, then this is clear since $\rho_i^n \geq 1$ on $Q_i^n \setminus I_n$. If $\gamma_i^n$ is contained in $\varphi_i^n(Q)$, then this is also immediate by the admissibility of $\widetilde{\rho}$. Finally, if $\gamma_i^n$ intersects both $I_n$ and $Q_i^n \setminus \varphi_i^n(Q)$, then $\gamma_i^n$ must traverse a vertical distance of $1/n$, and again the conclusion follows. We conclude that $\rho_n$ is admissible for $\Gamma$. 
    
    Next, we have the upper bound 
    \begin{align} \label{equ:exm_2_upper_bound}
        \int_Q \rho_n^2\, d\mathcal{L}^2 &
        \leq
        \int_Q 1\, d\mathcal{L}^2
        +
        \sum_{i=1}^n \int_{Q_i^n} (\rho_i^n)^2 \, d\mathcal{L}^2
        \leq
        2 + \frac{ \norm{ \rho }_{ L^{2}(Q) }^{2} }{ n }.
    \end{align}
    Observe that $2 = \Mod_{\|\cdot\|_2} \Gamma$. Thus, by taking $n$ to be sufficiently large, the modulus of $\Gamma$ with respect to $\widetilde{d}_{\omega_n}$ becomes arbitrarily close to the Euclidean modulus. 
    
    We can now define the Cantor set $E$ as follows. For a given $t \in (0,1/2)$ and $n \in \mathbb{N}$, let $I(t)$, $I_n(t)$ and $\omega_n(t)$ denote respectively the sets $I$ and $I_n$ and the weight $\omega_n$ constructed above. For all $j \in \mathbb{N}$, let $t_j = 2^{-j-2}$, observing that $\mathcal{L}^1(I(t_j)) = 1 - 2t_j$. Let $\widetilde{\omega}_{j} = \omega_{n_j}(t_j)$. By choosing $n_j$ sufficiently large, we can guarantee that
    \[
        \mathcal{L}^1(I_{n_j}(t_j) \cap I_{n_{j-1}}(t_{j-1}))
        \geq
        (1-4t_{j}) \mathcal{L}^1(I_{n_{j-1}}(t_{j-1}))
    \]
    and that $\Mod \pi_{ \widetilde{\omega}_{j} } \Gamma \leq 2 + 1/j$ by applying \eqref{equ:exm_2_upper_bound:prelim} and \eqref{equ:exm_2_upper_bound}. Inductively choosing $n_j$ in this manner, we have
    \[
        \mathcal{L}^1\left(\bigcap_{i=1}^j I_{n_i}(t_i)\right)
        \geq
        \prod_{i=1}^j (1-4t_i) = \prod_{i=1}^j (1-2^{-i}).
    \]
    Let $E = \bigcap_{j=1}^\infty I_{n_j}(t_j)$ and let $\omega = \chi_{\mathbb{R}^2 \setminus E}$. Then $\mathcal{L}^1(E) = \prod_{j=1}^\infty (1-2^{-j}) > 0$. Moreover, $\omega \geq \widetilde{\omega}_{j}$ for all $j \in \mathbb{N}$. This fact, combined with \Cref{prop:williams:L-Wversion}, yields that $2 \leq \Mod \pi_{ \omega } \Gamma \leq \Mod \pi_{ \widetilde{ \omega }_{j} } \Gamma$ for all $j \in \mathbb{N}$. We conclude that $\Mod \pi_{\omega} \Gamma = 2 = \Mod_{\|\cdot\|_2} \Gamma$.

    Let $\Gamma^*$ denote the family of paths connecting the bottom and top edges of $Q$. It is clear that the function $\rho^* = 1/2 \chi_{Q}$ is admissible for $\Gamma^*$ with respect to the metric $\widetilde{d}_\omega$. Thus $\Mod_{\widetilde{d}_\omega} \Gamma^* = 1/2 = \Mod_{\|\cdot\|_2} \Gamma^*$. By \Cref{lemm:reciprocal_quadrilaterals_mod}, this suffices to show that $\widetilde{d}_\omega$ is reciprocal. 
    \end{exm}

\section{Factoring quasiconformal mappings} \label{sec:example}

    The goal of this section is to prove \Cref{prop:pi:isothermal} and \Cref{thm:example}. To prepare for this, we first give in \Cref{sec:isothermal} an overview of isothermal quasiconformal mappings. See \cite{Iko:19} for a more complete treatment. \Cref{sec:positive_answer} gives the proof of \Cref{prop:pi:isothermal}. This is followed by a discussion in \Cref{sec:positive_answer:regular} of the problem of optimizing the distortion constant in \Cref{prop:pi:isothermal}. Finally, in \Cref{sec:proof:factorization:bad}, we prove \Cref{thm:example}.

\subsection{Isothermal Parametrizations}\label{sec:isothermal}
     
    Let $X$ be a quasiconformal surface. By Theorem 6.2 in \cite{Iko:19}, there exists a complete Riemannian surface $Y$ of constant curvature and a quasiconformal map
    \begin{equation*}
        \psi
        \colon
        Y
        \rightarrow
        X
    \end{equation*}
    with minimal pointwise distortion at almost every point: for every other Riemannian surface $Z$ and quasiconformal map $\varphi \colon Z \rightarrow X$, the inequality
    \begin{equation}
    \label{eq:minimization}
        \left(
            g_{ \psi }
            (g_{ \psi^{-1} } \circ \psi)
        \right) \circ ( \psi^{-1} \circ \varphi )
        \leq
        g_{ \varphi }
        (g_{ \varphi^{-1} } \circ \varphi)
    \end{equation}
    holds $\mathcal{H}^{2}_{Z}$-almost everywhere on $Z$. Recall that, for example, $g_{ \psi }$ and $g_{ \psi^{-1} } $ refer to minimal weak upper gradients of $\psi$ and $\psi^{-1}$, respectively. In this case, we say that $( Y,  \psi )$ is an \emph{isothermal parametrization} of $X$. By Corollary 4.7 of \cite{Iko:19}, any isothermal parametrization $\psi$ is quasiconformal with outer dilatation $K_O(\psi)$ at most $4/\pi$ and inner dilatation $K_I(\psi)$ at most $\pi/2$. Also, the pointwise distortion of $\psi$ is bounded from above by $\sqrt{2}$ $\mathcal{H}^{2}_{Y}$-almost everywhere.
    
    We elaborate on the meaning of \eqref{eq:minimization} in the case when $X = ( \mathbb{R}^{2},  d_N )$ for some norm $N$. Then we can take $Y = \mathbb{R}^{2}$ and $\psi$ to be a linear map
    \begin{equation*}
        \psi
        \colon
        \mathbb{R}^{2}
        \rightarrow
        ( \mathbb{R}^{2},  d_N )
    \end{equation*}
    such that $g_{ \psi } = L( N \circ \psi )$ and $g_{ \psi^{-1} } = \omega( N \circ \psi )^{-1}$. Recall that $L$ and $\omega$ denote, respectively, the maximal and minimal stretching, defined in \eqref{eq:upper_gradient} and \eqref{eq:minimal_gradient}.
    
    The inequality \eqref{eq:minimization} implies that, for all other linear maps $\varphi \colon \mathbb{R}^{2} \rightarrow ( \mathbb{R}^{2},  N )$, we have
    \begin{equation}
        \label{eq:isothermal:norms}
        \frac{ L( N \circ \psi ) }{ \omega( N \circ \psi ) }
        \leq
        \frac{ L( N \circ \varphi ) }{ \omega( N \circ \varphi ) }.
    \end{equation}
    In terms of the \emph{distortion} of a norm defined in \eqref{eq:distortion}, the inequality \eqref{eq:isothermal:norms} implies that $N\circ \psi$ has the smallest possible distortion among such pairs $\psi$ and $\varphi$. This can be phrased in terms of the \emph{Banach--Mazur distance} in convex geometry; see \cite{Rom:19} and \cite[Section 4]{Iko:19}. 
    
    An isothermal parametrization of a quasiconformal surface is essentially unique. This is also part of the content of Theorem 6.2 of \cite{Iko:19}, partially quoted here.
    \begin{thm}[\cite{Iko:19}]\label{thm:iso:unique}
    Let $\psi \colon Y \rightarrow X$ be an isothermal parametrization of $X$, and $\varphi \colon Z \rightarrow X$ a quasiconformal map from a Riemannian surface $Z$ onto $X$. Then $\varphi$ is isothermal if and only if $\psi^{-1} \circ \varphi$ is a conformal diffeomorphism.
    \end{thm}
    
    Let $N$ be an admissible reciprocal norm field on $\mathbb{R}^{2}$ that vanishes on the compact set $E \subset \mathbb{R}^{2}$. The following lemma is a consequence of Theorem 4.16 of \cite{Iko:19}.
    \begin{lemm}\label{lemm:isothermal:admissible}
    The identity map $\iota \colon \mathbb{R}^{2}   \rightarrow ( \mathbb{R}^{2}, d_{N} )$ is isothermal if and only if
    \begin{equation}
        \label{eq:lemm:isothermal:admissible}
        \frac{ L( N_{x} ) }{ \omega( N_{x} ) } \leq
        \frac{ L( N_{x} \circ \varphi ) }{ \omega( N_{x} \circ \varphi ) } 
    \end{equation}
    for all $\varphi \in \mathrm{GL}_2$, for $\mathcal{L}^{2}$-almost every $x \in \mathbb{R}^{2}$.
    \end{lemm}
    Observe that \eqref{eq:lemm:isothermal:admissible} is satisfied by the norm $N_x = \norm{\cdot}_\infty$, and more generally by any norm $N_x$ whose unit ball is a square. Thus \Cref{lemm:isothermal:admissible} has the following corollary.
    \begin{cor}\label{cor:isothermal:supremum}
    Suppose that $N$ is reciprocal and that the unit ball of $N_{x}$ is a square for $\mathcal{L}^{2}$-almost every $x \in \mathbb{R}^{2}$. Then the identity map
        $\iota
        \colon
        \mathbb{R}^{2}
        \rightarrow
        ( \mathbb{R}^{2}, d_{N} )$
    is isothermal.
    \end{cor}

\subsection{Proof of \Cref{prop:pi:isothermal}}\label{sec:positive_answer}

    Recall that we are assuming that $N$ is a reciprocal norm field such that $\pi_{N} \colon \Omega \rightarrow ( \Omega, d_{N} )$ is isothermal, and that $N$ is continuous outside the set $E = \left\{x \in \Omega: N_{x} = 0 \right\}$.
    
    Let $G$ be a complete Riemannian norm field on $\Omega$ of constant Gaussian curvature $-1$ or $0$, which exists by the classical uniformization theorem. The field is of the form $G = \omega \norm{ \cdot }_{2}$ for some smooth $\omega$. Consider the norm field
    \begin{equation*}
        M
        =
            \chi_{ \Omega \setminus E }
            \frac{ \omega }{ \omega( N ) }
            N
            +
            \chi_{ E }
            G.
    \end{equation*}
    The function $1/\omega( N )$ is continuous in $\Omega \setminus E$ due to the continuity of $N$ outside $E$. Then the distortion bound on $N$ implies that $M$ is a lower semicontinuous norm field satisfying $G \leq M \leq H G$ everywhere.
    
    Let $\widehat{d} = d_{M}$ denote the distance induced by $M$. Then
    \begin{equation*}
        d_{G}
        \leq
        \widehat{d}
        \leq
        Hd_{G},
    \end{equation*}
    so the identity map $P = \pi_{M} \colon ( \Omega, d_{G} ) \rightarrow ( \Omega, \widehat{d} )$ satisfies \eqref{equ:P_bi_Lipschitz} and in particular is $H$-bi-Lipschitz. \Cref{lemm:derivative_standard} states that its metric differential coincides with $M$ $\mathcal{L}^{2}$-almost everywhere.
    
    The proof is complete after we show that $\widehat{\iota} = \pi_{N} \circ P^{-1}$ is $1$-quasiconformal. Recall that the metric derivatives of $\pi_{N}$ and $P$ coincide with $N$ and $M$, respectively. Then, as a consequence of Corollary 5.15 and Proposition 5.12 of \cite{Iko:19}, the $1$-quasiconformality is equivalent to proving that for $\mathcal{L}^{2}$-almost every $x \in \Omega$, the distortion of the identity map from $( \mathbb{R}^{2}, M_{x} )$ to $( \mathbb{R}^{2}, N_{x} )$ equals one $\mathcal{L}^{2}$-almost everywhere.
    
    Observe that, by the change of variables formula \Cref{lemm:metric_derivative_Jacobian} and the Lusin's Condition ($N^{-1}$) of $\pi_{N}$, the set $E$ has zero $\mathcal{L}^{2}$-measure, so we only need to check the pointwise distortion in the complement of $E$. Here the claim is immediate, since $M_{z} = \omega(z)N_{z}/\omega( N_{ z } )$ for every $z \in \Omega \setminus E$. We conclude that $\widehat{ \iota }$ is $1$-quasiconformal.
    
    \subsection{Remarks on optimal distortion}\label{sec:positive_answer:regular}
    
    We discuss the question of when the optimal constant $H = \sqrt{2}$ in \eqref{equ:P_bi_Lipschitz} in \Cref{prop:pi:isothermal} can be achieved. We recall that any planar quasiconformal mapping $f\colon \Omega \to \widehat{\Omega}$ is a solution of the \emph{Beltrami equation} $f_{\overline{z}} = \mu f_z$, where $\mu\colon \Omega \to \mathbb{C}$ is a measurable function satisfying $\norm{\mu}_\infty < 1$. Conversely, the \emph{measurable Riemann mapping theorem} provides a homeomorphic solution to the Beltrami equation for any such $\mu$. The function $\mu$ is called the \emph{Beltrami coefficient}. Geometrically, the choice of a Beltrami coefficient corresponds to the choice of a measurable ellipse field on $\Omega$ modulo rescaling of the ellipses. See Chapter 5 of \cite{Ast:Iwa:Mar:09} for an in-depth overview of the topic. 
    
    Given a norm field $\widehat{N}$ on a domain $\widehat{\Omega} \subset \mathbb{R}^2$, one obtains an ellipse field on $\widehat{\Omega}$ by associating to each norm $\widehat{N}_x$ its distance ellipse, that is, the unique ellipse $\mathcal{E} \subset B_{\widehat{N}_x}(0,1)$ having minimal $\lambda \geq 1$ such that $B_{\widehat{N}_x}(0,1) \subset \lambda \mathcal{E}$. This in turn gives a Beltrami coefficient $\mu_{\widehat{N}}$ corresponding to $\widehat{N}$. We refer the reader to \cite[Sec. 4]{Iko:19} for more details.
    
    This choice of ellipse field also determines an underlying Riemannian structure on the metric space $(\widehat{\Omega}, d_{\widehat{N}})$. A consequence of the classical slit domain uniformization theorem \cite[Chapter III, Section 4]{Ah:Sa:60} and \cite[Corollary 6.3]{Iko:19} is the existence of a domain $\Omega \subset \mathbb{R}^2$ and a locally quasiconformal map $\psi \colon \Omega \to \widehat{ \Omega }$ such that $\widehat{ f } = \pi_{ \widehat{N} } \circ \psi$ is isothermal. Consider the distance $d( x, y ) = d_{ \widehat{N} }( \widehat{f}(x), \widehat{f}(y) )$ on $\Omega$ and the norm field $N = \widehat{N} \circ D\psi$. Then the identity map $\iota \colon \Omega \to ( \Omega, d )$ is isothermal and the metric differential of $\iota$ exists and equals $N$ $\mathcal{L}^{2}$-almost everywhere. If the norm field $N$ obtained in this manner is continuous and non-zero outside $E = \psi^{-1}\left( \{ x \in \widehat{ \Omega }: \widehat{N}_x = 0\} \right)$, then \Cref{prop:pi:isothermal} now holds with constant $H = \sqrt{2}$ for the space $(\Omega, d)$ and hence the original space $(\widehat{\Omega}, d_{\widehat{N}})$ as well.
    
    The question of when the norm field $N$ is continuous, in turn, depends upon the regularity of the map $\psi$. In fact, if $\psi$ is $\mathcal{C}^{1}$-continuous in $\Omega$, then $N$ is continuous and non-zero outside $E$. Since the map $\psi$ arises as a solution to the Beltrami equation, this leads to the question of regularity of solutions to the Beltrami equation. Indeed, if we consider a domain $U$ compactly contained in $\widehat{ \Omega }$, the restriction of $\psi^{-1}$ to $U$ solves the Beltrami equation induced by $\mu_{ \widehat{N} }|_{ U }$. The $\mathcal{C}^{1}$-continuity of $\psi$ in $U$ is known to hold, for example, when $\mu_{ \widehat{N} }|_{ U }$ is $\mathcal{C}^{1}$-continuous, locally Hölder continuous \cite[Theorem 15.0.7]{Ast:Iwa:Mar:09} or in $W^{1,p}_{ \loc }( U )$ for a large enough $p > 1$ depending on the $L^{\infty}$-norm of $\mu_{ \widehat{N} }|_{U}$ \cite[Proposition 4]{Ba:Cl:Or:19}.

    Solutions of the Beltrami equation for $\mu_{ \widehat{N} }$, even when $\widehat{N}$ is a continuous Riemannian norm field, need not always be $\mathcal{C}^1$-continuous.  In the following, we use complex notation $z = z_1 + iz_2$ to denote the point $(z_{1}, z_{2}) \in \mathbb{R}^2$ and $\overline{z} = z_{1} - iz_{2}$ to denote the complex conjugate of $z$. See Section 2.4 of \cite{Ast:Iwa:Mar:09} for a brief overview of complex notation. The following example is based on Section 15.1 of \cite{Ast:Iwa:Mar:09}. Let
    \[\mu(z) = \frac{ z }{ \overline{z}( 1 + \log \norm{z}^{2}_{2} ) }\]
    and consider the continuous Riemannian norm field $\widehat{N}$ on $\widehat{ \Omega } = B_{ \norm{ \cdot }_{2} }( 0, e^{-1/2} )$ defined by $\widehat{N}_{z}(v) = \norm{ v + \mu(z) \overline{v} }_{2}$. Then $\mu(z) = \mu_{\widehat{N}_z}$, where $\mu_{\widehat{N}_z}$ is the Beltrami coefficient corresponding to the $\widehat{N}$ as described earlier in this remark. Even though $\widehat{N}$ is continuous, every solution for the Beltrami equation for $\mu_{ \widehat{N}_{z} } = \mu(z)$ has a discontinuous derivative at the origin. This is seen by considering the particular solution $g(z) = -z \log \norm{z}_{2}^{2}$ and noticing that the differential $Dg$ is discontinuous at the origin. It is enough to check this property for $g$ since, by the Stoilow factorization theorem \cite[Theorem 5.5.1]{Ast:Iwa:Mar:09}, every other quasiconformal solution is of the form $\Psi \circ g$ for some conformal diffeomorphism $\Psi$.
    
    For more general norm fields, we have the additional complexity that $\mu_{ \widehat{N} }$ can be smooth even though $\widehat{N}$ is not. For example, consider the continuous norm field $\widehat{N}$ defined by $\widehat{N}_{z}(v) = \norm{ e^{ i \norm{ z }_{2} } v }_{\infty}$. Since the supremum norm is not $\mathcal{C}^{1}$-continuous in $\mathbb{R}^{2} \setminus \left\{0\right\}$ we see that $\widehat{N}$ is not $\mathcal{C}^{1}$-continuous, for example by considering the basepoint $z = \pi/4$ and the vector $v = 1$, even though $\mu_{ \widehat{N} } = 0$. The identity $\mu_{ \widehat{N} } = 0$ follows from \Cref{cor:isothermal:supremum}.
    
\subsection{Proof of \Cref{thm:example}}\label{sec:proof:factorization:bad}
    
    In this section, we present the construction used to prove \Cref{thm:example}, namely of a quasiconformal surface whose isothermal parametrization cannot be factored as a bi-Lipschitz mapping postcomposed with a quasiconformal mapping of smaller distortion. We begin by introducing the notation and parameters involved in \Cref{sec:construction_notation}. We develop various properties of this construction in the following subsections, culminating with the proof of \Cref{thm:example} in \Cref{sec:main_theorem_proof}.
    
\subsubsection{Notation} \label{sec:construction_notation}

	Let us introduce the notation used in our construction. Our first task is to construct a sequence of nested Cantor sets, denoted by $K_1, K_2, \ldots$ and satisfying $K_1 \supset K_2 \supset \cdots$. There are two intermediate steps used in our construction of the sets $K_i$. First, we define sets $E_i^j$ for all $i,j \in \mathbb{N}$, $j \geq i$, to serve as base collections of squares from which the Cantor sets are extracted. The sets $E_i^j$ are each the union of a collection of congruent closed squares $Q_i^j(k,l)$ that covers almost all of $[0,1]^2$. The main feature of our construction is that the squares $Q_i^j(k,l)$  have the standard non-rotated alignment for odd values of $i$, while the square $Q_i^j(k,l)$ are aligned diagonally for even values of $i$. 
	
	In the second intermediate step, we define inductively \[F_i^j = E_i^j \cap F_{i-1}^j \cap F_i^{j-1} \cap [0,1]^2\]
	for all $i,j \in \mathbb{N}$, $j \geq i$, with the convention that $F_0^j = F_i^{i-1} = [0,1]^2$ for all $i,j$. By taking $F_i = \bigcap_j F_i^j$, we obtain a collection of nested Cantor sets. However, for our construction to work, we need the further property that the intersection of the Cantor sets is small. For this reason, we later define $K_i$ to be a subset of $F_i$ with the property that $\diam K_i \to 0$ as $i \to \infty$.  
	
	We take a moment to fix some additional notation. In the following, let $I = J = [0,1]$ and let $Q = I \times J = [0,1]^2$. We identify $I$ with the set $[0,1] \times \{0\}$ and $J$ with the set $\{0\} \times [0,1]$. Let $\pi_1$ denote the standard projection map from $\mathbb{R}^2$ onto the first coordinate axis, and let $\pi_2$ denote the standard projection map from $\mathbb{R}^2$ onto the second coordinate axis. 
	
	As mentioned above, the even-numbered Cantor sets are formed from squares that are rotated by $\pi/4$ from the standard alignment. Let $Q^*$ denote the square with vertices $(1/2,-1/2)$, $(3/2,1/2)$, $(1/2,3/2)$, and $(-1/2,1/2)$. Let $I^* = J^* = [0,\sqrt{2}]$. We also identify $I^*$ with the set $[0, \sqrt{2}] \times\{0\}$ and $J^*$ with the set $\{0\} \times [0, \sqrt{2}]$. Let $\varphi\colon \mathbb{R}^2 \to \mathbb{R}^2$ be the orientation-preserving isometry that maps $[0,\sqrt{2}]^2$ onto $Q^*$ and satisfies $\varphi(0,0) = (1/2,-1/2)$. Explicitly,
	\[\varphi(x,y) = (1/2, -1/2) + \frac{1}{\sqrt{2}}(x-y,x+y). \] 
	Thus $\varphi(I^* \times J^*) = Q^*$.  Next, let $\pi_1^*$ denote the projection map from $Q^*$ onto $\varphi(I^*)$, and let $\pi_2^*$ denote the projection map from $Q^*$ onto $\varphi(J^*)$. Explicitly, $\pi_1^*(x,y) = \varphi(\pi_1(\varphi^{-1}(x,y)),0)$ and $\pi_2^*(x,y) = \varphi(0,\pi_2(\varphi^{-1}(x,y)))$.
	
	\begin{figure}[t] 
        \centering
    	\begin{tikzpicture}[>=latex, scale=5]
    	\draw[thick] (0,0) -- (1,0) -- (1,1) -- (0,1) -- (0,0);
    	\draw[thick] (-.167,-.167) -- (1.167,-.167) -- (1.167,1.167) -- (-.167, 1.167) -- (-.167,-.167);
    	\foreach \i in {1,2,3} { 
    	\foreach \j in {1} {
    	\draw[thick,fill=lightgray] (-.167+.333*\i-.14,-.333+.333*\j) -- (-.167+.333*\i,-.333+.333*\j+.14) -- (-.167+.333*\i+.14,-.333+.333*\j);
    	} }
    	\foreach \i in {1,2,3} { 
    	\foreach \j in {2,3} {
    	\draw[thick,fill=lightgray] (-.167+.333*\i-.14,-.333+.333*\j) -- (-.167+.333*\i,-.333+.333*\j+.14) -- (-.167+.333*\i+.14,-.333+.333*\j) -- 
    	(-.167+.333*\i,-.333+.333*\j-.14) --
    	(-.167+.333*\i-.14,-.333+.333*\j);
    	} }
    	\foreach \i in {1,2,3} { 
    	\foreach \j in {1,2,3} {
    	\draw[thick,fill=lightgray] (-.167+.333*\i-.14,.333*\j) -- (-.167+.333*\i,.333*\j-.14) -- (-.167+.333*\i+.14,.333*\j);
    	} }
    	\foreach \i in {3} { 
    	\foreach \j in {1,2,3} {
    	\draw[thick,fill=lightgray] (.333*\i,-.167+.333*\j+.14) -- (.333*\i-.14,-.167+.333*\j) -- (.333*\i,-.167+.333*\j-.14);
    	} }
    	\foreach \i in {1,2} { 
    	\foreach \j in {1,2,3} {
    	\draw[thick,fill=lightgray] (.333*\i,-.167+.333*\j+.14) -- (.333*\i-.14,-.167+.333*\j) -- (.333*\i,-.167+.333*\j-.14) --
    	(.333*\i+.14,-.167+.333*\j) --
    	(.333*\i,-.167+.333*\j+.14);
    	} }
    	\foreach \i in {0,1,2} { 
    	\foreach \j in {1,2,3} {
    	\draw[thick,fill=lightgray] (.333*\i,-.167+.333*\j+.14) -- (.333*\i+.14,-.167+.333*\j) -- (.333*\i,-.167+.333*\j-.14);
    	} }
    	\draw [decorate,decoration={brace,amplitude=10pt},xshift=-1pt,yshift=0pt]
(-.167,-.167) -- (-.167,1.167) node [black,midway,xshift=-18pt]
{\large $\frac{1}{N_i^j}$};
\draw [decorate,decoration={brace,amplitude=10pt},xshift=1pt,yshift=0pt]
(1.167,1.167) -- (1.167,1.003) node [black,midway,xshift=22pt]
{\large $\frac{1}{a_i^jN_i^j}$};
\draw [decorate,decoration={brace,amplitude=10pt},xshift=1pt,yshift=0pt]
(1.167,.996) -- (1.167,0) node [black,midway,xshift=44pt]
{\large $\left(1-\frac{2}{a_i^j} \right)\frac{1}{N_i^j}$};
    	\end{tikzpicture}
    	\caption{A square $Q_i^j(k,l) \subset E_i^j$ for $i$ odd and the intersection $Q_i^j(k,l) \cap E_{i+1}^j$, shaded gray. The large outer square is $I_i^j(k) \times J_i^j(l)$.}
    	\label{fig:shaded_area_outer}
	\end{figure}
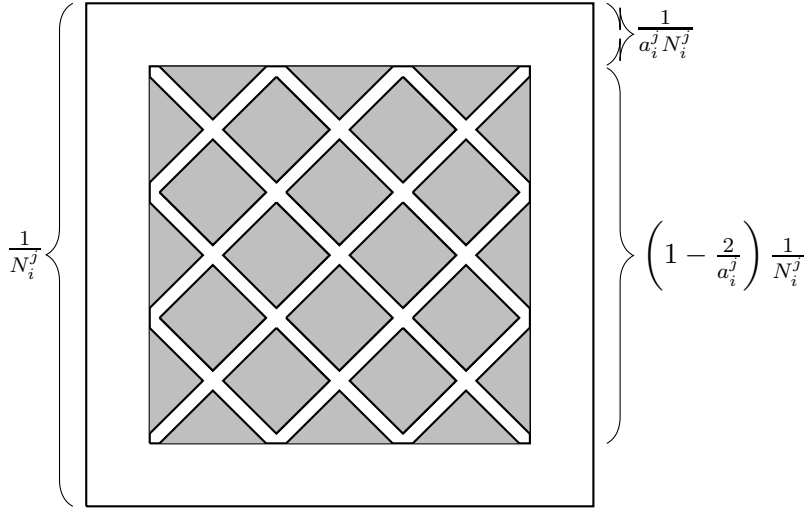

	

    The definition of the sets $E_i^j$ involves three sets of parameters: $\varepsilon_i^j >0$, $N_i^j \in \mathbb{N}$, and $a_i^j \in \mathbb{N}$. A short explanation of these parameters is the following. The first parameter $\varepsilon_i^j$ gives an upper bound on the proportion of area lost when passing from one step of the construction to the next. The second parameter $N_i^j$ gives the number of subdivisions of the initial interval $I$ or $I^*$ that are made when forming the squares that comprise $E_i^j$. The final parameter $a_i^j$ corresponds to the side length of these squares. The precise relation is that the side length of a square in $E_i^j$ is $(1-2(a_i^j)^{-1})/N_i^j$ for $i$ odd and $\sqrt{2}(1-2(a_i^j)^{-1})/N_i^j$ for $i$ even.

    \subsubsection{Constructing the sets $E_i^j$} \label{sec:constructing_E}
    
    For two pairs of indices $(i,j)$ and $(i',j')$, we say that $(i,j) \preceq (i',j')$ if $j<j'$ or if $j=j'$ and $i \leq i'$. The relation $\preceq$ gives an ordering on the set of indices $(i,j)$. We consider the sets $E_i^j$ as being traversed in this order. We also write $(i, j) \prec (i', j')$ if $j < j'$ or if $j = j'$ and $i < i'$. Recall that here and throughout this proof we consider only those indices $i,j \in \mathbb{N}$ for which $j \geq i$. This ordering is illustrated in \Cref{fig:ordering}.
    
    We first choose the parameters $\varepsilon_i^j>0$ so that they satisfy $\prod_{i,j} (1-\varepsilon_i^j) \geq 1/2$. The factors in the product are traversed according to the ordering on $\{(i,j)\}$ just defined.
    
    The sets $E_i^j$ are defined for all $i,j \in \mathbb{N}$ satisfying $j \geq i$ in the following way. Assume for the moment that we have made suitable choices of $N_i^j, a_i^j \in \mathbb{N}$. In the case that $i$ is odd, we divide $I$ into $N_i^j$ equal subintervals $I_i^j(k) = [(k-1)/N_i^j,k/N_i^j]$ and $J$ into $N_i^j$ equal subintervals $J_i^j(l) = [(l-1)/N_i^j,l/N_i^j]$. If $i$ is even, we divide $I^*$ into $N_i^j$ equal subintervals $I_i^j(k) = [\sqrt{2}(k-1)/N_i^j,\sqrt{2}k/N_i^j]$ and $J^*$ into $N_i^j$ equal subintervals $J_i^j(l) = [\sqrt{2}(l-1)/N_i^j,\sqrt{2}l/N_i^j]$. This yields a collection of squares $I_i^j(k) \times J_i^j(l)$, where $k,l \in \{1,\ldots, N_i^j\}$. If $i$ is odd, let $Q_i^j(k,l)$ be the square of side length $(1-2(a_i^j)^{-1})/N_i^j$ with the same center and alignment as $I_i^j(k) \times J_i^j(l)$. If $i$ is even, define $Q_i^j(k,l)$ to be the square of side length $\sqrt{2}(1-2(a_i^j)^{-1})/N_i^j$ with the same center and alignment as $\varphi(I_i^j(k) \times J_i^j(l))$. In this case, the square $Q_i^j(k,l)$ is contained in $Q^*$ and is aligned diagonally. Let 
    \[E_i^j = \bigcup_{k,l} Q_i^j(k,l).\]
    
    The square $Q_i^j(k,l)$, for $i$ odd, is given explicitly by
    \begin{equation*}
        \left[\left(k-1+ \frac{1}{a_i^j}\right)\frac{1}{N_i^j},\left(k- \frac{1}{a_i^j}\right)\frac{1}{N_i^j}  \right] \times \left[\left(l- 1+\frac{1}{a_i^j}\right)\frac{1}{N_i^j},\left(l - \frac{1}{a_i^j}\right)\frac{1}{N_i^j}  \right].
    \end{equation*}
    Let $v_i^j(k,l,1), v_i^j(k,l,2), v_i^j(k,l,3), v_i^j(k,l,4)$ denote the four vertices of $Q_i^j(k,l)$, traversed counterclockwise from the bottom left. Let $w_i^j(k,l)$ denote the center point of $Q_i^j(k,l)$.
    
    \begin{figure}[t]
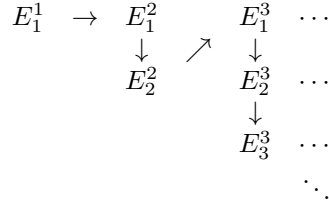
 
	  \begin{tabular}{ccccccc}
	       $E_1^1$ & $\rightarrow$ & $E_1^2$ & & $E_1^3$ &  $\cdots$ &\\
	        & & $\downarrow$ & $\nearrow$ & $\downarrow$ & & \\
	        & & $E_2^2$ & & $E_2^3$ & $\cdots$ &\\
	       & & & & $\downarrow$ \\
	        &  & & & $E_3^3$ & $\cdots$ & \\
	        & & & & & $\ddots$ \\  
	  \end{tabular}
	  \caption{The sets $E_i^j$ as ordered by $\preceq$.} \label{fig:ordering}
	\end{figure}

    Similarly, the square $Q_i^j(k,l)$, for $i$ even, is the image under $\varphi$ of the square
    \begin{equation} \label{equ:even_square}
            \left[
                \left(k-1+ \frac{1}{a_i^j}\right)
                \frac{\sqrt{2}}{N_i^j},
                \left(k- \frac{1}{a_i^j}\right)
                \frac{\sqrt{2}}{N_i^j}
            \right]
            \times
            \left[
                \left(l-1+ \frac{1}{a_i^j}\right)
                \frac{\sqrt{2}}{N_i^j},
                \left(l - \frac{1}{a_i^j}\right)
                \frac{\sqrt{2}}{N_i^j}
            \right].
    \end{equation}
    Let $v_{i}^{j}(k, l, 1), v_{i}^{j}(k, l, 2), v_{i}^{j}(k, l, 3), v_{i}^{j}(k, l, 4)$ denote the four vertices of $Q_i^j(k,l)$, where $v_{i}^{j}(k, l, 1)$ is the image under $\varphi$ of the bottom left vertex of the square in \eqref{equ:even_square} and the rest are labelled in counterclockwise order. Let $w_i^j(k,l)$ denote the center point of $Q_i^j(k,l)$.

    The values of $N_i^j$ and $a_i^j$ are chosen inductively using the ordering $\preceq$. Let $N_1^1 = 2$ and choose $a_1^1 \in \mathbb{N}$ so that $\mathcal{L}^2(E_1^1) \geq 1- \varepsilon_1^1$. For the inductive step, assume that we have chosen $N_{i'}^{j'}$ and $a_{i'}^{j'}$ for some pair $(i',j')$, and that 
    \[ \mathcal{L}^2 \left(\bigcap_{(i'',j'') \preceq (i',j')} E_{i''}^{j''}\right) \geq \prod_{(i'',j'') \preceq (i',j')} (1-\varepsilon_{i''}^{j''}). \]  
    Let $(i,j)$ denote the pair immediately succeeding $(i',j')$. 
    
    Define now $N_i^j = 2a_{i'}^{j'} N_{i'}^{j'}$. We then choose $a_i^j$ so that
    \[
        \mathcal{L}^2\left(\bigcap_{(i'',j'') \preceq (i,j)} E_{i''}^{j''}\right)
        \geq
        \prod_{(i'',j'') \preceq (i,j)} (1-\varepsilon_{i''}^{j''}) .
    \] 
    This can be done because $E_i^j$ can be made to have arbitrarily large area in, respectively, $Q$ or $Q^*$ by making $a_i^j$ sufficiently large.  
    
    We make the following observation. Fix $(i,j)$ and consider a square $Q_i^j(k,l)$. For all $(i',j')$ such that $(i,j) \prec (i',j')$ and $m,n \in \{1, \ldots, N_{i'}^{j'}\}$, the square $I_{i'}^{j'}(m) \times J_{i'}^{j'}(n)$, if $i'$ is odd, or $\varphi(I_{i'}^{j'}(m) \times J_{i'}^{j'}(n))$, if $i'$ is even, is either entirely contained in $Q_i^j(k,l)$, has interior disjoint from $Q_i^j(k,l)$, or intersects $Q_i^j(k,l)$ in a triangle whose vertices are three of the vertices of $I_{i'}^{j'}(m) \times J_{i'}^{j'}(n)$. 
    
    We also observe a uniformity to how the squares are distributed. For each $i,j,k,l$, we divide the square $Q_i^j(k,l)$ into four triangles whose vertices are two adjacent vertices of $Q_i^j(k,l)$ and the midpoint of $Q_i^j(k,l)$. Denote these by $T_i^j(k,l,1)$, $T_i^j(k,l,2)$, $T_i^j(k,l,3)$, $T_i^j(k,l,4)$, where $T_i^j(k,l,m)$ contains the edge 
    $[v_i^j(k,l,m),v_i^j(k,l,m+1)]$, taking $v_i^j(k,l,5) = v_i^j(k,l,1)$.  
    
    \begin{lemm} \label{lemm:congruent_intersection}
    Let $i,j,i',j' \in \mathbb{N}$, where $(i,j) \prec (i',j')$. For all $k,l \in \{1, \ldots, N_i^j\}$ and $m \in \{1,\ldots, 4\}$ satisfying $T_i^j(k,l,m) \subset Q$, the sets $T_i^j(k,l,m) \cap E_{i'}^{j'}$ are all congruent.  
    \end{lemm}
    \begin{proof}
    This proof depends on the property that $2a_i^jN_i^j$ divides $N_{i'}^{j'}$. As a result, squares at different levels of the construction intersect nicely. We consider the case when $i$ is odd. 
    
    First, suppose that $i'$ is also odd. For each $m \in \{1,\ldots ,4\}$, consider the edge $e_i^j(k,l,m)$ as defined above. We have $\pi_1(v_i^j(k,l,1)) = \pi_1(v_i^j(k,l,4)) = k_1(k,l)/N_{i'}^{j'}$ and $\pi_1(v_i^j(k,l,3)) = \pi_1(v_i^j(k,l,2)) = k_3(k,l)/N_{i'}^{j'}$, where 
    \[
        k_1(k,l)
        =
        \frac{
            (a_i^jk-a_i^j+1)N_{i'}^{j'}
        }{
            a_i^jN_i^j
        }
        \, \, \, \text{ and } \, \, \,
        k_3(k,l)
        =
       \frac{
            (a_i^jk-1)N_{i'}^{j'}
        }{
            a_i^jN_i^j
        }.
    \]
    Similarly, we have $\pi_2(v_i^j(k,l,2)) = \pi_2(v_i^j(k,l,1)) = k_2(k,l)/N_{i'}^{j'}$ and $\pi_4(v_i^j(k,l,3)) = \pi_4(v_i^j(k,l,4)) = k_4(k,l)/N_{i'}^{j'}$, where
    \[
        k_2(k,l)
        =
        \frac{
            (a_i^jl-a_i^j + 1)N_{i'}^{j'}
        }{
            a_i^jN_i^j
        }
        \, \, \, \text{ and } \, \, \,
        k_4(k,l)
        = 
        \frac{
            (a_i^jl-1)N_{i'}^{j'}
        }{
            a_i^jN_i^j
        }.
    \]
    Observe that $k_i(k,l) \in \mathbb{N}$ for all $i \in \{1, \ldots, 4\}$. We have then
    \[
        Q_i^j(k,l)
        =
        [k_1(k,l)/N_{i'}^{j'}, k_3(k,l)/N_{i'}^{j'}]
        \times
        [k_2(k,l)/N_{i'}^{j'}, k_4(k,l)/N_{i'}^{j'}].
    \]  
    We conclude from this that the intersection $Q_i^j(k,l) \cap E_{i'}^{j'}$ is precisely the union of the squares
    \[
        \{
            Q_{i'}^{j'}(k',l'):
            k_2(k,l)+1 \leq k' \leq k_4(k,l),
            k_1(k,l)+1 \leq l' \leq k_3(k,l)
        \}.
    \]
    We also observe that
    \[|k_3(k,l) - k_1(k,l)| = |k_4(k,l) - k_2(k,l)| = \frac{(a_i^j-2) N_{i'}^{j'}}{a_i^jN_i^j}.\]
    Thus the sets $Q_i^j(k,l) \cap E_{i'}^{j'}$ are congruent for all $k,l \in \{1, \ldots, N_i^j\}$.
    Moreover, notice that each set $Q_i^j(k,l) \cap E_{i'}^{j'}$ is invariant under rotations by $\pi/4$ about the center point $w_i^j(k,l)$. We conclude from this that the sets $T_i^j(k,l,m) \cap E_{i'}^{j'}$ are all congruent. 
    
    Next, suppose that $i'$ is even. Consider now a triangle $T_i^j(k,l,m)$. 
    The two shorter edges of $T_i^j(k,l,m)$ are the edges of a rectangle $R_i^j(k,l,m)$ of side length $\sqrt{2}(a_i^j-2)/(a_i^jN_i^j)$. To keep the exposition more manageable, we write out the argument only for $T_i^j(k,l,1)$. We compute
    \begin{align*}
        \varphi^{-1}(v_i^j(k,l,1)) & = \left( \frac{(a_i^jk+a_i^jl-2a_i^j+2)\sqrt{2}}{2a_i^jN_i^j}, -\frac{(-a_i^jk+a_i^jl+a_i^jN_i^j)\sqrt{2}}{2a_i^jN_i^j}\right) \\
        \varphi^{-1}(v_i^j(k,l,2)) & = \left( \frac{(a_i^jk+a_i^jl-a_i^j)\sqrt{2}}{2a_i^jN_i^j}, \frac{(-a_i^jk+ a_i^jl -a_i^j+2+a_i^jN_i^j)\sqrt{2}}{2a_i^jN_i^j} \right) .
    \end{align*}
    Comparing this with \eqref{equ:even_square} and using the property that $2a_i^jN_i^j$ divides $N_{i'}^{j'}$, we have
    \[\varphi^{-1}(R_i^j(k,l,m)) = \left[\frac{k_1(k,l)\sqrt{2}}{N_{i'}^{j'}},\frac{k_3(k,l)\sqrt{2}}{N_{i'}^{j'}}\right] \times \left[\frac{k_2(k,l)\sqrt{2}}{N_{i'}^{j'}},\frac{k_4(k,l)\sqrt{2}}{N_{i'}^{j'}}\right] \]
    for some $k_1(k,l), \ldots, k_4(k,l) \in \mathbb{N}$ satisfying
    \[|k_3(k,l) - k_1(k,l)| = |k_4(k,l) - k_2(k,l)| = \frac{(a_i^j-2)N_{i'}^{j'}}{2a_i^jN_i^j}.\]
    The intersection $R_i^j(k,l,1) \cap E_{i'}^{j'}$ is the union of the squares
    \[  \{  Q_{i'}^{j'}(k',l'):
            k_2(k,l)+1 \leq k' \leq k_4(k,l),
            k_1(k,l)+1 \leq l' \leq k_3(k,l)\}.\]
    Thus the sets $R_i^j(k,l,1) \cap E_{i'}^{j'}$ are all congruent, and by symmetry it follows that the sets $T_i^j(k,l,1) \cap E_{i'}^{j'}$ are all congruent as well.
    

    The case when $i$ is even is similar and its proof is omitted.
    \end{proof} 
    

	\subsubsection{Constructing the Cantor sets} \label{sec:Cantor_sets}
	
	For all $i,j$, let $F_0^j = Q$ and $F_i^{i-1} = Q$. Define now \[F_i^j = E_i^j \cap F_{i-1}^j \cap F_i^{j-1}\] 
	for all $j \geq i$. Observe that $\bigcap_{(i'',j'') \preceq (i,j)} E_{i''}^{j''} \subset F_i^j$, so we have $\mathcal{L}^2(F_i^j) \geq 1/2$ for all $i,j$. Next, let $F_i = \bigcap_{j \geq i} F_i^j$. 
	
	Let $K_0 = \mathbb{R}^2$. For each $i \geq 1$, pick inductively a square $Q_i = Q_i^i(j_i,k_i)$ with the property that $Q_i \subset Q_{i-1}$. Let \[K_i = F_i \cap Q_i.\] 
	From Lemma \ref{lemm:congruent_intersection}, it follows that $\mathcal{L}^2(K_i) = \mathcal{L}^2(F_i)/(N_i^i)^2$. Moreover, we have that $\diam K_i \to 0$ as $i \to 0$, and in particular that $\bigcap_i F_i$ is a single point set.

	\subsubsection{Dense networks of paths} \label{sec:network_paths}
	  	
	The following portion of the argument relates to having a ``dense network of paths'' at every stage.
	  
	
	
	
	We define the following subset of $K_i$. If $i$ is even, let 
	\[H_i = K_i \cap \pi_2^{-1}(I \setminus \pi_2(K_{i+1})).\] 
	If $i$ is odd, let 
	\[H_i = K_i \cap (\pi_2^*)^{-1}(I^* \setminus \pi_2^*(K_{i+1})).\] 
	For example, in \Cref{fig:shaded_area_outer} representing the case where $i$ is odd, a point $x \in K_i \cap Q_i^j(k,l)$ belongs to $H_i$ if the line $t \mapsto x + (t,t)$ does not intersect any of the gray boxes.
	
	\begin{lemm} \label{lemm:previous_level_density}
	For every point $x \in K_{i+1}$ and $r>0$, the set $H_i \cap B_{\norm{\cdot}_2}(x,r)$ has positive $\mathcal{L}^2$-measure.
	\end{lemm}
	\begin{proof}
	In the first case, we assume that $i$ is even, and hence that $i+1$ is odd. Let $x \in K_{i+1}$ and $r>0$. Consider a square $Q_{i+1}^j(k,l)$ containing $x$ for some $j$ sufficiently large so that $Q_{i+1}^j(k,l) \subset B(x,r/3)$ and such that $Q_{i+1}^j(k,l) \subset Q_{i+1}$. Pick a horizontal edge $S$ of $Q_{i+1}^j(k,l)$ whose interior is contained in $\interior(Q_{i+1})$.
	
	Consider now the set $E_{i}^{j+1}$. Subdivide $S$ into $(a_{i+1}^j-2)N_{i}^{j+1}/(a_{i+1}^jN_{i+1}^j)$ congruent subintervals. Each subinterval is the diagonal of a square $\varphi(I_i^{j+1}(k')\times J_i^{j+1}(l'))$, with corresponding square $Q_{i}^{j+1}(k',l') \subset E_{i}^{j+1}$. From such a square $Q_{i}^{j+1}(k',l')$, we may extract a triangle $T_{i}^{j+1}(k',l',m')$, as defined prior to the statement of \Cref{lemm:congruent_intersection}, whose interior does not intersect $E_{i+1}^j$. Observe further that $T_{i}^{j+1}(k',l',m') \subset Q_{i}$, so that 
	\[T_{i}^{j+1}(k',l',m') \cap F_{i} = T_{i}^{j+1}(k',l',m') \cap K_{i}.\] 
	As a consequence of \Cref{lemm:congruent_intersection}, we have that
	\[
	    \mathcal{L}^2(T_{i}^{j+1}(k',l',m') \cap F_{i})
	    =
	    \frac{\mathcal{L}^2(F_{i})}{4(N_{i}^{j+1})^2}.
	\] 
	Moreover, $T_{i}^{j+1}(k',l',m')$ lies in the neighborhood of $Q_{i+1}^j(k,l)$ of radius $$1/N_{i}^{j+1} \leq \diam Q_{i+1}^j(k,l) < 2r/3,$$ so $T_{i}^{j+1}(k',l',m') \subset B_{\norm{\cdot}_2}(x,r)$. Also, we have that \[\interior(T_{i}^{j+1}(k',l',m')) \cap F_{i} \subset H_i.\] 
	This verifies the claim.
	
	The case that $i$ is even and $i+1$ is odd is similar, and we omit the details.
	\end{proof} 
	
	\begin{lemm} \label{lemm:quantitative_density}
	Let $x \in \pi_1(K_i)$ be a Lebesgue density point of $\pi_1(K_i)$, where $i \in \mathbb{N}$ is odd. Let $\delta>0$, and let $t_0>0$ be such that 
	\[\frac{\mathcal{L}^1(\pi_1(K_i) \cap (x,x+t))}{t} \geq 1-\delta \]
	for all $t \in (0,t_0)$. Then for all $y \in \pi_1(K_i)$ satisfying $|y-x| < 2\delta t$,
	\[\frac{\mathcal{L}^1(\pi_1(K_i) \cap (y,y+t))}{t} \geq 1-2\delta. \]
	The same result holds with $\pi_2$ instead of $\pi_1$. If $i$ is even, the corresponding result holds for $\pi_1^*$ and $\pi_2^*$, identifying $I^*$ and $J^*$ with the interval $[0,\sqrt{2}]$.
	\end{lemm} 
	\begin{proof}
	The first claim follows from the relationship
	\[ \mathcal{L}^1(\pi_1(K_i) \cap (y,y+t)) \geq \mathcal{L}^1(\pi_1(K_i) \cap (x,x+t)) - |x-y|.\]
	The other claims follow from a similar inequality.
	\end{proof}
	  
	\subsubsection{Defining the metric on $\mathbb{R}^2$} \label{sec:define_metric}
	
%

    Define a norm field $N$ on $\mathbb{R}^2$ by the formula 
	\[
        N_x
        =
        \left\{
        \begin{array}{ll}
            2^{-i/2} \|\cdot\|_1 & \text{ if } x \in K_{i} \setminus K_{i+1}\text{, $i$ even,}
        \\
            2^{-(i-1)/2}\|\cdot\|_\infty & \text{ if } x \in K_{i} \setminus K_{i+1}\text{, $i$ odd.} 
        \end{array}
        \right.
	\]
	The norm field $N$ is admissible in the sense of \Cref{defi:seminorm_admissible}, in particular being lower semicontinuous, and induces a metric $d$ on $\mathbb{R}^2$ as described in \Cref{sec:distances}. Observe that $N$ vanishes at a single point. \Cref{thm:removable_implies_reciprocal} and \Cref{cor:isothermal:supremum} imply the following.
	\begin{prop}\label{prop:qc_equivalence}
	The identity map $\iota \colon ( \mathbb{R}^{2}, \norm{ \cdot }_{2} ) \rightarrow ( \mathbb{R}^{2}, d )$ is an isothermal quasiconformal mapping.
	\end{prop}

	\subsubsection{Proof of \Cref{thm:example}} \label{sec:main_theorem_proof}
	
	We suppose to the contrary that there is a metric space $(\widehat{X},\widehat{d})$ such that a factorization $\iota = \widehat{\iota} \circ P$ as in the statement of \Cref{thm:example} exists, that is,  that $P$ is bi-Lipschitz and that $\widehat{\iota}$ has distortion $H( \widehat{\iota} ) < \sqrt{ 2 }$. Since $\iota$ is Lipschitz, it follows that $\widehat{\iota}$ is also Lipschitz.
	
	By considering the metric $\widehat{d}( P(x), P(y) )$ on $\mathbb{R}^{2}$, we assume without loss of generality that $\widehat{X} = \mathbb{R}^2$ and that $\widehat{\iota}$ and $P$ are each the identity map on $\mathbb{R}^2$. Let $\widehat{N}$ denote the metric derivative of the map $P\colon (\mathbb{R}^2,\norm{\cdot}_2) \to (\mathbb{R}^2, \widehat{d})$ as defined in \Cref{defi:metric_derivative_Ivanov}.

	By assumption, the identity map $\widehat{\iota} \colon (\mathbb{R}^2, \widehat{d}) \to (\mathbb{R}^2, d)$ is quasiconformal with $H(\widehat{\iota}) < \sqrt{2}$. Moreover, since $P$ is Lipschitz, there exists $C>0$ such that $\widehat{N}_x \leq C\norm{\cdot}_1$ for every $x \in \mathbb{R}^2$. Let $v = (1,0)$ and $w = (1/\sqrt{2}, 1/\sqrt{2})$, and let $a = H(\widehat{\iota})/\sqrt{2} < 1$. It suffices to show that, for all $i \geq 0$ and almost every $x \in K_i \setminus K_{i+1}$,
	\begin{equation}\label{equ:inductive_claim}
        \begin{array}{ll}
            \widehat{N}_x(w) \leq Ca^{i} &\text{ if $i$ is even,} 
            \\
            \widehat{N}_x(v) \leq Ca^{i} & \text{ if $i$ is odd.}
        \end{array}  
	\end{equation}
	This provides a contradiction. Indeed, given that $P$ is bi-Lipschitz, $\widehat{N}_{x}(w)$ and $\widehat{N}_{x}(v)$ are bounded from below for all $x \in \mathbb{R}^2$ by some constant $C' > 0$.

    Observe that, when $i$ is even, $N_{x}(w) = \sqrt{2} N_{x}(v)$ for all $x \in K_{i} \setminus K_{i+1}$.  Similarly, when $i$ is odd, $N_{x}(v) = \sqrt{2}N_{x}(w)$. It follows from Proposition 5.12 of \cite{Iko:19} that the pointwise distortion of $\widehat{ \iota }$ coincides with the distortion of the identity map from $( \mathbb{R}^{2}, \widehat{N}_{x} )$ to $( \mathbb{R}^{2}, N_{x} )$ for almost every $x \in \mathbb{R}^2$. As a consequence, for almost every $x \in K_{i} \setminus K_{i+1}$,
	\begin{equation}\label{equ:inductive_claim:key}
        \begin{array}{ll}
            \widehat{N}_x(v) \leq a \widehat{N}_{x}(w) &\text{ if $i$ is even,} 
            \\
            \widehat{N}_{x}(w) \leq a \widehat{N}_x(v) & \text{ if $i$ is odd}.
        \end{array}  
	\end{equation}
	
    We verify \eqref{equ:inductive_claim} by induction on $i$. The claim is immediate for $i=0$, recalling that $K_{0} = \mathbb{R}^{2}$. For the inductive step, fix $i \geq 1$ and assume that \eqref{equ:inductive_claim} holds for almost every $x \in \mathbb{R}^2 \setminus K_{i}$. We show that \eqref{equ:inductive_claim} holds for almost every $x \in K_i \setminus K_{i+1}$. Let $N_{i-1}$ denote the set of points in $K_{i-1} \setminus K_{i}$ for which \eqref{equ:inductive_claim} or \eqref{equ:inductive_claim:key} fails. We split into two cases based upon on whether $i$ is odd or even. The idea is the same in each, but the bookkeeping requires separate statements.
	
	\subsubsection*{Case 1} Assume that $i$ is odd. By the inductive hypothesis, we have $\widehat{N}_x(v) \leq a \widehat{N}_x(w) \leq Ca^{i}$ for every $x \in (K_{i-1} \setminus K_i)\setminus N_{i-1}$, where $N_{i-1}$ has $\mathcal{L}^2$-measure zero. We claim that $\widehat{N}_x(v) \leq Ca^{i}$ for almost every $x \in K_i \setminus K_{i+1}$. Assume to the contrary that there exists a set $G \subset K_i \setminus K_{i+1}$ of positive measure and a constant $b>0$ such that $\widehat{N}_x(v) \geq (C+b)a^{i}$ for all $x \in G$.
	
	For all $t \in [0,1]$, let $\gamma_{t}\colon I \to \mathbb{R}^2$ be the path defined by $\gamma_{t}(s) = (s,t)$. 
	According to \Cref{lemm:metric_speed_Lips}, for every path $\gamma_{t}$ and every subinterval $I' \subset I$,
	\[
	    \ell_{\widehat{d}}(\gamma_{t}|_{I'})
	    =
	    \int_{I'} \widehat{N}_{\gamma_t(s)}(v)\, d\mathcal{L}^{1}.
	\]
	Consider now the interval $[s_0,s_0+h]$ for some $s_0 \in (0,1)$ and $h \in (0, 1-s_0)$. Differentiating, we have for $\mathcal{L}^{2}$-almost every $(s_0,t) \in G$ that
	\[
	    \lim_{h \to 0}
	        \frac{
	            \widehat{d}( \gamma_{t}( s_{0} ), \gamma_{t}( s_{0} + h ) )
	        }{h}
	    =
	    \lim_{h \to 0} \frac{\ell_{\widehat{d}}( \gamma_t|_{ [s_0,s_0+h] } ) }{h}
	    \geq
	    (C+b)a^{i}.
	\]  
	In particular, for almost every $x \in G$, there exists $r_0 >0$ such that
	\begin{equation}\label{equ:arguebycontradiction}
	    \widehat{d}(x,x+ rv ) \geq (C+ b/2)a^{i}r
	\end{equation}
	for all $r \in (0,r_0)$.

    On the other hand, consider now a point $x \in G$ such that $\pi_1^*(x)$ is a Lebesgue density point of $\pi_1^*(K_{i-1})$ and $\pi_2^*(x)$ is a Lebesgue density point of $\pi_2^*(K_{i-1})$. Note that by Fubini's theorem, $\mathcal{L}^{2}$-almost every point in $G$ has this property. Let $\delta>0$ and let $t_0 = t_0(\delta)$ be such that the hypothesis in \Cref{lemm:quantitative_density} is satisfied for both the point $\pi_1^*(x)$ and the point $\pi_2^*(x)$.
    
    For all $\varepsilon > 0$, let $H(x, \varepsilon )$ be the set comprising those points $y \in H_{i-1} \cap \overline{B}_{ \norm{ \cdot }_{2} }( x, \varepsilon )$ for which
    \begin{equation*}
        \mathcal{H}^{1}_{ \norm{ \cdot }_{2} }(
            N_{i-1}
            \cap
             \pi_2^{-1}(\pi_{2}(y))
        )
        =
        0.
    \end{equation*}
    Recall that the set $H_{i-1}$ is defined in \Cref{sec:network_paths}. By \Cref{lemm:previous_level_density}, the set $H_{i-1} \cap \overline{B}_{ \norm{ \cdot }_{2} }( x, \varepsilon )$ has positive $\mathcal{L}^2$-measure. Since $N_{i-1}$ has $\mathcal{L}^2$-measure zero, an application of Fubini's theorem shows that $H(x,\varepsilon)$ is a full measure subset of $H_{i-1} \cap \overline{B}_{ \norm{ \cdot }_{2} }( x, \varepsilon )$. Let $r \in (0, t_0)$ and $\varepsilon \in (0,2\delta r)$. 

    Consider a point $y \in H(x,\varepsilon)$. Let $\gamma_y\colon [0,r] \to \mathbb{R}^2$ be the path defined by $\gamma_{y}(s) = y + sv$. \Cref{lemm:metric_speed_Lips} implies that 
    \begin{equation*}
        \ell_{ \widehat{d} }( \gamma_{y} )
        =
        \int_{ [0,r] }
            \widehat{N}_{ \gamma_{y}(s) }( v )
        \,d\mathcal{L}^{1}( s ),
    \end{equation*}
    and the definition of $H(x, \varepsilon )$ implies that
    \begin{equation}
        \label{equ:key:step:induction}
        \widehat{N}_{ z }( v )
        \leq
        C a^{i}
    \end{equation}
    for $\mathcal{H}^{1}_{ \norm{ \cdot }_{2} }$-almost every $z \in K_{i-1} \cap \abs{ \gamma_{y} }$.
%
%
%
    
    Next, we estimate the $\mathcal{H}^{1}_{ \norm{ \cdot }_{2} }$-measure of $K_{i-1} \cap \abs{ \gamma_{y} }$. To this end, observe that the path $\gamma_y^1\colon [0,r] \to \mathbb{R}$, $\gamma_y^1(s) = y + sw/\sqrt{2}$, intersects $K_{i-1}$ in a set congruent to $\pi_1^*(K_{i-1}) \cap \pi_1^*(|\gamma_y^1|)$. Similarly, the path $\gamma_y^2\colon [0,r] \to \mathbb{R}$, $\gamma_y^2(s) = y + s\overline{w}/\sqrt{2}$, where $\overline{w} = (1/\sqrt{2},-1/\sqrt{2})$, intersects $K_{i-1}$ in a set congruent to $\pi_2^*(K_{i-1}) \cap \pi_2^*(|\gamma_y^2|)$. Since $|\pi_m^*(y) - \pi_m^*(x)| <2\delta r$, \Cref{lemm:quantitative_density} gives, for $m \in \{1,2\}$,
    \begin{equation}\label{equ:temporary:ref}
        \frac{
            \mathcal{H}_{\norm{\cdot}_2}^1(
                K_{i-1} \cap |\gamma_y^m|
            )
        }{ r/\sqrt{2} }
        \geq 
        1-2\delta.
    \end{equation}
    We combine this with the following observation: for any measurable sets $E_1, E_2 \subset [0,r]$ satisfying $|E_j| \geq (1-\varepsilon_j)r$ for some $\varepsilon_j \in (0,1)$, $j \in \{1,2\}$, the diagonal path $\gamma\colon [0,r] \to [0,r]^2$ defined by $\gamma(s) = (s,s)$ intersects $E_1 \times E_2$ in a set of length at least $\sqrt{2}(1-\varepsilon_1 - \varepsilon_2)r$. Since $K_{i-1}$ is constructed as a product set relative to which $\gamma_y$ is a diagonal path, we conclude from \eqref{equ:temporary:ref} that
    \begin{equation} \label{equ:K_i_intersection}
        \frac{\mathcal{H}_{\norm{\cdot}_2}^1(K_{i-1} \cap |\gamma_y|)}{r}
        \geq
        1-4\delta. 
    \end{equation}
    Using \eqref{equ:key:step:induction} and the fact that $\widehat{N}_z(v) \leq C$ for all $z \in \mathbb{R}^2$, the inequality \eqref{equ:K_i_intersection} gives
    \[
        \widehat{d}(y,y+rv) \leq (1-4\delta)Ca^{i}r + 4\delta C r.
    \]  
    Next, by making the initial choice of $\delta$ sufficiently small, we have
    \[
        \widehat{d}(y,y+rv) \leq (1+\delta)Ca^{i}r.
    \]
    From this and the relationship $\widehat{d} \leq Cd_{\norm{ \cdot }_{1}} \leq \sqrt{2}Cd_{\norm{ \cdot }_{2}}$, it follows that 
    \[\widehat{d}(x,x+rv) \leq 2\sqrt{2} C\varepsilon + (1+\delta)Ca^{i}r.\]
    Since $\varepsilon \in (0,2\delta r)$ is arbitrary, we obtain 
    \[
        \widehat{d}(x,x+rv) \leq (1+\delta)Ca^{i}r.
    \] 
    Since this estimate holds for $\mathcal{L}^2$-almost every $x \in G$, this contradicts our earlier statement \eqref{equ:arguebycontradiction} when $\delta$ is sufficiently small. We conclude that $\widehat{N}_x(v) \leq Ca^{i}$ for almost every $x \in K_i\setminus K_{i+1}$.
    
    \subsubsection*{Case 2} We now consider the case that $i$ is even. The idea is the same as in the first case, but now everything is rotated by $\pi/4$. By the inductive hypothesis, we have that $\widehat{N}_x(w) \leq a\widehat{N}_x(v) \leq Ca^{i}$ for every $x \in (K_{i-1}\setminus K_i)\setminus N_{i-1}$. We claim that $\widehat{N}_x(w) \leq Ca^{i}$ for almost every $x \in K_i \setminus K_{i+1}$. Assume to the contrary that there exists a set $G \subset K_i \setminus K_{i+1}$ of positive measure and a constant $b>0$ such that $\widehat{N}_x(w) \geq (C+b)a^i$ for all $x \in G$.
    
    For all $t \in J^*$, let $\gamma_t\colon I^* \to \mathbb{R}^2$ be the path defined by $\gamma_t(s) = \varphi(s,t)$. Consider as before the interval $[s_0,s_0+h]$ for some $s_0 \in (0,\sqrt{2})$ and $h \in (0, \sqrt{2}-s_0)$. Differentiating, we have that \[\lim_{h \to 0} \frac{\ell_{\widehat{d}}(\gamma_t|_{ [s_0,s_0+h] } )}{h} \geq (C+b)a^i.\]
    In particular, for $\mathcal{L}^{2}$-almost every $x \in G$, there exists $r_{0} > 0$ such that
    \begin{equation}\label{equ:arguebycontradiction:2}
        \widehat{d}(x,x+rw) \geq (C+ b/2)a^{i}r
    \end{equation}
    for all $r \in ( 0, r_{0} )$.

    On the other hand, consider a point $x \in G$ such that $\pi_1(x)$ is a Lebesgue density point of $\pi_1(K_{i-1})$ and $\pi_2(x)$ is a Lebesgue density point of $\pi_2(K_{i-1})$. Let $\delta > 0$ and let $t_0 = t_0(\delta)$ be the corresponding value in \Cref{lemm:quantitative_density}. For all $\varepsilon >0$, define the set $H(x,\varepsilon)$ as the set of points $y \in H_{i-1} \cap \overline{B}_{\norm{\cdot}_2}(x,\varepsilon)$ for which 
    \[\mathcal{H}^{1}_{ \norm{ \cdot }_{2} }(
            N_{i-1}
            \cap
             (\pi_2^*)^{-1}(\pi_{2}^*(y))
        )
        =
        0.\]
    As before, $H(x, \varepsilon)$ is a full measure subset of $H_{i-1} \cap \overline{B}_{\norm{\cdot}_2}(x,\varepsilon)$. Let $r \in (0,t_0)$ and $\varepsilon \in (0,2\delta r)$.
    
    For all $y \in H(x,\varepsilon)$, define the path $\gamma_y\colon [0,r] \to \mathbb{R}^2$ by $\gamma_{y}(s) = y + sw$. Recall from \Cref{lemm:metric_speed_Lips} that
    \[\ell_{\widehat{d}}(\gamma_{y}) = \int_{[0,r]}  \widehat{N}_{\gamma_y(s)}(w)\, d\mathcal{L}^{1}.\] 
    Moreover, $\widehat{N}_{z}( w ) \leq C a^{i}$ for $\mathcal{H}^{1}_{ \norm{ \cdot }_{2} }$-almost every $z \in K_{i-1} \cap \abs{ \gamma_{y} }$. Arguing as in the first case, we obtain the inequality
    \[
        \widehat{d}(y,y+rw) \leq (1-4\delta)Ca^{i}r + 2\delta C r. 
    \]
    Next, by taking $\delta$ sufficiently small, we then have $\widehat{d}(y,y+rw) \leq (1+\delta)Ca^ir$.
     As before, since $\varepsilon \in (0,2\delta r)$ is arbitrary,
    \[
        \widehat{d}(x,x+rw)
        \leq
        (1+\delta)Ca^{i}r,
    \]
    which contradicts \eqref{equ:arguebycontradiction:2} for sufficiently small $\delta > 0$. We conclude that $\widehat{N}_x(w) \leq Ca^{i}$ for almost every $x \in K_i \setminus K_{i+1}$.

\bibliographystyle{alpha}
\bibliography{bibliography}

\end{document}